\documentclass[10pt]{amsart}
\usepackage{amssymb, enumitem}
\usepackage[all]{xy}
\usepackage{hyperref, aliascnt}
\usepackage[english]{babel}
\usepackage{tikz}

\def\today{\number\day\space\ifcase\month\or   January\or February\or
   March\or April\or May\or June\or   July\or August\or September\or
   October\or November\or December\fi\   \number\year}

\theoremstyle{definition}
\newtheorem{lma}{Lemma}[section]

\newaliascnt{thmCt}{lma}
\newtheorem{thm}[thmCt]{Theorem}
\aliascntresetthe{thmCt}

\newaliascnt{corCt}{lma}
\newtheorem{cor}[corCt]{Corollary}
\aliascntresetthe{corCt}

\newaliascnt{propCt}{lma}
\newtheorem{prop}[propCt]{Proposition}
\aliascntresetthe{propCt}

\newtheorem*{thm*}{Theorem}
\newtheorem*{cor*}{Corollary}
\newtheorem*{prop*}{Proposition}

\newcounter{theoremintro}
\newtheorem{thmintro}[theoremintro]{Theorem}
\newtheorem{corintro}[theoremintro]{Corollary}
\newtheorem{egintro}[theoremintro]{Example}

\newaliascnt{pgrCt}{lma}

\aliascntresetthe{pgrCt}

\newaliascnt{dfCt}{lma}
\newtheorem{df}[dfCt]{Definition}
\aliascntresetthe{dfCt}

\newaliascnt{remCt}{lma}
\newtheorem{rem}[remCt]{Remark}
\aliascntresetthe{remCt}

\newaliascnt{remsCt}{lma}

\aliascntresetthe{remsCt}

\newaliascnt{egCt}{lma}
\newtheorem{eg}[egCt]{Example}
\aliascntresetthe{egCt}

\newaliascnt{egsCt}{lma}

\aliascntresetthe{egsCt}

\newaliascnt{qstCt}{lma}

\aliascntresetthe{qstCt}

\newaliascnt{pbmCt}{lma}

\aliascntresetthe{pbmCt}

\newaliascnt{notaCt}{lma}

\aliascntresetthe{notaCt}

\newcommand{\beq}{\begin{equation}}
\newcommand{\eeq}{\end{equation}}
\newcommand{\beqa}{\begin{eqnarray*}}
\newcommand{\eeqa}{\end{eqnarray*}}
\newcommand{\bal}{\begin{align*}}
\newcommand{\eal}{\end{align*}}
\newcommand{\bi}{\begin{itemize}}
\newcommand{\ei}{\end{itemize}}
\newcommand{\be}{\begin{enumerate}}
\newcommand{\ee}{\end{enumerate}}

\newcommand{\ep}{\varepsilon}
\newcommand{\zt}{\zeta}

\newcommand{\Q}{{\mathbb{Q}}}

\newcommand{\Z}{{\mathbb{Z}}}
\newcommand{\R}{{\mathbb{R}}}
\newcommand{\C}{{\mathbb{C}}}
\newcommand{\N}{{\mathbb{N}}}

\newcommand{\Hi}{{\mathcal{H}}}
\newcommand{\K}{{\mathcal{K}}}

\newcommand{\B}{{\mathcal{B}}}
\newcommand{\U}{{\mathcal{U}}}

\newcommand{\T}{{\mathbb{T}}}
\newcommand{\D}{{\mathcal{D}}}

\newcommand{\Ot}{{\mathcal{O}_2}}
\newcommand{\OI}{{\mathcal{O}_{\I}}}
\newcommand{\OIst}{{\mathcal{O}_{\I}^{\mathrm{st}}}}

\newcommand{\Cu}{{\mathrm{Cu}}}

\newcommand{\sr}{{\mathrm{sr}}}
\newcommand{\dr}{{\mathrm{dr}}}
\newcommand{\RR}{{\mathrm{RR}}}

\newcommand{\id}{{\mathrm{id}}}

\newcommand{\spec}{{\mathrm{sp}}}

\newcommand{\diag}{{\mathrm{diag}}}

\newcommand{\Aut}{{\mathrm{Aut}}}

\newcommand{\Hom}{{\mathrm{Hom}}}

\newcommand{\Ad}{{\mathrm{Ad}}}

\newcommand{\Ann}{{\mathrm{Ann}}}

\newcommand{\dimRok}{{\mathrm{dim}_\mathrm{Rok}}}
\newcommand{\cdimRok}{{\mathrm{dim}_\mathrm{Rok}^\mathrm{c}}}
\newcommand{\dimnuc}{{\mathrm{dim}_\mathrm{nuc}}}

\newcommand{\ifo}{if and only if }

\newcommand{\ca}{$C^*$-algebra}

\newcommand{\uca}{unital $C^*$-algebra}

\newcommand{\Rp}{Rokhlin property}

\newcommand{\Rdim}{Rokhlin dimension}

\newcommand{\I}{\infty}

\title[]{Rokhlin dimension: duality, tracial properties, and crossed products}

\thanks{The first named author was partially supported by the D.~K.~Harrison Prize from the University
of Oregon, by the Deutsche Forschungsgemeinschaft (SFB 878), and by a Postdoctoral Research Fellowship
from the Humboldt Foundation. This research was supported by GIF grant 1137/2011, and by Israel Science Foundation Grant 476/16.}



\author[Eusebio Gardella]{Eusebio Gardella}
\address{Eusebio Gardella
Mathematisches Institut, Fachbereich Mathematik und Informatik der
Universit\"at M\"unster, Einsteinstrasse 62, 48149 M\"unster, Germany.}
\email{gardella@uni-muenster.de}
\urladdr{www.math.uni-muenster.de/u/gardella/}
\author{Ilan Hirshberg}
\address{Ilan Hirshberg
Department of Mathematics, Ben-Gurion University of the Negev, Be’er
Sheva 8410501, Israel.}
\email{ilan@math.bgu.ac.il}
\urladdr{www.math.bgu.ac.il/~ilan/}
\author{Luis Santiago}
\address{Luis Santiago
Institute of Mathematics, University of Aberdeen, Fraser Noble Building, Aberdeen
AB24 3UE, UK.}
\email{lmoreno@abdn.ac.uk}
\urladdr{www.homepages.abdn.ac.uk/lmoreno/pages/index.html}
\begin{document}

\begin{abstract}
We study compact group actions with finite Rokhlin dimension, particularly in relation to
crossed products. For example, we characterize
the duals of such actions, generalizing previous partial results for the Rokhlin property. As
an application, we determine the ideal structure of their crossed products.
Under the assumption of so-called commuting
towers, we show that taking crossed products by such actions preserves a
number of relevant classes of $C^*$-algebras, including: $D$-absorbing
$C^*$-algebras, where $D$ is a strongly self-absorbing $C^*$-algebra, stable
$C^*$-algebras, $C^*$-algebras with finite nuclear
dimension (or decomposition rank), $C^*$-algebras with finite stable rank
(or real rank), and $C^*$-algebras whose $K$-theory is either trivial, rational, or $n$-divisible for $n\in\N$.
The combination of nuclearity and the UCT is also shown to be preserved by these actions. Some of
these results are new even in the well-studied case of the Rokhlin property.
Additionally, and under some technical
assumptions, we show that finite Rokhlin dimension with commuting towers
implies the (weak) tracial Rokhlin property.

At the core of our arguments is a certain local
approximation of the crossed product by a continuous $C(X)$-algebra with
fibers that are stably isomorphic to the underlying
algebra. The space $X$ is computed in some cases of interest, and we use its
description to construct a $\Z_2$-action on a unital AF-algebra and on a unital Kirchberg
algebra satisfying the UCT, whose Rokhlin dimensions with and without commuting
towers are finite but do not agree.
\end{abstract}

\maketitle

\tableofcontents

\renewcommand*{\thetheoremintro}{\Alph{theoremintro}}
\section*{Introduction}

The goal of this paper is to study a variety of structural properties for actions of finite
(and more generally compact) groups on $C^*$-algebras with finite Rokhlin dimension. The
concept of Rokhlin dimension for actions of finite groups and actions of the integers was
introduced by the second author, Winter and Zacharias in \cite{HirWinZac_rokhlin_2015} as
a generalization of the well-studied Rokhlin property for group actions. (See
\cite{HerJon_models_1983}, \cite{Izu_finiteI_2004}, and \cite{OsaPhi_crossed_2012} for finite
group actions with the Rokhlin property, and \cite{HirWin_rokhlin_2007}, \cite{Gar_crossed_2017},
and \cite{Gar_compact_2018} for compact group actions.) The Rokhlin property can be viewed as
a regularity condition for the group action, which can be used to show that various structural
properties pass from a $C^*$-algebra to its crossed product; see for example \cite{HirWin_rokhlin_2007},
\cite{OsaPhi_crossed_2012}, \cite{San_crossed_2015}, \cite{Gar_crossed_2017}. However, particularly
in the case of finite group actions, the Rokhlin property is a very restrictive hypothesis to place
on the action, as it implies that the unit of the algebra can be written nontrivially
as a sum of projections indexed by the group. In the context of Rokhlin dimension, the tower of
projections indexed by the group is replaced by several towers consisting of positive contractions,
each of which is indexed by the group. In this formulation, Rokhlin dimension zero
agrees with the Rokhlin property, while the higher values allow for greater flexibility.
In particular, actions with finite Rokhlin dimension may exist, and even be generic,
on algebras that do not admit any action with the (tracial) Rokhlin property.

A different and earlier weakening of the Rokhlin property is the tracial Rokhlin property
(\cite{Phi_tracial_2011}), in which the Rokhlin projections are not required to add up to 1, but rather
just up to a small error in trace. This property can be shown to hold for many actions of interest which do not
have the Rokhlin property. Nevertheless, as it demands the existence
of projections, there are many \ca s that do not admit any such action.
There exist a number of weakenings of the tracial Rokhlin property,
in which projections are replaced by positive elements; see \cite{Arc_crossed_2008},
\cite{HirOro_tracially_2013} and \cite{MatSat_stability_2012}. In general, there is no particular reason
for the tracial Rokhlin property (in either its weak version, with positive elements, or strong version, with
projections) to imply finite Rokhlin dimension or vice versa, and one of our goals is
to find sufficient conditions for such implications to hold.

The main interest in establishing that an action has finite Rokhlin dimension, from the standpoint of structure 
and classification of nuclear $C^*$-algebras, stems from the fact that this can be used to show that various 
structural properties of interest pass to crossed products; see for example \cite{HirWinZac_rokhlin_2015},
\cite{HirPhi_rokhlin_2015} and \cite{Gar_regularity_2017}. We briefly review the main structural properties of 
interest, particularly in the context of the Elliott classification program, which has been essentially resolved 
recently (\cite{EllGonLinNiu_classification_2015,TikWhiWin_quasidiagonality_2017}), and refer the reader 
to~\cite{EllTom_regularity_2008} for a survey the Elliott program and regularity properties in this context. 
A separable unital infinite dimensional $C^*$-algebra $D$ is said to be \emph{strongly self-absorbing} 
(\cite{TomWin_strongly_2007}) if $D \cong D \otimes D$, and furthermore the first coordinate embedding 
$x \mapsto x \otimes 1$ is approximately unitarily equivalent to an isomorphism. (Any strongly self-absorbing 
$C^*$-algebra has to be nuclear regardless of which tensor product is a-priori used in the definition, hence 
there is no ambiguity.)  If $D$ is a strongly self-absorbing $C^*$-algebra, then a $C^*$-algebra $A$ is said 
to be $D$-\emph{stable} or $D$-\emph{absorbing} if $A \cong A \otimes D$. The \emph{Jiang-Su algebra}, denoted 
$\mathcal{Z}$, is a particularly important strongly-self absorbing $C^*$-algebra. The explicit construction 
was first introduced in \cite{JiaSu_simple_1999}, however it can be better characterized as an initial object 
in the category of strongly self-absorbing $C^*$-algebras, in the sense that $D \cong D \otimes \mathcal{Z}$ 
for any strongly self-absorbing $C^*$-algebra $D$ (\cite{Win_strongly_2011}). The Jiang-Su algebra is 
$KK$-equivalent to the complex numbers, and tensoring by it does not change the Elliott invariant of a simple 
separable nuclear $C^*$-algebra, assuming its $K_0$-group is unperforated, and indeed, $\mathcal{Z}$-stability 
turned out to be an essential ingredient in classifiability. Capping extensive work on the Toms-Winter 
conjecture, for simple separable unital nuclear $C^*$-algebras, $\mathcal{Z}$-stability was recently shown 
to be equivalent to finite \emph{nuclear dimension} (\cite{CETWW_2019}), another key hypothesis for classification. 
Nuclear dimension (\cite{KirWin_covering_2004,WinZac_dim_nuc_2010}) is a generalization of covering dimension 
for nuclear $C^*$-algebras. Recall that a $C^*$-algebra $A$ is nuclear if for any finite set $F \subseteq A$ 
and any $\varepsilon>0$ there exists a finite dimensional algebra $B$ and a diagram
$$
\xymatrix{
A\ar[rr]^{\id}\ar[dr]_{\varphi} & & A \\
&  B \ar[ur]_{\psi} &
}
$$
such that $\varphi$ and $\psi$ are completely positive maps and $\|\psi(\varphi(a)) - a\|<\varepsilon$ for all $a \in F$. 
A completely positive map $\theta$ between $C^*$-algebras is said to be an \emph{order zero} map if preserves 
orthogonality, that is, whenever $x,y$ are positive elements such that $xy =0$ then $\theta(x)\theta(y) = 0$. 
Returning to the diagram above, the map $\psi\colon B \to A$ is said to be $d$-decomposable if one can 
decompose $B$ as a direct sum $B \cong \bigoplus_{k=0}^d B_k$ of $C^*$-subalgebras such that the restriction 
of $\psi$ to each summand is order zero. Such approximations can always be found (\cite{HirKirWhi_2012}), 
if one does not bound the number of summands, and $A$ is said to have nuclear dimension at most $d$ if one can always 
find decompositions as above involving at most $d+1$ summands. Other 
structural properties we consider include finite real rank and stable rank, other notions of dimension which 
were introduced much earlier; we refer the reader to Blackadar's book for a reference (\cite{Bla_ktheory_1998}). 
We note that for simple separable nuclear unital and stably finite $C^*$-algebras, R{\o}rdam 
(\cite{Ror_Z-absorbing_2004}) showed that $\mathcal{Z}$-stability implies stable rank 1, obtained conditions for 
real rank 0, and also showed that the Cuntz semigroup is weakly unperforated. The Cuntz semigroup is a semigroup 
constructed out of positive elements in a way which is analogous to the construction of the Murray-von-Neumann 
semigroup of projections. Weak unperforation of the Cuntz semigroup, for simple separable nuclear unital 
$C^*$-algebras, is conjectured to be equivalent to $\mathcal{Z}$-stability, and this was verified under certain 
conditions on the trace space (\cite{KirRor_2014,Sat_2012,TomWhiWin_2015}), and for
certain $C^*$-algebras with stable rank one (\cite{Thi_ranks_2017}).
The Universal Coefficient Theorem (UCT) of Rosenberg and Schochet (\cite{RosSch_UCT_1987}) states that for 
$C^*$-algebras $A$ in a suitable ``bootstrap'' class of $C^*$-algebras, for any separable $C^*$-algebra $B$, 
the canonical map 
$KK(A,B) \to \Hom(K_*(A),K_*(B))$ is surjective, and the kernel is furthermore identified with an appropriate 
Ext group. The UCT plays a central role in Elliott's program, and it is a major open problem whether any separable 
nuclear $C^*$-algebra satisfies the UCT. This is known to be equivalent to the question of whether the UCT is 
preserved under crossed products by the circle or finite cyclic groups (\cite[23.15.12]{Bla_ktheory_1998}).

Recognizing when a given action has finite Rokhlin dimension is
not always straightforward (and computing its actual dimension is even more challenging). When the acting group is abelian, one of our main
results allows one to detect the Rokhlin dimension of an action $\alpha\colon G\to\Aut(A)$ by looking at its dual
$\widehat{\alpha}\colon\widehat{G}\to\Aut(A\rtimes_\alpha G)$. To this end, we define (see \autoref{df:repdim})
the \emph{representability dimensions}, $\dim_{\mathrm{rep}}(\beta)$ and $\dim^{\mathrm{c}}_{\mathrm{rep}}(\beta)$, with and without commuting colors,
of a discrete group action $\beta\colon \Gamma\to \Aut(B)$, and prove the following:

\begin{thmintro}(See \autoref{thm:duality})
Let $G$ be a compact abelian group, let $A$ be a \ca, and let $\alpha\colon G\to\Aut(A)$ be an action. Then
\[\dimRok(\alpha)=\dim_\mathrm{rep}(\widehat{\alpha}) \ \ \mbox{ and } \ \ \cdimRok(\alpha)=\dim_\mathrm{rep}^{\mathrm{c}}(\widehat{\alpha}).\]
\end{thmintro}

Theorem~A is used to show that actions with finite Rokhlin dimension have full strong Connes spectrum, and hence that ideals in the crossed
product are induced by invariant ideals in the algebra; see \autoref{prop: FullConnesSpec} and \autoref{cor: MorEquiv}.

Actions with finite Rokhlin dimension are very closely connected to free actions on spaces. For example, it was shown in Lemma~2.3 of~\cite{HirPhi_rokhlin_2015}
and Theorem~4.5 of~\cite{Gar_rokhlin_2017} that for a finite dimensional metrizable space $X$, an action of a compact group $G$ on $C_0(X)$ has finite
Rokhlin dimension if and only if the induced action of $G$ on $X$ is free. The connections go far beyond the commutative case, at least in the formulation of
Rokhlin dimension with commuting towers. Indeed, for every compact group $G$ and every $d\geq 0$, there exists a universal compact free $G$-space $X_{G,d}$
such that an action $\alpha\colon G\to\Aut(A)$ of $G$ on a \ca\ $A$ satisfies $\cdimRok(\alpha)\leq d$ if and only if there exists an asymptotically
central equivariant homomorphism from $C(X_{G,d})$ into $A$; see \autoref{thm:XRpRdim}
(for finite groups, this is implicit in Lemma~1.9 of~\cite{HirPhi_rokhlin_2015}).
This observation was used in Theorem~4.6 of~\cite{HirPhi_rokhlin_2015} to show that there do not exist actions of non-trivial compact Lie groups on the Jiang-Su algebra $\mathcal{Z}$
or the Cuntz algebra $\OI$ with finite Rokhlin dimension with commuting towers.

In this work, we take these ideas further. By identifying the spaces $X_{G,d}$ as simplicial complexes, we prove the following:

\begin{thmintro}(See \autoref{thm: RdimwTRp})
Let $A$ be an infinite dimensional, simple, finite, unital \ca\ with strict comparison
and at most countably many extreme quasitraces. Let $G$ be a finite group and let $\alpha\colon G\to
\Aut(A)$ be an action. If $\cdimRok(\alpha)<\I$, then $\alpha$ has the weak tracial Rokhlin property.
\end{thmintro}

We say a few words about the proof of Theorem~B. Let $X_{G,d}$ be the universal free $G$-space associated to $G$ and $d=\cdimRok(\alpha)$. For the sake
of argument, suppose that there is a unital, central and equivariant inclusion $C(X_{G,d})\to A$. The restriction of each (extremal) quasitrace on $A$ to
$C(X_{G,d})$ induces a Borel probability measure on $X_{G,d}$. Let $\{\mu_n\}_{n\in\N}$ be the collection of such measures. In order to prove that
the given action has the weak tracial Rokhlin property (see \autoref{df: tRp wtRp}), it suffices to find an open set $U\subseteq X_{G,d}$ such that
$gU\cap hU=\emptyset$ for $g,h\in G$ with $g\neq h$, and $\mu_n\left(X_{G,d}\setminus \bigcup\limits_{g\in G}gU\right)$ is small for all $n\in\N$. (Given
such a set, one chooses a positive contraction supported on $U$, and considers its $G$-translates.) The existence of such an open set is proved in
\autoref{thm:FreeActionCellComplex}.

The argument described above breaks down in the absence of traces. However, different methods yield an even
stronger result in this case. Indeed, when
$A$ is simple, exact, and has strict comparison, it has no nonzero traces if and only if it is purely infinite.
In the nuclear case, it must therefore
be a Kirchberg algebra. In this context, we use a lemma of Kishimoto from \cite{Kis_outer_1981} to prove that a
number of generally inequivalent notions,
actually coincide for actions on Kirchberg algebras.

\begin{thmintro} (See \autoref{thm:outertRp})
Let $G$ be a finite group, let $A$ be a Kirchberg algebra, and let $\alpha\colon G\to\Aut(A)$ be an action. Then
the following are equivalent:
\be
\item $\alpha$ has the (weak or strong) tracial Rokhlin property;
\item $\dimRok(\alpha)\leq 1$ (or just $\dimRok(\alpha)<\I$);
\item $\alpha_g$ is not inner for all $g\in G\setminus\{1\}$.
\ee\end{thmintro}

Suppose that $\alpha\colon G\to\Aut(A)$ is an action with $\cdimRok(\alpha)\leq d<\I$. With $X_{G,d}$ as above, it can be shown that the crossed
product $A\rtimes_\alpha G$ can be locally approximated by a continuous $C^*$-bundle with base space $X_{G,d}/G$ and fibers isomorphic to $A\otimes\K(L^2(G))$;
see \autoref{thm:CXGalgebra} and \autoref{prop:ApproxCP}. 
This fact allows us to transfer a number of regularity properties, many of which are relevant
from the point of view of classification, from $A$ to $A\rtimes_\alpha G$ and $A^\alpha$. Indeed, we show:

\begin{thmintro}(See \autoref{thm:preservationCP} and \autoref{thm:StrCompRdim1})
Let $G$ be a second countable compact group with $\dim(G)<\I$, let $A$ be a \ca, and let $\alpha\colon G\to\Aut(A)$ be an action
with finite Rokhlin dimension (with commuting towers). The following properties pass from $A$ to $A\rtimes_\alpha G$ and $A^\alpha$:
\be
\item Absorbing a given strongly self-absorbing \ca\ (\cite{TomWin_strongly_2007}).
\item Having finite nuclear dimension (or decomposition rank).
\item Having finite stable rank (or real rank).
\item Satisfying the UCT, being nuclear and having (uniquely) divisible $K$-theory, or trivial $K$-theory.
\item Being separable, nuclear, and satisfying the UCT.
\item If $A$ has ``no $K_1$-obstructions" (see \autoref{df:noK1obstr}),
and $\cdimRok(\alpha)\leq 1$, having almost unperforated Cuntz semigroup.
\ee\end{thmintro}

See also \cite{GarLup_equivariant_2016} for other permanence results.

Recall (see \cite{EllGonLinNiu_classification_2015} and \cite{TikWhiWin_quasidiagonality_2017}) that separable, simple
\uca s with finite nuclear dimension and satisfying the UCT, are classified by their Elliott invariant (which essentially
consists of $K$-theory and traces; see \cite{EllTom_regularity_2008} for a survey), and the same
is true in this case for $A\otimes\K$. 
We obtain the following consequence of Theorem~D:

\begin{corintro}
Let $G$ be a compact group of finite covering dimension, let $A$ be a \uca, and let $\alpha\colon G\to\Aut(A)$ be an action
with $\cdimRok(\alpha)<\infty$. If $A$ is in the classifiable class mentioned above, then
so are $A\rtimes_\alpha G$ and $A^\alpha$.
\end{corintro}

The free $G$-spaces $X_{G,d}$ can be explicitly computed, although their description is not always simple to state.
For $G=\Z_2$, the space $X_{\Z_2,d}$ is equivariantly isomorphic to $S^d$ with the antipodal
action; see \autoref{lma:S1rotation}. We use this description to produce the following example:

\begin{egintro}(See \autoref{eg:cdimRokdimRokDifferent})
There exist a unital Kirchberg algebra satisfying the UCT and an action $\alpha\colon \Z_2\to \Aut(A)$ with $\cdimRok(\alpha)=2$ and
$\dimRok(\alpha)=1$.
\end{egintro}

There is a similar example on a simple AF-algebra; see \autoref{eg:cdimRok2}. These are the first examples of actions whose Rokhlin dimensions
with and without commuting towers are both finite but do not agree. Similar examples for actions of $\Z$ are presently not known.
Finally, the explicit description of the spaces $X_{\Z_2,d}$ is used to show that in some cases of interest, the Rokhlin dimension with commuting
towers of a given action is either zero or infinite; see \autoref{prop:OtwoRp} and \autoref{thm:UHFRp}.
\newline

\textbf{Acknowledgements:} The authors would like to thank N.~Christopher Phillips for sharing with them
some unpublished material that led to \autoref{thm:FreeActionCellComplex}. The first named author
thanks Etienne Blanchard, Marius Dadarlat and Martino Lupini for helpful correspondence, and Hannes Thiel for calling our attention
to the references \cite{NagOsaPhi_ranks_2001} and \cite{Sud_stable_2005}. We particularly thank Shirly Geffen for thoroughly
reading this paper, and for pointing out numerous typos and mistakes. Parts of this work were
completed while the authors were attending the following focus semesters: the \emph{Thematic Program on Abstract Harmonic Analysis, Banach and Operator
Algebras} at the Fields Institute, Toronto, Canada, in January-June 2014; the \emph{Focus Programme on $C^*$-algebras} at
the Universit\"at M\"unster, M\"unster, Germany, in April-July 2015; and \emph{Classification of operator algebras: complexity, rigidity, and dynamics} at
the Institut Mittag-Leffler, Sweden, in January-April 2016. The financial support provided by these programs is gratefully acknowledged.
Further work was conducted during visits of the first named author to Aberdeen in May 2015, and to the second named author in October 2016,
and of the second named author to M\"unster in June 2015. The authors thank the hosting math departments for the hospitality.

\section{Finite Rokhlin dimension and duality}

We begin by introducing some notation and terminology.

\begin{df}\label{df:SeqAlgs}
Let $A$ be a \ca. Let $\ell^\I(\N,A)$ denote the set of all bounded sequences in $A$ with the supremum norm
and pointwise operations. Set
\[c_0(\N,A)=\{(a_n)_{n\in\N}\in\ell^\I(\N,A)\colon \lim_{n\to\I}\|a_n\|=0\}.\]
Then $c_0(\N,A)$ is an ideal in $\ell^\I(\N,A)$, and we denote the quotient
$\ell^\I(\N,A)/c_0(\N,A)$ by $A_\I$.
We write $\eta_A\colon \ell^\I(\N,A)\to A_\I$ for the quotient map. We identify $A$ with the subalgebra of
$\ell^\I(\N,A)$ consisting of the constant sequences, and with a subalgebra of $A_\I$ by
taking its image under $\eta_A$. If $D$ is any subalgebra of $A$, then $A_\I\cap D'$ denotes the relative
commutant of $D$ inside of $A_\I$.

For a subalgebra $D\subseteq A$, write $\Ann(D,A_\I)$ for the annihilator of $D$ in $A_\I$, which is an ideal in $A_\I\cap D'$.
Following Kirchberg (\cite{Kir_central_2006}), we set
\[F(D,A)=(A_\I\cap D')/\Ann(D,A_\I),\]
and write $\kappa_{D,A}\colon A_\I\cap D'\to F(D,A)$ for the quotient map. When $D=A$, we abbreviate $F(A,A)$ to $F(A)$, and $\kappa_{A,A}$
to $\kappa_A$. Observe that $F(D,A)$ is unital whenever $D$ is $\sigma$-unital.

If $\alpha\colon G\to\Aut(A)$ is an action of $G$ on $A$, and $D$ is an $\alpha$-invariant subalgebra of $A$, then there
are (not necessarily continuous) actions
of $G$ on $\ell^\I(\N,A)$, on $A_\I$, on $A_\I\cap D'$, and on $F(D,A)$, respectively denoted, with a slight
abuse of notation, by $\alpha^\I$, $\alpha_\I$, $\alpha_\I$ and $F(\alpha)$. Following Kishimoto (\cite{Kis_oneparameter_1996}),
we set
\[\ell^\I_\alpha(\N,A)=\{a\in \ell^\I(\N,A)\colon g\mapsto \alpha^\I_g(a) \mbox{ is continuous}\}.\]
We also set $A_{\I,\alpha}=\eta_A(\ell^\I_\alpha(\N,A))$ and $F_{\alpha}(A)=\kappa_{D,A}(A_{\I,\alpha}\cap D')$.
By construction, $A_{\I,\alpha}$ and $F_{\alpha}(D,A)$ are invariant under $\alpha_\I$ and $F(\alpha)$, respectively, so the restrictions
of $\alpha_\I$ and $F(\alpha)$ to $A_{\I,\alpha}$ and $F_{\alpha}(D,A)$, which we also denote by $\alpha_\I$ and $F(\alpha)$, are continuous.
Again, $F_\alpha(D,A)$ is unital whenever $D$ is $\sigma$-unital.
\end{df}

\begin{rem}
In the context of the definition above, the algebra $A_{\I,\alpha}$ also agrees with the subalgebra of $A_\I$ where
the induced action $\alpha_\I$ acts continuously, by the main theorem of \cite{Bro_continuity_2000}.
\end{rem}

Given a compact group $G$, we denote by $\verb'Lt'\colon G\to\Aut(C(G))$ the
action of left translation. If $H$ is a subgroup of $G$, we also denote by $\verb'Lt'\colon H\to \Aut(C(G))$
the restriction of $\verb'Lt'$ to $H$.

We reproduce the definition of Rokhlin dimension for compact group actions as it appears in Definition~3.2 of~\cite{Gar_rokhlin_2017}.
Recall that a completely positive contractive map $\varphi\colon C\to A$ between \ca s is said to have
\emph{order zero} if $\varphi(c_1)\varphi(c_2)=0$ whenever $c_1,c_2\in C_+$ satisfy $c_1c_2=0$. Given a completely positive
contractive order zero map $\varphi\colon C\to A$, Theorem~2.3 in~\cite{WinZac_completely_2009} asserts that there exist
a positive contraction $h\in A^{\ast\ast}$ and a homomorphism $\pi\colon C_0((0,1]]\otimes C\to A$ such that
$\varphi(c)=h\pi(\id_{(0,1]}\otimes c)$ for all $c\in C$. In this case, we write $\varphi=h\pi$ for short.

\begin{df}\label{df:Rdim}
Let $G$ be a compact group, let $A$ be a \ca,
and let $\alpha\colon G\to\Aut(A)$ be a continuous action. We say that $\alpha$ has \emph{\Rdim\ $d$}, if
$d$ is the least integer such that for every $\sigma$-unital $\alpha$-invariant subalgebra $D\subseteq A$,
there exist equivariant completely positive contractive order zero maps
$$\varphi_0,\ldots,\varphi_d\colon (C(G),\texttt{Lt})\to (F_\alpha(D,A),F(\alpha))$$
such that $\varphi_0(1)+\ldots+\varphi_d(1)=1$.
We denote the \Rdim\ of $\alpha$ by $\dimRok(\alpha)$. If no integer $d$ as above exists, we say that
$\alpha$ has \emph{infinite Rokhlin dimension}, and write $\dimRok(\alpha)=\I$.

If one can always choose the maps $\varphi_0,\ldots,\varphi_d$ to have commuting ranges, then we say that
$\alpha$ has \emph{Rokhlin dimension $d$ with commuting towers}, and write $\cdimRok(\alpha)=d$.\end{df}

\begin{rem} For any compact group action $\alpha$, we always have $\dimRok(\alpha)\leq \cdimRok(\alpha)$.
The inequality can be strict (see Example~4.8 in~\cite{Gar_rokhlin_2017}), even when both dimensions are finite
(see \autoref{eg:cdimRok2} and \autoref{eg:cdimRokdimRokDifferent}).
\end{rem}

It is straightforward to check that when $G$ is finite and $A$ is separable, \autoref{df:Rdim} agrees with Definition~1.14 in
\cite{HirPhi_rokhlin_2015}.


We need a series of easy lemmas about (equivariant) order zero maps. The following is
Proposition~2.3 in~\cite{Gar_regularity_2017}. The part about the dual coaction can be
proved by considering the induced homomorphism from the cone as in the proof of
Proposition~2.3 in~\cite{Gar_regularity_2017}, and using Theorem~3.5 in~\cite{RaeSinWil_equivariant_1989}.

\begin{lma}\label{cpcequivCP}
Let $G$ be a locally compact group, let $A$ and $B$ be \ca s, and let $\alpha\colon G\to\Aut(A)$ and $\beta\colon G\to\Aut(B)$
be actions. Given an equivariant completely positive contractive order zero map $\rho\colon A\to B$, the expression
\[\sigma(\xi)(g)=\rho(\xi(g)),\]
for $\xi\in L^1(G,A,\alpha)$ and $g\in G$, determines a completely positive contractive order zero map $\sigma\colon A\rtimes_\alpha G\to B\rtimes_\beta G$.
Moreover, when $G$ is amenable, this map is equivariant with respect to the
dual coactions of $G$ on $A\rtimes_\alpha G$ and $B\rtimes_\beta G$.
\end{lma}

In the next lemma, $f(\rho)$ denotes the order zero map obtained from $\rho$ using continuous
functional calculus as in Section~2 of~\cite{WinZac_completely_2009}. Explicitly, if $\rho=h\pi$ is a decomposition
as in Theorem~2.3 in \cite{WinZac_completely_2009} (see the comments above \autoref{df:Rdim}), then $f(\rho)=f(h)\pi$.

\begin{lma}\label{lma:FunctCalcOz}
Let $A$ and $B$ be \ca s, let $G$ be a locally compact group, and let
$\alpha\colon G\to\Aut(A)$ and $\beta\colon G\to \Aut(B)$ be actions.
Let $\rho\colon A\to B$ be an equivariant completely positive contractive order
zero map, and let $f\in C_0((0,1])$ be a positive function.
Let $f(\rho)\colon A\to B$ be the completely
positive order zero map obtained from Corollary~3.2 in~\cite{WinZac_completely_2009}.
Then $f(\rho)$ is equivariant.
\end{lma}
\begin{proof}
Denote by $\pi\colon C_0((0,1])\otimes A\to B$ the homomorphism induced by $\rho$
as in the conclusion of Theorem~2.3 in~\cite{WinZac_completely_2009}. Give $C_0((0,1])$ the
trivial $G$-action. Then $\pi$ is equivariant by Corollary~2.10 in~\cite{Gar_rokhlin_2017}.

The homomorphism $\pi_f\colon C_0((0,1])\otimes A\to B$ determined by $f(\rho)$
is determined by $\pi_f(\id_{(0,1]}\otimes a)=\pi(f\otimes a)$ for $a\in A$. It is
clear that $\pi_f$ is also equivariant, and again by Corollary~2.10 in~\cite{Gar_rokhlin_2017},
it follows that $f(\rho)$ is equivariant.
\end{proof}

We will need the notion of an order zero representation of a group, which we define below. It generalizes the notion
of a unitary representation, in the same way that order zero maps generalize $\ast$-homomorphisms.

\begin{df}\label{df:ozRepGp}
Let $G$ be a locally compact group, and let $B$ be a \ca. We say that a strongly continuous
function $u\colon G\to B$ is an \emph{order zero representation} of $G$ on $B$ if the following conditions
are satisfied:
\be\item $u_g$ is a normal contraction for all $g\in G$, and $u_1$ is positive;
\item $u_gu_h=u_1u_{gh}$ for all $g,h\in G$; and
\item $u_g^*=u_{g^{-1}}$ for all $g\in G$.\ee
\end{df}

The next result shows how to ``dilate'' an order zero representation to a unitary representation, similarly
to how one dilates an order zero map to a homomomorphism.

\begin{prop}\label{prop:ozRepozMor}
Let $G$ be a locally compact group, let $B$ be a \ca, and let $u\colon G\to B$ be an order zero representation.
Then there exist a projection $p\in B^{\ast\ast}$ and a unitary representation
$v\colon G\to \U(pB^{\ast\ast}p)$ commuting with $u_1$, such that $u_g=u_1v_g$ for all $g\in G$.
\end{prop}
\begin{proof}
We denote $h=u_1$, which is a positive contraction that commutes with $u_g$ for all $g\in G$.
For $t\in (0,1]$, use Borel functional calculus
to define a projection $p_t\in B^{\ast\ast}$ by $p_t=\chi_{(t,1]}(h)$. Then $p_t$ commutes with $u_g$ for all $g\in G$,
and the strong operator limit of $p_t$, as $t\to 0$, is $p_0$. We write $p$ in place of $p_0$.

Represent $B$ faithfully on a Hilbert space $\Hi$.
Denote by $T$ the (unbounded) inverse of $ph$, which is defined on $p(\Hi)$. Observe that $Tp_t$ is a bounded
and positive operator on $p_t(\Hi)$. For $g\in G$, set $V_g=Tu_g$, which is a possibly unbounded operator defined on all of $\Hi$.
Set $v_g=V_gp$, regarded as an operator on $p(\Hi)$. We claim that $v_g$ is a unitary (and, in particular, bounded).
To show this, let $t\in (0,1)$, and let $\xi,\eta\in \Hi$. We use that $u_g$ commutes with $Tp_t$ at the fourth step, to get
\begin{align*}
\langle v_gp_t(\xi),v_g p_t(\eta)\rangle &= \langle Tu_gp_t(\xi),v_gp_t(\eta)\rangle \\
&= \langle Tp_tu_gp_t(\xi),v_gp_t(\eta)\rangle \\
&= \langle u_gp_t(\xi),Tv_gp_t(\eta)\rangle \\
&= \langle p_t(\xi),Tu_g^*v_gp_t(\eta)\rangle \\
&= \langle p_t(\xi),v_{g^{-1}}v_gp_t(\eta)\rangle\\
&=\langle p_t(\xi),p_t(\eta)\rangle,
\end{align*}
since $v_gv_{g^{-1}}=Tu_gTu_{g^{-1}}=T^2u_1^2$ acts as a unit on $p_t(\Hi)$.
It follows that the restriction of $v_g$ to $p_t(\Hi)$ is a unitary, with inverse $v_{g^{-1}}$. A similar
computation shows that $g\mapsto v_g|_{p_t(\Hi)}$ is a unitary representation. Since $p_t\to p$
strongly as $t\to 0$, we conclude that $v_g$ is a unitary operator on $p(\Hi)$, and that the resulting
map $v\colon G\to \U(p(\Hi))$ is a unitary representation. In other words, $v_g\colon G\to \U(pB^{\ast\ast}p)$
is a unitary representation.

Finally, it is clear that $u_g=hv_g=v_gh$ for all $g\in G$, so the proof is finished.
\end{proof}

In particular, it follows that order zero representations of a group $G$ are in one-to-one correspondence with
(completely positive contractive) order zero maps from $C^*(G)$. Instead of giving a proof for this, we prove
a more general result for covariant representations of crossed products.

We denote by $A\rtimes_\alpha G$ the maximal crossed product of $A$ by $\alpha$. In the following proposition,
when $u$ is a unitary representation, the conclusion is precisely the universal property of the crossed
product.

\begin{prop}\label{prop:cpcozCP}
Let $G$ be a locally compact group, let $A$ be a \ca, and let $\alpha\colon G\to\Aut(A)$ be an action.
Let $B$ be another \ca, let $\varphi\colon A\to M(B)$ be a homomorphism, and let $u\colon G\to B$
be an order zero representation, satisfying $u_g\varphi(a)=\varphi(\alpha_g(a))u_g$
for all $g\in G$ and for all $a\in A$. Define a map
\[\psi\colon L^1(G,A,\alpha)\subseteq A\rtimes_\alpha G\to B\]
by $\psi(\xi)=\int_G \varphi(\xi(g))u_g\ dg$ for all $\xi\in L^1(G,A,\alpha)\subseteq A\rtimes_\alpha G$.
Then $\psi$ extends to a completely positive contractive order zero map $A\rtimes_\alpha G\to B$.\end{prop}
\begin{proof}
Let $p\in B^{\ast\ast}$ be the projection
and let $v\colon G\to\U(pB^{\ast\ast}p)$ be the unitary representation provided by \autoref{prop:ozRepozMor}.
It is clear that $u_1$ commutes with the image of $\varphi$. Represent $B$ faithfully on a Hilbert space $\Hi$.
An argument similar to the one used
in the proof of \autoref{prop:ozRepozMor} (using the unbounded operator $T$ and the projections $p_t$),
shows that $(v,\varphi)$ is a covariant representation of $(G,A,\alpha)$. Let $\pi\colon A\rtimes_\alpha G\to \B(\Hi)$
be its integrated form, which on $L^1(G,A,\alpha)$ is given by $\pi(f)=\int_G \varphi(f(g))v_g\ dg$ for
all $f\in L^1(G,A,\alpha)$. Then $u_1$ commutes with the image of $\pi$, and it is clear that
$\pi(f)u_1=\psi(f)$ for all $f\in L^1(G,A,\alpha)$. It thus follows that $\psi$ extends to a completely positive
contractive order zero map, as desired.
\end{proof}

In the rest of this section, we characterize actions that are dual to compact groups actions with finite Rokhlin
dimension. As it turns out, there is a dimensional notion dual to Rokhlin
dimension (with and without commuting towers), which we call the representability dimension (with and without commuting colors);
see the definition below. The relationship between Rokhlin dimension and representability dimension is clarified in \autoref{thm:duality}.

\begin{df} \label{df:repdim}
Let $\Gamma$ be a discrete group, let $B$ be a \ca, and let $\beta\colon \Gamma\to\Aut(B)$ be an action.
Given $d\in\N$, we say that $\beta$ is \emph{approximately representable with $d$ colors}, and write
$\dim_{\mathrm{rep}}(\beta)\leq d$, if for every finite subsets $F\subseteq B$ and $K\subseteq \Gamma$,
and for every $\ep>0$, there
exist contractions $x_{\gamma}^{(j)}\in B$, for $\gamma\in\Gamma$ and $j=0,\ldots,d$, with $x_{1}^{(j)}$
positive for all $j=0,\ldots,1$, and satisfying:
\be\item $\left\|((x_{\gamma}^{(j)})^*x_{\gamma}^{(j)}-x_{\gamma}^{(j)}(x_{\gamma}^{(j)})^*)b\right\|<\ep$ for all $\gamma\in K$ and $b\in F$;
\item $\left\|(x_{\gamma}^{(j)}x_{\delta}^{(j)}-x_{1}^{(j)}x_{\gamma\delta}^{(j)})b\right\|<\ep$ for all $\gamma\in K$, $\delta\in \Gamma$, and $b\in F$;
\item $\left\|(\beta_\gamma(x_{\delta}^{(j)})-x_{\gamma\delta\gamma^{-1}}^{(j)})b\right\|<\ep$ for all $\gamma\in K$, $\delta\in\Gamma$, and $b\in F$;
\item $\left\|\left(\sum\limits_{j=0}^d (x_{1}^{(j)})^*x_{1}^{(j)}\right)b-b\right\|<\ep$ for all $b\in F$;
\item $\left\|x_{\gamma}^{(j)}b-\beta_\gamma(b)x_{\gamma}^{(j)}\right\|<\ep$
for all $\gamma\in K$, and $b\in F$.\ee
We write $\dim_{\mathrm{rep}}(\beta)$ for the smallest integer $d$ satisfying $\dim_{\mathrm{rep}}(\beta)\leq d$,
and call it the \emph{representability dimension} of $\beta$.

Similarly, given $d\in\N$, we say that $\beta$ is \emph{approximately representable with $d$ commuting colors}, and write
$\dim^{\mathrm{c}}_{\mathrm{rep}}(\beta)\leq d$, if for every finite subset $F\subseteq B$ and for every $\ep>0$, there
exist contractions $x_{\gamma}^{(j)}\in B$, for $\gamma\in\Gamma$ and $j=0,\ldots,d$, satisfying the conditions (1) through (5) above,
in addition to:
\be\item[(6)] $\left\|(x_{\gamma}^{(j)}x_{\delta}^{(k)}-x_{\delta}^{(k)}x_{\gamma}^{(j)})b\right\|<\ep$ \ee
for all $\gamma,\delta\in \Gamma$, for all $j,k=0,\ldots,d$, and for all $b\in F$.
As before, we write $\dim^{\mathrm{c}}_{\mathrm{rep}}(\beta)$ for the smallest integer $d$ satisfying $\dim^{\mathrm{c}}_{\mathrm{rep}}(\beta)\leq d$,
and call it the \emph{representability dimension with commuting colors} of $\beta$.
\end{df}

In the definition above, we get an equivalent notion if condition (5) is replaced by
\be\item[(5')] $\left\|(x_{\gamma}^{(j)}b-\beta_\gamma(b)x_{\gamma}^{(j)})c\right\|<\ep$
for all $\gamma\in K$, and $b,c \in F$.\ee
To see the equivalence, apply (5') to the finite set $\widetilde{F}=\bigcup_{\gamma\in K}\beta_\gamma(F\cup F^*)$,
so that
\[\beta_\gamma(b)x_{\gamma}^{(j)}c\approx_{\ep} \beta_\gamma(b)\beta_{\gamma}(c)x_\gamma^{(j)},\]
for all $b,c\in F$, since $\beta_\gamma(b)\in \widetilde{F}$.

\begin{rem}\label{rmk:RepDimApprRep}
In the context of the above definition, when $\Gamma$ is finite and $B$ is unital, then $\dim_{\mathrm{rep}}(\beta)=0$ if and only if
$\beta$ is approximately representable in the sense of Definition~3.6 and Remark~3.7 in~\cite{Izu_finiteI_2004}.
More generally, our notion of representability dimension zero generalizes Definition~4.21
in~\cite{BarSza_sequentially_2016}, except that the said definition is  only for separable \ca s and countable groups.
\end{rem}

Let $\beta\colon \Gamma\to\Aut(B)$ be an action of a discrete group $\Gamma$ on a \ca\ $B$. For $\gamma\in \Gamma$, we let $v_\gamma\in M(B\rtimes_\beta\Gamma)$
be the canonical unitary implementing $\beta_\gamma$. We define an (inner) action
\[\lambda^\beta\colon \Gamma\to\Aut(B\rtimes_\beta\Gamma)\]
by $\lambda^\beta_\gamma=\Ad(v_\gamma)$ for $\gamma\in\Gamma$. Whenever $B\rtimes_\beta\Gamma$ is regarded as a $\Gamma$-algebra, it is with respect to $\lambda^\beta$.
Finally, for $\gamma\in\Gamma$ we write $u_\gamma\in C^*(\Gamma)$ for the canonical unitary.

In the next two results, the assumption that the acting group be amenable
is unnecessary if the algebra on which it acts is unital. However, we will need these results
for dual actions of infinite compact abelian groups, which always act on nonunital algebras (see \autoref{thm:duality}).

\begin{thm}\label{thm:eqAppRepColors}
Let $\Gamma$ be an amenable countable group, let $B$ be a
\ca, let $\beta\colon \Gamma\to\Aut(B)$ be an action, and let $d\in\N$.
Then the following are equivalent:
\be\item $\dim_{\mathrm{rep}}(\beta)\leq d$;
\item There exist completely positive contractive order zero maps
\[\rho_0,\ldots,\rho_d\colon C^*(\Gamma)\to B_\I\]
satisfying the following conditions for all $\gamma,\delta\in \Gamma$ and all $j=0,\ldots,d$:
\be\item[(2.a)] $(\beta_\I)_\gamma(\rho_j(u_\delta))b=\rho_j(u_{\gamma\delta\gamma^{-1}})b$ for all $b\in B\subseteq B_\infty$;
\item[(2.b)] $\left(\sum\limits_{j=0}^d\rho_j(1)^2\right)b=b$ for all $b\in B\subseteq B_\infty$;
\item[(2.c)] $\rho_j(u_\gamma)b=(\beta_\I)_\gamma(b)\rho_j(u_\gamma)$ for all $b\in B\subseteq B_\infty$.\ee
\item There exist completely positive contractive, $\Gamma$-equivariant order zero maps
\[\psi_0,\ldots,\psi_d\colon B\rtimes_\beta\Gamma \to B_\I\]
satisfying the following conditions for all $j=0,\ldots,d$, for all $b\in B$, and for all $x\in B\rtimes_\beta\Gamma$:
\be\item[(3.a)] $\psi_j(bx)=b\psi_j(x)$;
\item[(3.b)] $\sum\limits_{j=0}^d\psi_j(b)=b$.
\ee\ee

There is an analogous statement for representability dimension with commuting colors:
in (2) above, the maps $\rho_j$ must have commuting ranges, while in (3) we must have $[\psi_j(x),\psi_k(y)]=0$
for all $j,k=0,\ldots,d$, whenever $[x,y]=0$.
\end{thm}
\begin{proof}
We only prove the theorem for the case of $\dim_{\mathrm{rep}}(\beta)\leq d$; the case of commuting colors is analogous.

That (1) implies (2) follows immediately from \autoref{prop:ozRepozMor} and the universal property
of $C^*(\Gamma)$. We show that (3) implies (1). Let $\ep>0$, and let $F\subseteq B$ and $K\subseteq \Gamma$ be
finite subsets. Without loss of generality, we assume that $F$ contains only contractions.
For $j=0,\ldots,d$, define $\varphi_j=\psi_j^{1/2}$ using functional calculus for order zero maps
(see Corollary~3.2 in~\cite{WinZac_completely_2009}). If $\psi_j$ has the
form $\psi_j(z)=h_j\pi_j(z)$ for all $z\in C^*(\Gamma)$, where $h_j$ is a positive contraction and $\pi_j$ is a
homomorphism as in Theorem~2.3 in~\cite{WinZac_completely_2009}, then $\varphi_j$ has the form $\varphi_j(z)=h_j^{1/2}\pi_j(z)$
for all $z\in C^*(\Gamma)$. By \autoref{lma:FunctCalcOz}, the map $\varphi_j$ is also equivariant.

Recall that any approximate unit in $B$ is also an approximate unit in $B\rtimes_\beta\Gamma$.
Choose a positive contraction $e\in B$ satisfying the following conditions for all $k=1,2$, for all $b\in F$,
and for all $\gamma\in K$:
\[\|e^kb-b\|< \ep, \ \|be^k-b\|<\ep, \ \mbox{and} \ \|\beta_{\gamma^{-1}}(e^k)-e^k\|<\ep.\]
(One can take $e$ to be a suitable
element in an approximate unit for $B$, and average its images under
$\beta$ over a F\o lner set $L\subseteq \Gamma$ satisfying $\frac{|KL\triangle L|}{|L|}<\ep$. Recall also that
if $(a_\lambda)_{\lambda\in\Lambda}$ is an approximate unit, then so is $(a_\lambda^2)_{\lambda\in\Lambda}$.)

For $j=0,\ldots,d$ and $\gamma\in\Gamma$, set $x_\gamma^{(j)}=\varphi_j(ev_\gamma)\in B_\I$. We check the first three conditions
in \autoref{df:ozRepGp}, since the other ones are similar.

For $\gamma\in K$, for $b\in B$, and for $j=0,\ldots,d$, we have
\begin{align*} \left\|b\left((x_{\gamma}^{(j)})^*x_{\gamma}^{(j)}-x_{\gamma}^{(j)}(x_{\gamma}^{(j)})^*\right)\right\|&
= \left\|b\left( \varphi_j(ev_\gamma)^*\varphi_j(ev_\gamma)-\varphi_j(ev_\gamma)\varphi_j(ev_\gamma)^*\right)\right\|\\
&= \left\|h_j^{1/2}\varphi_j(bv_\gamma^*e^2v_\gamma)-h_j^{1/2}\varphi_j(be^2)\right\|\\
&= \left\|h_j^{1/2}\varphi_j(b\beta_{\gamma^{-1}}(e^2)-be^2)\right\|\\
&\leq \left\|h_j^{1/2}\right\|\|\varphi_j(b\beta_{\gamma^{-1}}(e^2)-be^2)\|&\\
&\leq \|b\|\|\beta_{\gamma^{-1}}(e^2)-e^2\|<\ep,
\end{align*}
which proves the first condition.

Recall that $v_1=1$. For $\gamma\in K$, for $\delta\in\Gamma$, for $b\in F$, and for $j=0,\ldots,d$, we have
\begin{align*} \left\|b\left(x_{\gamma}^{(j)}x_{\delta}^{(j)}-x_{1}^{(j)}x_{\gamma\delta}^{(j)}\right)\right\|&
=\| \varphi_j(bev_\gamma)\varphi_j(ev_\delta)-\varphi_j(bev_1)\varphi_j(ev_{\gamma\delta})\|\\
&=\|\varphi_j(be)\varphi_j(v_\gamma ev_\delta)-\varphi_j(be)\varphi_j(ev_{\gamma\delta})\|\\
&=\|\varphi_j(be)\|\|\varphi_j(v_\gamma ev_\delta-ev_{\gamma\delta})\|\\
&\leq \|\beta_{\gamma}(e)-e\|<\ep,
\end{align*}
which proves the second condition.

For $\gamma\in K$, for $\delta\in\Gamma$, for $b\in B$, and for $j=0,\ldots,d$, we use that $\varphi_j$ is equivariant at the second step to get
\begin{align*} \left\|b\left((\beta_\I)_{\gamma}(x_{\delta}^{(j)})-x_{\gamma\delta\gamma^{-1}}^{(j)}\right)\right\|&
=\| b\left((\beta_\I)_\gamma(\varphi_j(ev_\delta))-\varphi_j(ev_{\gamma\delta\gamma^{-1}})\right)\|\\
&=\|b\left(\varphi_j(v_\gamma ev_\delta v_{\gamma^{-1}})-\varphi_j(ev_{\gamma\delta\gamma^{-1}})\right)\|\\
&=\|\varphi_j(bv_\gamma ev_\delta v_{\gamma^{-1}})-\varphi_j(bev_{\gamma\delta\gamma^{-1}})\|\\
&=\|v_\gamma e-ev_\gamma\| = \|\beta_{\gamma}(e)-e\|<\ep,
\end{align*}
which proves the third condition.

We leave the verification of the remaining two conditions to the reader.

We now show that (2) implies (3). Define a linear map $\psi_j\colon B\rtimes_\beta\Gamma\to B_\I$ by
\[\psi_j(bv_\gamma)=b\rho_j^2(u_\gamma)\]
for all $j=0,\ldots,d$, for all $b\in B$ and for all $\gamma\in\Gamma$. (Recall that $\rho_j^2$ is defined
using functional calculus for order zero maps.)
The assignment $\gamma\mapsto \rho_j^2(v_\gamma)$
is an order zero representation of $\Gamma$ in the sense of \autoref{df:ozRepGp}, and
by condition (2.c), it
satisfies the covariance condition $\rho_j^2(v_\gamma)b=\beta_\gamma(b)\rho_j^2(v_\gamma)$
for all $\gamma\in\Gamma$ and $b\in B$ (it suffices to multiply the identity in (2.c) by $\rho_j(1)$
on both sides). Hence, it follows from \autoref{prop:cpcozCP} that
$\psi_j$ is a completely positive, contractive order zero map.
It clearly satisfies $\psi_j(bx)=b\psi_j(x)$ for all $b\in B$ and for all $x\in B\rtimes_\beta\Gamma$.
Using condition (2.a) at the second step, we deduce that
\[\sum_{j=0}^d \psi_j(b)= b\sum_{j=0}^d \rho_j(1)^2=b\]
for all $b\in B$.

Finally, we show that $\psi_j$ is equivariant.
First, note that condition (2.a) implies that
\[(\beta_\I)_{\gamma}(\rho_j^2(u_\delta))b=\rho_j^2(u_{\gamma\delta\gamma^{-1}})b\]
for all $\gamma,\delta\in\Gamma$ and all $b\in B$. Indeed, to obtain the identity above,
it suffices to multiply the equality in (2.a) by $\rho_j(1)$ on both sides, since
$\rho_j(1)\rho_j(x)=\rho_j^2(x)$ for all $x\in C^*(\Gamma)$.
Let $\gamma,\delta\in\Gamma$, and let $b\in B$. We use this observation at the
\begin{align*}
\psi_j(\lambda^\beta_\gamma(bv_\delta))&= \psi_j(v_\gamma bv_\delta v_{\gamma^{-1}})\\
&= \psi_j(\beta_\gamma(b)v_{\gamma\delta\gamma^{-1}})\\
&= \beta_{\gamma}(b)\rho_j^2(u_{\gamma\delta\gamma^{-1}})\\
&= \beta_\gamma(b)\beta_\gamma(\rho_j^2(u_\gamma))\\
&= \beta_\gamma(\psi_j(bv_\delta)),
\end{align*}
as desired. This finishes the proof.
\end{proof}

\begin{lma}\label{lma:StepDuality}
Let $G$ be an amenable second countable locally compact group, let $A$ and $C$ be \ca s, with $C$ unital,
and let $\alpha$ and $\gamma$ be actions
of $G$ on $A$ and $C$, respectively. Given an $\alpha$-invariant, $\sigma$-unital
subalgebra $D\subseteq A$ and $d\in\N$, the following are equivalent:
\be\item There exist equivariant completely positive contractive order zero maps
$\varphi_0,\ldots,\varphi_d\colon C \to F_\alpha(D,A)$ satisfying $\sum\limits_{j=0}^d\varphi_j(1)=1$.
\item There exist equivariant completely positive contractive order zero maps
$\theta_0,\ldots,\theta_d\colon C\otimes_{\mathrm{max}}D \to A_{\I,\alpha}$ satisfying
\bi\item $\theta_j(c\otimes a_1a_2)= a_1\theta_j(c\otimes a_2)$ for all $j=0,\ldots,d$, all $a_1,a_2\in D$ and all $c\in C$; and
\item $\sum\limits_{j=0}^d\theta_j(1\otimes a)=a$ for all $a\in D$.
\ei
\ee
\end{lma}
\begin{proof}
That (1) implies (2) is easily seen by tensoring each $\varphi_j$ with the identity on $D$ and using Lemma~2.3 in~\cite{Gar_crossed_2017}.
We show that (2) implies (1).
Using an easy diagonal argument, it is enough to show the following: for every $\ep>0$, for every self-adjoint finite subset $F\subseteq D$,
and for every compact subset $K\subseteq G$, there exist completely positive contractive order zero maps
$\varphi_0,\ldots,\varphi_d\colon C \to A_{\I,\alpha}$ satisfying the following conditions for all $j=0,\ldots,d$, for all
$c\in C$ and for all $a\in F$:
\be
\item[(a)] $\|\varphi_j(c)a-a\varphi_j(c)\|<\ep\|c\|$;
\item[(b)] $\sup\limits_{g\in K}\|(\alpha_\I)_g(\varphi_j(c))-\varphi_j(\gamma_g(c))\|<\ep \|c\|$;
\item[(c)] $\left\|\left(\sum\limits_{j=0}^d\varphi_j(1)\right)a-a\right\|<\ep$.\ee

Let $\ep>0$, $F\subseteq D$ and $K\subseteq G$ as above be given. Find a positive contraction $e\in D$ satisfying
\[\sup\limits_{g\in K}\|\alpha_g(e)-e\|<\ep, \ \|ea-ae\|<\ep, \ \|ea-a\|<\ep \ \mbox{and} \ \|ae-a\|<\ep\]
for all $a\in F$. (It is enough to take a suitable element in an approximate unit for $D$, and average its images under
$\alpha$ over a F\o lner set $L\subseteq G$ satisfying $\frac{|KL\triangle L|}{|L|}<\ep$.)
For $j=0,\ldots,d$, define $\varphi_j\colon C\to A_{\I,\alpha}$ by $\varphi_j(c)=\theta_j(c\otimes e)$ for $c\in C$.
It is clear that $\varphi_j$ is a completely positive contractive order zero map. Moreover, given $a\in F$ and
$c\in C$, we have
\begin{align*}
\|\varphi_j(c)a-a\varphi_j(c)\|&=\|\theta_j(c\otimes e)a-a\theta_j(c\otimes e)\|\\
&=\|\theta_j(c\otimes (ea-ae))\|\\
&\leq \|c\|\|ea-ae\|<\ep\|c\|,\end{align*}
which verifies condition (a). To check condition (b), given $c\in C$ and $j=0,\ldots,d$, we have
\begin{align*}\sup\limits_{g\in K}\|(\alpha_\I)_g(\varphi_j(c))-\varphi_j(\gamma_g(c))\|&=
\sup\limits_{g\in K}\|(\alpha_\I)_g(\theta_j(c\otimes e))-\theta_j(\gamma_g(c)\otimes e)\|\\
&=\sup\limits_{g\in K}\|\theta_j(\gamma_g(c)\otimes \alpha_g(e))-\theta_j(\gamma_g(c)\otimes e)\|\\
&\leq \|\gamma_g(c)\| \sup\limits_{g\in K}\|\alpha_g(e)-e\|<\ep \|c\|,\end{align*}
as desired. To check condition (c), we let $a\in F$ and compute
\begin{align*}
\left\|\left(\sum\limits_{j=0}^d\varphi_j(1)\right)a-a\right\|&=\left\|\left(\sum\limits_{j=0}^d\theta_j(1\otimes e)a\right)-a\right\|\\
&=\left\|\sum\limits_{j=0}^d\theta_j(1\otimes ea)-a\right\|\\
&=\left\|ea-a\right\|<\ep.
\end{align*}
This completes the proof.
\end{proof}

The following is the main result of this section. It generalizes previously known characterizations of duals of Rokhlin actions
of Izumi when $G$ is finite
and $A$ is separable (Theorem~3.8 in~\cite{Izu_finiteI_2004}).
In particular, since we do not make any cardinality assumptions or $G$ or on $A$, we obtain new information
even in the well-studied case of Rokhlin actions (when $d=0$ ).

\begin{thm}\label{thm:duality}
Let $A$ be a \ca, let $G$ be a second countable compact abelian group,
let $\alpha\colon G\to\Aut(A)$ be an action, and
denote by $\widehat{\alpha}\colon \widehat{G}\to\Aut(A\rtimes_\alpha G)$ its dual action. Then
\[\dim_{\mathrm{Rok}}(\alpha)=\dim_{\mathrm{rep}}(\widehat{\alpha}) \ \ \mbox{ and } \ \
\dim^{\mathrm{c}}_{\mathrm{Rok}}(\alpha)=\dim^{\mathrm{c}}_{\mathrm{rep}}(\widehat{\alpha}).\]
\end{thm}
\begin{proof} We only show the statement for the formulation without commuting towers, since the other one is proved
analogously. It is enough to show that for every $d\in\N$, we have
\[\dim_{\mathrm{Rok}}(\alpha)\leq d \ \ \mbox{ if and only if } \ \ \dim_{\mathrm{rep}}(\widehat{\alpha})\leq d.\]

We divide the proof into proving the equivalence of the following statements, where we identify $D$ with the subalgebra
of $C(G,D)$ of constant functions:
\be\item[(a)] $\dim_{\mathrm{Rok}}(\alpha)\leq d$;
\item[(b)] For every $\alpha$-invariant $\sigma$-unital subalgebra $D\subseteq A$, there exist $G$-equivariant completely positive contractive order zero maps
\[\theta_0,\ldots,\theta_d\colon C(G,D) \to A_{\I,\alpha}\]
satisfying
\bi\item $\theta_j(af)= a\theta_j(f)$ for all $j=0,\ldots,d$, all $a\in D$ and all $f\in C(G,D)$; and
\item $\sum\limits_{j=0}^d\theta_j(a)=a$ for all $a\in D$.
\ei
\item[(c)] For every $\alpha$-invariant $\sigma$-unital subalgebra $D\subseteq A$, there exist $\widehat{G}$-equivariant completely positive contractive order zero maps
\[\psi_0,\ldots,\psi_d\colon C(G,D)\rtimes_{\texttt{Lt}\otimes\alpha}G \to (A\rtimes_\alpha G)_{\I}\]
satisfying:
\bi\item $\psi_j(bx)=b\psi_j(x)$ for all $b\in D\rtimes_\alpha G$, where we regard $D\rtimes_\alpha G$ both as a subset of
$A\rtimes_\alpha G$ and of $C(G,D)\rtimes_{\texttt{Lt}\otimes\alpha}G$, and for all
$x\in C(G,D)\rtimes_{\texttt{Lt}\otimes\alpha}G$; and
\item $\sum\limits_{j=0}^d \psi_j(b)=b$ for all $b\in D\rtimes_\alpha G$.
\ei
\item[(d)] $\dim_{\mathrm{rep}}(\widehat{\alpha})\leq d$.\ee

We show that (a) is equivalent to (b).
By definition, $\dim_{\mathrm{Rok}}(\alpha)\leq d$ is equivalent to the existence, for every $\alpha$-invariant $\sigma$-unital subalgebra $D\subseteq A$,
of equivariant completely positive
contractive maps $\varphi_0,\ldots,\varphi_d\colon C(G)\to F_\alpha(D,A)$ satisfying $\sum\limits_{j=0}^d\varphi_j(1)=1$.
By \autoref{lma:StepDuality}, that is equivalent to the statement in (b).

We prove that (b) implies (c). For $j=0,\ldots,d$, let
$\psi_j\colon C(G,D)\rtimes_{\texttt{Lt}\otimes\alpha}G\to (A\rtimes_\alpha G)_{\I}$ be the $\widehat{G}$-equivariant
completely positive contractive order zero map associated to $\theta_j$ as in \autoref{cpcequivCP}.
For $\xi\in L^1(G,C(G,D),\texttt{Lt}\otimes\alpha)$, this map is given by
$\psi_j(\xi)(g)=\theta_j(\xi(g))$ for all $g\in G$. We claim that these maps satisfy the conditions under item (c) above.
To verify the first one, it is enough to assume that $b\in D\rtimes_\alpha G$ has the form
$b=\sum_{k=1}^n f_ka_k$ for some $f_1,\ldots,f_n\in C(G)$ and $a_1,\ldots,a_n\in D$, since elements of this
form are dense in $L^1(G,D,\alpha)$, and hence in $D\rtimes_\alpha G$.
Similarly, we can assume that $x\in C(G,D)\rtimes_{\texttt{Lt}\otimes\alpha}G$ has the form
$x=\sum_{\ell=1}^m h_\ell y_\ell$ for some $h_1,\ldots,h_m\in C(G)$ and $y_1,\ldots,y_m\in C(G,D)$.
Given $g\in G$, we use the second first identity in (b) at the third step to get:
\begin{align*}
\psi_j(bx)(g)&=\psi_j\left(\sum_{k=1}^n \sum_{\ell=1}^m f_ka_kh_\ell y_\ell\right)(g)\\
&= \theta_j\left(\sum_{k=1}^n \sum_{\ell=1}^m f_k(g)a_kh_\ell(g) y_\ell\right)\\
&= \sum_{k=1}^n \sum_{\ell=1}^m f_k(g)h_\ell(g)\theta_j\left(a_k y_\ell\right)\\
&= \sum_{k=1}^n f_k(g)a_k\sum_{\ell=1}^m h_\ell(g)\theta_j\left(y_\ell\right)\\
&= b(g) \theta_j(x(g)) = (b\psi_j(x))(g),
\end{align*}
as desired. Similarly, for $g\in G$, we have
\[\sum_{j=0}^d\psi_j(b)(g)=\sum_{j=0}^d\theta_j(b(g))=b(g),\]
and hence $\sum_{j=0}^d\psi_j(b)=b$. This proves (c).

We now show that (c) implies (a).
We take $\widehat{G}$-crossed products in the statement of (c) and apply Takai duality together with \autoref{cpcequivCP} to obtain
$G$-equivariant completely positive contractive order zero maps
\[\widetilde{\theta}_0,\ldots,\widetilde{\theta}_d\colon C(G, D\otimes\K(L^2(G))) \to (A\otimes\K(L^2(G)))_{\I,\alpha}\]
satisfying the conditions under item (b) above. (The verification of these conditions is
analogous to how we verified that (b) implies (c), by working with appropriate dense subalgebras
of the crossed product.) By the equivalence between (a) and (b) applied to
the $G$-action $\alpha\otimes\Ad(\lambda)$ on $A\otimes\K(L^2(G))$, we conclude that $\dim_{\mathrm{Rok}}(\alpha\otimes\Ad(\lambda))\leq d$.
Since $\dim_{\mathrm{Rok}}(\alpha\otimes\Ad(\lambda))=\dim_{\mathrm{Rok}}(\alpha)$
by part~(3) of Proposition~6.8 in~\cite{GarKalLup_rokhlin_2017}, this shows that (c) implies (a).

That (c) is equivalent to (d) is the content of \autoref{thm:eqAppRepColors}, since the $\widehat{G}$-algebra
$(C(G,D)\rtimes_{\texttt{Lt}\otimes\alpha}G,\widehat{\texttt{Lt}\otimes\alpha})$
is equivariantly isomorphic to $(D\rtimes_\alpha G\rtimes_{\widehat{\alpha}}\widehat{G}, \lambda^{\widehat{\alpha}})$ by Takai duality.
This completes the proof of the theorem.
\end{proof}

We present some applications of \autoref{thm:duality} to the ideal structure of crossed products by actions with finite Rokhlin dimension; see
\autoref{cor: MorEquiv}.
Recall the following definition from \cite{Kis_simple_1980}.

\begin{df}\label{df: StrConnSpectr}
Let $G$ be a second countable compact abelian group, let $A$ be a \ca, and let $\alpha\colon G\to\Aut(A)$ be a continuous action. Given
$\tau$ in $\widehat{G}$, denote by $A_\tau$ its associated eigenspace, this is,
\[A_\tau=\{a\in A\colon \alpha_g(a)=\tau(g) a \ \mbox{ for all } g \in G\}.\]
The \emph{strong Arveson spectrum} of $\alpha$, denoted $\widetilde{\mbox{Sp}}(\alpha)$, is the set
\[\widetilde{\mbox{Sp}}(\alpha)=\{\tau\in \widehat{G}\colon \overline{A_\tau^*A A_\tau}=A\}.\]
If $\widetilde{\mbox{Sp}}(\alpha)=\widehat{G}$, then we say that $\alpha$ is \emph{saturated}.

Denote by $\mathrm{Her}_\alpha(A)$ the set of all $\alpha$-invariant hereditary subalgebras of $A$. Then
the \emph{strong Connes spectrum} of $\alpha$, denoted $\widetilde{\Gamma}(\alpha)$, is the set
\[\widetilde{\Gamma}(\alpha)=\bigcap_{B\in \mathrm{Her}_\alpha(A)}\widetilde{\mbox{Sp}}(\alpha|_B).\]
If $\widetilde{\Gamma}(\alpha)=\widehat{G}$, then we say that $\alpha$ is \emph{hereditarily saturated}.
\end{df}

\begin{prop} \label{prop: FullConnesSpec}
Let $G$ be a compact abelian group, let $A$ be a \ca, and let $\alpha\colon G\to\Aut(A)$ be an action. If
$\dimRok(\alpha)<\I$, then $\alpha$ is hereditarily saturated, this is, $\widetilde{\Gamma}(\alpha)=\widehat{G}$.\end{prop}
\begin{proof} By Lemma 3.4 in \cite{Kis_simple_1980}, we have
\[\widetilde{\Gamma}(\alpha)=\{\tau\in\widehat{G}\colon \widehat{\alpha}_\tau (I)\subseteq I\ \mbox{ for all ideals } I\subseteq A\rtimes_\alpha G\}.\]
Let $I$ be an ideal in $A\rtimes_\alpha G$, and let $\tau$ be an element in $\widehat{G}$. Fix $b$ in $I$ and set
$d=\dimRok(\alpha)$. For every $m$ in $\N$, use \autoref{thm:duality} to find contractions $x_\gamma^{(j),m}$, for $\gamma\in\Gamma$ and $j=0,\ldots,d$
such that, in particular,
\[\left\|\widehat{\alpha}_\tau(b)-\sum_{j=0}^d (x_\gamma^{(j),m})^*bx_\gamma^{(j),m}\right\|<\frac{1}{m}.\]
Since $I$ is an ideal, it follows that $\sum\limits_{j=0}^d (x_\gamma^{(j),m})^*bx_\gamma^{(j),m}$ belongs to $I$.
We conclude that
$\widehat{\alpha}_\tau(b)$ is the limit in norm of elements of $I$, so it belongs to $I$ itself. Hence
$\widehat{\alpha}_\tau(I)\subseteq I$. Since $I$ and $\tau$ are arbitrary, we conclude that $\widetilde{\Gamma}(\alpha)=\widehat{G}$.\end{proof}

A similar result for $\R$-actions has been proved in Proposition~3.11 in~\cite{HirSzaWinWu_rokhlin_2017}.

\begin{cor}\label{cor: MorEquiv}
Let $G$ be a second countable compact abelian group, let $A$ be a \ca, and let $\alpha\colon G\to\Aut(A)$ be an action. If
$\dimRok(\alpha)<\I$, then every ideal $J$ in $A\rtimes_\alpha G$ has the form $J=I\rtimes_\alpha G$ for some
$\alpha$-invariant ideal $I$ in $A$. Moreover, $A^G$ is Morita equivalent to $A\rtimes_\alpha G$. In particular, if $A$ is simple,
then so is $A\rtimes_\alpha G$. \end{cor}
\begin{proof} We have $\widetilde{\Gamma}(\alpha)=\widehat{G}$ by \autoref{prop: FullConnesSpec}. The
first claim now follows from Theorem~5.14 in~\cite{Phi_freeness_2009} (see \cite{Kis_simple_1980} for a proof).\\
\indent It follows from a result in \cite{GooLazPel_spectra_1994} (reproduced as Theorem 5.10 in \cite{Phi_freeness_2009}) that $\alpha$
is hereditarily saturated. Hence it is saturated, and the second claim follows from Proposition 7.1.3 in
\cite{Phi_equivariant_1987}.\end{proof}

The following application was announced in Example~5.1 of~\cite{Gar_regularity_2017}.

\begin{eg}
Fix $\theta\in\R\setminus\Q$, and let $A_\theta$ be the associated irrational
rotation algebra, that is, the universal \ca\ generated by two unitaries $u$ and $v$ subject
to the relation $uv=e^{\pi i \theta}vu$. Define an action $\gamma\colon\T\to\Aut(A_\theta)$
by $\gamma_\zeta(u)=\zeta u$ and $\gamma_\zt(v)=v$ for all $\zeta\in\T$. We claim that
$\dimRok(\gamma)=\I$. Indeed, the crossed product $A_\theta\rtimes_\gamma\T$ is easily seen
to be isomorphic to $C(S^1)\otimes\K$, so it is not simple. Since $A_\theta$ is simple, the
result follows from \autoref{cor: MorEquiv}.

A similar argument shows that other related actions, such as $\beta\colon\T^2\to\Aut(A_\theta)$
given by $\beta_{(\zeta,\omega)}(u)=\zeta u$ and $\beta_{(\zeta,\omega)}(v)=\omega v$ for
$(\zeta,\omega)\in\T^2$, have infinite Rokhlin dimension. This is in stark contrast with the
fact, proved in \cite{HirPhi_rokhlin_2015}, that the restrictions of $\gamma$ and $\beta$ to
finite subgroups of $\T$ have finite Rokhlin dimension with commuting towers.
(See also \autoref{prop:ZnAThetaRdim1} for a more general result.)
\end{eg}

\section{Tracial properties}

In this section, we show how actions of finite groups with finite Rokhlin dimension with commuting
towers enjoy a weak form of the tracial Rokhlin property from \cite{Phi_tracial_2011}, where
projections are replaced by positive elements, and the remainder is assumed to be small in all
tracial states. This notion is called the \emph{weak tracial Rokhlin property}; see \autoref{df: tRp wtRp}
for the precise definition. We point out that similar notions have been considered by a number of other
authors (\cite{Arc_crossed_2008, MatSat_stability_2012, Wan_tracial_2013}), and that all of these notions
agree for reasonably well-behaved $C^*$-algebras (for example, those which have strict comparison; see
\autoref{df:StrComp} below).

If $X$ is a simplicial complex and $k\in\N$, we denote its $k$-th skeleton by $X^{(k)}$.

\begin{thm}\label{thm:FreeActionCellComplex}
Let $G$ be a finite group, let $X$ be a finite simplicial complex.
Let $G$ act freely on $X$ in such a way that for all $k$ in $\N$,
no point in one $k$-cell of $X$ is mapped to another point in the same
$k$-cell. Let $(\mu_n)_{n\in\N}$ be a sequence of finite Borel
measures on $X$. Then there exists an open set $U\subseteq X$ such that
\be\item[(a)] $gU\cap hU=\emptyset$ for all $g,h\in G$ with $g\neq h$;
\item[(b)] $\mu_n\left(X\setminus \bigcup\limits_{g\in G}gU\right)=0$ for all $n$ in $\N$.\ee
\end{thm}
\begin{proof}
We abbreviate the orbit space $X/G$ to $Y$ throughout, and observe that the assumptions on the action imply that
$Y$ is also a
simplicial complex of the same dimension of $X$, and that $Y^{(k)}=X^{(k)}/G$ for
$k=0,\ldots,\dim(X)$. We write $\pi\colon X\to Y$ for the quotient map, which is a local
homeomorphism since the action is free and the group is finite. This means that for every $y\in Y$,
there exist an open set $O_y$ in $Y$ containing $y$, and a continuous function $s\colon O_y\to X$
satisfying $\pi\circ s=\id_{O_y}$.
(The function $s$ is called a \emph{local cross-section}.)
Set $d=\dim(X)$. By the assumptions on the action, every $d$-cell in $Y$ is contained in an open
subset where a local cross section is defined. Similarly, there is a local cross
section defined on all of $Y^{(d)}\setminus Y^{(d-1)}$. Finally, for $n\in\N$, we will denote by $\nu_n$
the measure on $Y$ defined as the push-forward of $\mu_n$, that is, $\nu_n=\pi_\ast(\mu_n)$.

Assume first that $\dim(X)=0$, so that $X$ is finite. We claim that $X$ is equivariantly homeomorphic to
$Y \times G$, where $G$ acts trivially on $Y$ and by translation on $G$.  Since $Y$ is finite,
there exists a section $s\colon Y \to X$. Define a map $\lambda\colon X\to G$ by
$x=\lambda(x)s(\pi(x))$. Note that the assignment $x\mapsto \lambda(x)$ is well-defined because the action
is free.
The map $\phi\colon X \to Y\times G$
given by
$\phi(x)=(\pi(x),\lambda(x))$
for $x\in X$, is easily seen to be an equivariant homeomorphism, with inverse given by $(y,g)\mapsto gs(y)$ for
$y\in Y$ and $g\in G$. This proves the claim. With $\phi$ as before, set $U=\phi^{-1}\left(Y\times\{e\}\right)$.
It is straightforward to check that $\bigcap\limits_{g\in G} gU=\emptyset$ and that
$\bigcup\limits_{g\in G}gU=X$.

We may assume, from now on, that $\dim(X)>0$.
We will show that
there exists an open set $V\subseteq Y$ satisfying:
\be
\item[(A)] $\nu_n(Y^{(k)}\setminus V)=0$ for all $n\in\N$ and all $k=0,\ldots, \dim(X)$, and
\item[(B)] there is a local cross section defined on all of $V$. \ee

Once this is proved, let $U$ be the image of $V$ under some local cross
section. Then $U$ is an open set in $X$ which clearly satisfies condition (a) in the statement, and moreover
\[X\setminus \bigcup_{g\in G}gU=X\setminus \pi^{-1}(V)=\pi^{-1}(Y \setminus V).\]
In particular, $\mu_n\left(X\setminus \bigcup_{g\in G}gU\right)=0$ for all $n\in\N$, as desired.

We obtain this set by constructing its intersection with the lo\-wer-di\-men\-sio\-nal skeleta, and
using induction. (In fact, our construction yields uncountably many pairwise distinct open sets
satisfying conditions (A) and (B) above.)
Recall that $Y$ has only finitely many cells of each dimension.
For $y\in Y^{(0)}$, let $E_1^{(y)},\ldots, E_{d_y}^{(y)}$ denote the collection of 1-cells
in $Y$ whose boundaries contain $y$, and let $O_y$ be an open set in $Y$ which contains $y$
and is contained in the domain of some local cross-section for the quotient map $X\to Y=X/G$.

\textbf{Claim:} there exist continuous functions
$f_{y,j}\colon [0,1]\to E_j^{(y)}$, for $y\in Y^{(0)}$ and $j=1,\ldots,d_y$,
satisfying:

\be\item The image of $f_{y,j}$ is contained in $O_y$, for all $y\in Y^{(0)}$ and all $j=1,\ldots,d_y$;
\item $f_{y,j}$ is a homeomorphism onto its image, for all $y\in Y^{(0)}$ and all $j=1,\ldots,d_y$;
\item $f_{y,j}(0)=y$, for all $y\in Y^{(0)}$ and all $j=1,\ldots,d_y$;
\item $\nu_n(f_{y,j}(1))=0$, for all $n\in\N$, all $y\in Y^{(0)}$ and all $j=1,\ldots,d_y$;
\item $f_{y,j}([0,1])\cap f_{z,k}([0,1])=\emptyset$ for all $y,z\in Y^{(0)}$ with $y\neq z$, and
for all $j=1,\ldots,d_y$ and $k=1,\ldots,d_z$.
\ee

To construct these functions, start with any collection of functions $h_{y,j}\colon [0,1]\to E^{(y)}_j$
satisfying (1), (2) and (3). By restricting them to an initial segment of the form $[0,r]$, for $r\in (0,1)$,
we can ensure that condition (5) also holds. We now explain how to fulfill condition (4). Set
\[T=\{t\in (0,1)\colon \nu_n(h_{y,j}(t))=0 \ \mbox{ for all } n\in\N, y\in Y^{(0)}, j=1,\ldots,d_y\}.\]
Then the complement of $T$ is countable (and in particular $T$ is uncountable).
Indeed, if $[0,1]\setminus T$
were uncountable, there would exist $m\in \N$, $z\in Y^{(0)}$ and
$k\in \{1,\ldots,d_y\}$ such that $\nu_m(h_{z,k}(t))>0$ for uncountably many $t\in [0,1]$. Since $h_{z,k}$ is
a homeomorphism, it follows that $\nu_m$ has an uncountable set of atoms, which contradicts the fact that
it is a probability measure. Hence the complement $T$ is countable.

The claim then follows by fixing $t_0\in T$, and letting $f_{y,j}$ be given by
$f_{y,j}(t)=h_{y,j}(t/t_0)$ for all $t\in [0,1]$.

Set
\[V=Y\setminus \left(\bigcup_{y\in Y^{(0)}}\bigcup_{j=1}^{d_y}\left\{f_{y,j}(1)\right\}\right),\]
which is an open subset of $Y$ satisfying $\nu_n(Y^{(1)}\setminus V)=0$ for all $n\in\N$.
By the choice of the
functions (specifically, by condition (1)), there exists a local cross section defined on all of
$V$ (obtained by considering cross sections on the connected components of $V$).
The base step of the induction is complete.


Let $m\leq \dim(X)$, and suppose we have constructed an open $W$ satisfying conditions (A) and (B) above,
for $k=1,\ldots,m-1$. We will construct an open set $V$ satisfying conditions (A) and (B) for $k=1,\ldots,m$.
Denote by $C_1,\ldots, C_{\ell}$ the connected components
of $W_i$. For $k=1,\ldots,\ell$, let $E_{k,1},\ldots, E_{k,d_k}$ denote the
collection of all $m$-cells in $Y$ whose boundaries intersect $C_k$.

We write $\Delta_m$ for the $m$-dimensional simplex, which we identify as
\[\Delta_m=\left\{(\lambda_1,\ldots,\lambda_m)\in \R_+^m\colon \sum_{k=1}^m \lambda_k\leq 1\right\}.\]
We also identify $\Delta_{m-1}$ with the subset of $\Delta_m$ of vectors whose last coordinate vanishes.
For $t\in [0,1]$, set
\[\Delta_m^{(t)}=\{(\lambda_1,\ldots,\lambda_m)\in \R^m_+\colon (\lambda_1,\ldots,\lambda_m/t)\in \Delta_m\}.\]
Note that $\Delta_m^{(1)}=\Delta_m$, and $\Delta_m^{(0)}=\Delta_{m-1}$ if $m\geq 3$.
When $m=2$, the subsimplex $\Delta_2^{(t)}$ can be identified with the subinterval $[0,t]$.
For $m=3$, the subsimplex $\Delta_3^{(t)}$ is shown in the picture above.
\begin{figure}
\centering
\begin{tikzpicture}
\draw (0,0) node[anchor=east]{$(0,1,0)$} -- (4,0) node[anchor=west]{$(1,0,0)$} -- (2,3.46) node[anchor=south]{$(0,0,1)$} --cycle;
\draw (0,0) node[anchor=east]{$(0,1,0)$} -- (4,0) node[anchor=west]{$(1,0,0)$} -- (2,1.6) node[anchor=south]{$(0,0,t)$} --cycle;
\end{tikzpicture}
\end{figure}

For a set $Z$, we write $Z^\circ$ for its interior.

\textbf{Claim:} there exist continuous functions $f_{k,j} \colon\Delta_m \to E_{k,d_k}$,
for $k=1,\ldots,\ell$ and $j=1,\ldots,d_k$, satisfying the following conditions:
\be\item[(1')] The image of $f_{k,j} $ is contained in $W_i$, for all $k=1,\ldots,\ell$ and all $j=1,\ldots,d_k$;
\item[(2')] $f_{k,j} $ is a homeomorphism onto its image, for all $k=1,\ldots,\ell$ and all $j=1,\ldots,d_k$;
\item[(3')] $f_{k,j} (\Delta_{m-1}^\circ)\subseteq C_k \cap E _{k,d_k}$, for all $k=1,\ldots,\ell$ and all $j=1,\ldots,d_k$;
\item[(4')] $\nu_n(f_{k,j} (\partial\Delta_{m}\setminus \Delta_{m-1}^\circ))=0$, for all $n\in\N$, all $k=1,\ldots,\ell$ and all $j=1,\ldots,d_k$;
\item[(5')] $f_{k,j} (\Delta_m)\cap f_{k',i} (\Delta_m)=\emptyset$ for all $k,k'=1,\ldots,\ell$ with $k\neq k'$, and
for all $j=1,\ldots,d_k$ and $i=1,\ldots,d_{k'}$.
\ee

To prove the existence of these functions, one argues similarly as in the inductive step:
start with any set $\{h_{k,j}\colon k=1,\ldots,\ell, j=1,\ldots,d_k\}$
of functions satisfying (1'), (2'), (3'). Since this collection of functions is finite,
there exists $r\in (0,1]$ such that the restrictions of the functions $h_{k,j}$ to $\Delta_m^{(r)}$
have disjoint ranges (so that (5') is satisfied).
Arguing as before, and using that $\nu_n$ is a probability measure for all $n\in\N$, it follows that the set
\[T=\bigcap_{n=1}^\I\bigcap_{k=1}^\ell \bigcap_{j=1}^{d_k}\left\{t\in (0,1)\colon \nu_n\left(h_{k,j}\left(\partial \Delta_m^{(t)}\setminus \Delta^\circ_{m-1}\right)\right)=0\right\}
\]
has countable complement in $[0,1]$.
Fix $t_0\in T$, and let $f_{\ell,j}$ be the function given by
$f_{y,j}(\lambda_1,\ldots,\lambda_m)=h_{y,j}(\lambda_1,\ldots,\lambda_m/t_0)$ for all $(\lambda_1,\ldots,\lambda_m)\in \Delta_m$.
Then conditions (1') through (5') are satisfied, and the claim is proved.

Set
\[V=Y\setminus \left(\bigcup_{k=1}^{\ell}\bigcup_{j=1}^{d_k}f_{k,j}(\partial \Delta_m\setminus \Delta_{m-1}^\circ)\right),\]
which is an open subset of $Y$ satisfying $\nu_n(Y\setminus V)=0$ for all $n\in\N$.
By the choice of the
functions (specifically, by condition (1')), there exists a local cross section defined on all of
$V$ (obtained by considering cross sections on the connected components of $V$). This finishes the induction, and proves the theorem.
\end{proof}

We remark that our argument in the proof of \autoref{thm:FreeActionCellComplex} depends on having at most
countably many finite Borel measures. We do not know whether the conclusion of \autoref{thm:FreeActionCellComplex} holds
if this assumption is dropped.

The following is a variant of Definition~5.2 in~\cite{HirOro_tracially_2013}. It is closely related to the projectionless
free tracial Rokhlin property of Archey (\cite{Arc_crossed_2008}), and the weak Rohlin properties of Matui-Sato (\cite{MatSat_stability_2012})
and Wang \cite{Wan_tracial_2013}.

\begin{df}\label{df: tRp wtRp}
Let $A$ be a simple unital \ca, let $G$ be a finite group and let $\alpha\colon G\to\Aut(A)$ be an action.
\be\item We say that $\alpha$ has the \emph{weak tracial Rokhlin property} if for every finite set $F\subseteq A$,
for every $\ep>0$, and for every positive element $x\in A$ of norm one, there exist positive contractions $a_g$ in $A$ for $g\in G$,
such that:
\be\item $\|\alpha_g(a_h)-a_{gh}\|<\ep$ for all $g,h\in G$;
\item $\|a_ga_h\|<\ep$ for all $g,h\in G$ with $g\neq h$;
\item $\|a_gb-ba_g\|<\ep$ for all $g\in G$ and all $b\in F$;
\item With $a=\sum\limits_{g\in G}a_g$, the element $1-a$ is Cuntz-subequivalent to $x$;
\item $\|axa\|>1-\ep$.
\ee
\item We say that $\alpha$ has the \emph{tracial Rokhlin property} if the positive contractions in (1) can be chosen
to be orthogonal projections.\ee
\end{df}

Condition (e) is automatically satisfied whenever $A$ is finite and infinite dimensional;
see Lemma~1.16 in~\cite{Phi_tracial_2011}.

\vspace{0.3cm}

We now want to elaborate on an observation made in~\cite{HirPhi_rokhlin_2015}. Fix a finite group $G$ and a nonnegative
integer $d\in\N$. By Lemma~1.9 in~\cite{HirPhi_rokhlin_2015} (see also
\autoref{thm:XRpRdim} below for a more general argument), there exists a compact free $G$-space $Y$
(depending on both $d$ and $G$) that is universal for actions with $\cdimRok\leq d$, in the following sense:
An action $\alpha\colon G\to\Aut(A)$
of $G$ on a \uca\ $A$ has $\cdimRok(\alpha)\leq d$ if and only if there is a unital equivariant homomorphism
$C(Y)\to A_\infty\cap A'$.

The space $Y$ has a very concrete description, which we proceed to describe.
Denote by $C$ the universal commutative \uca\ generated by positive contractions $f_g^{(j)}$, for $g\in G$
and $j=0,\ldots,d$, satisfying the following relations:
\be
\item $f_g^{(j)}f_h^{(j)}=0$ whenever $h\neq g$, for all $j=0,\ldots,d$;
\item $\sum\limits_{g\in G}\sum\limits_{j=0}^d f_g^{(j)}=1$.
\ee
Define an action $\gamma\colon G\to\Aut(C)$ on generators by $\gamma_g(f_h^{(j)})=f_{gh}^{(j)}$ for all $g,h\in G$
and all $j=0,\ldots,d$. Set $Y=\widehat{C}$, the maximal ideal space of $C$.
Then $Y$ is a free $G$-space, and it is readily checked that an action $\alpha\colon G\to\Aut(A)$
of $G$ on a \uca\ $A$ has $\cdimRok(\alpha)\leq d$ if and only if there is a unital equivariant homomorphism
$C(Y)\to A_\infty\cap A'$. (This homomorphism may not be injective.)

The description of $C(Y)$ as a universal \ca\ allows one to identify the space $Y$ as a simplicial complex.
We briefly describe this structure: each of the contractions $f_g^{(j)}$ determines a 0-simplex. Moreover, there is a 1-simplex between
$f_g^{(j)}$ and $f_h^{(k)}$ whenever $f_g^{(j)}f_h^{(k)}\neq 0$. In general, for $n\in\N$, there is an $n$-simplex with 0-dimensional
boundary $\left\{f_{g_0}^{(j_0)},\ldots,f_{g_n}^{(j_n)}\right\}$ whenever $f_{g_0}^{(j_0)}\cdots f_{g_n}^{(j_n)}\neq 0$.
In particular, $Y$ is a $d$-dimensional simplicial complex. (An explicit computation of the $G$-space $Y$ when $G=\Z_2$ is given in
\autoref{lma:S1rotation}: one gets $S^d$ with the antipodal map.)
Moreover, the action of $G$ on $Y$ is easily seen to be compatible with its simplicial structure,
and has the property that for all $k=0,\ldots,d$, no point in one $k$-cell of $Y$ is mapped to another point in the same $k$-cell.
It follows that actions constructed in this way satisfy the assumptions of \autoref{thm:FreeActionCellComplex}.

In the following, we work with quasitraces because we do not assume the algebra to be exact. For a quasitrace
$\tau$ on a \ca\ $A$, we denote by $d_\tau\colon (A\otimes\K)_+\to [0,\I]$ its associated dimension function, which is
given by $d_\tau(a)=\lim\limits_{n\to\I}\tau(a^{1/n})$ for all $a\in (A\otimes\K)_+$.

\begin{df}\label{df:StrComp}
We say that a \ca\ $A$ has \emph{strict comparison of positive elements by quasitraces}, usually referred to as
``strict comparison'' for short, if for every $a,b\in (A\otimes \mathcal{K})_+$ satisfying
$d_\tau(a)<d_\tau(b)$ for all quasitraces $\tau$ on $A$, then $a\precsim b$.
\end{df}

Let $\alpha\colon G\to\Aut(A)$ be an action of a finite group $G$ on a simple, unital \ca\ $A$.
If $A$ has strict comparison, then the weak tracial
Rokhlin property and the tracial Rokhlin property for $\alpha$ can be reformulated using quasitracial states.
Indeed, it is easy to see that under these assumptions, condition (d) in \autoref{df: tRp wtRp} can be replaced
by the condition that for every $\ep>0$, there is a positive contraction $a$ satisfying
$d_\tau(1-a)<\ep$ for all $\tau\in QT(A)$. Equivalently, we may only require $d_\tau(1-a)<\ep$ for all
\emph{extreme} quasitracial states on $A$.

\begin{thm} \label{thm: RdimwTRp}
Let $A$ be an infinite dimensional, simple, finite, unital \ca\ with strict comparison
and at most countably many extreme quasitraces. Let $G$ be a finite group and let $\alpha\colon G\to
\Aut(A)$ be an action. If $\cdimRok(\alpha)<\I$, then $\alpha$ has the weak tracial Rokhlin property.\end{thm}
\begin{proof}
Let $\{\tau_n\}_{n\in\N}$ be an enumeration of the set of extreme quasitraces on $A$,
allowing for repetition if this set is finite. Let $\omega\in \beta\N\setminus\N$ be a free
ultrafilter, and, for $n\in\N$, denote by $\sigma_n\colon A_\omega\to\C$ the quasitrace obtained from $\tau_n$.
Denote by $Y$ the free $G$-simplicial complex from the discussion before this
theorem; see also Lemma~1.9 in~\cite{HirPhi_rokhlin_2015} and \autoref{thm:XRpRdim}. Let
$\varphi\colon C(Y)\to A_\omega \cap A'$ be a unital equivariant homomorphism. For $n\in\N$, the map
$\sigma_n\circ\varphi\colon C(Y)\to \C$ is a quasitrace, and since $C(Y)$ is commutative, it is a
tracial state. By the Riesz Representation Theorem, there exists a Borel probability measure
$\mu_n$ on $Y$ such that $(\sigma_n\circ\varphi)(f)=\int\limits_Y f(y)\ d\mu_n(y)$ for all $f\in C(Y)$.

To prove that $\alpha$ has the weak tracial Rokhlin property, we will show that there exists
a completely positive contractive, $G$-equivariant, order zero map
\[\psi\colon (C(G),\texttt{Lt})\to (A_\omega\cap A',\alpha_\omega)\]
such that $\sigma_n(1-\psi(1))=0$ for all $n\in\N$.

Fix $m\in\N$.
Let $U\subseteq Y$ be an open subset as in the conclusion of \autoref{thm:FreeActionCellComplex} for
the sequence $(\mu_n)_{n\in\N}$.
Choose a continuous function $0\leq f^{(m)}\leq 1$ supported on $U$ satisfying:
\[\mu_n\left(\{x\in U\colon f^{(m)}(x)\neq 1\}\right)<\frac{1}{m}\]
for all $n=1,\ldots,m$. For $g\in G$, set $f^{(m)}_g=g\cdot f^{(m)}\in C(Y)$.

Let $(F_n)_{n\in\N}$ be an increasing sequence of finite subsets of $A$ with dense union.
Use the Choi-Effros Lifting Theorem for $\varphi\colon C(Y)\to A_\omega\cap A'$
to find a unital completely positive linear map
$\psi_m\colon C(Y)\to A$ such that the following conditions hold:
\be\item $\|\psi_m(f^{(m)}_g)\psi_m(f^{(m)}_h)-\psi_m(f^{(m)}_gf^{(m)}_h)\|<\frac{1}{m}$ for all $g,h\in G$;
\item $\|\alpha_h(\psi_m(f^{(m)}_g))-\psi_m(f^{(m)}_{hg})\|<\frac{1}{m}$ for all $g,h\in G$;
\item $\|\psi_m(f^{(m)}_g)b-b\psi_m(f^{(m)}_g)\|<\frac{1}{m}$ for all $g\in G$ and all $b\in F_m$;
\item for every $n=1,\ldots,m$, we have
\[d_{\tau_n}\left(1-\sum\limits_{g\in G}\psi_m(f^{(m)}_g)\right)<\frac{1}{m}.\]
\ee
For $g\in G$, set $a^{(m)}_g=\psi_m(f^{(m)}_g)\in A$ for $m\in\N$, and set
\[a_g=\eta_A((a_g^{(m)})_{m\in\N})\in A_\omega.\]
(See \autoref{df:SeqAlgs} for the definition of the canonical map $\eta_A\colon \ell^\I(A)\to A_\omega$.)
It is then easy to verify that the assignment $g\mapsto a_g$ extends to a completely positive contractive
order zero map $\psi\colon C(G)\to A_\omega\cap A'$ which is $G$-equivariant and satisfies
$\sigma_n(1-\psi(1))=0$ for all $n\in\N$. This finishes the proof.
\end{proof}

\begin{rem} In the theorem above, simplicity is only needed because it is required in the definition
of the weak tracial Rokhlin property. If the algebra is not assumed to be simple or to have strict
comparison, the conclusion is as follows: for every $\ep>0$ and for every finite subset $F\subseteq A$,
there exist positive contractions $f_g \in A$, for $g\in G$, satisfying
\be\item $\|\alpha_g(f_h)-f_{gh}\|<\ep$ for all $g,h\in G$;
\item $\|f_gf_h\|<\ep$ for all $g,h\in G$ with $g\neq h$;
\item $\|f_ga-af_g\|<\ep$ for all $g\in G$ and for all $a\in F$;
\item with $f=\sum\limits_{g\in G}f_g$, we have $d_\tau(1-f)<\ep$ for all $\tau\in T(A)$.\ee\end{rem}

\begin{cor}\label{cor: RdimTRp}
Let $A$ be a simple, unital, infinite dimensional \ca\ with tracial rank zero and at most countably
many extreme tracial states. Let $G$ be a finite group and let $\alpha\colon G\to
\Aut(A)$ be an action. If $\cdimRok(\alpha)<\I$, then $\alpha$ has the tracial Rokhlin property.\end{cor}
\begin{proof} It follows from \autoref{thm: RdimwTRp} that $\alpha$ has the weak tracial Rokhlin property.
Since $A$ has tracial rank zero and is infinite dimensional, Theorem~1.9 in \cite{Phi_finite_2015} implies that
$\alpha$ has the tracial Rokhlin property.
\end{proof}

\autoref{thm: RdimwTRp} and \autoref{cor: RdimTRp} hold in greater generality for actions on Kirchberg algebras. Indeed, in this case, finite Rokhlin dimension
(without commuting towers) is in fact equivalent to the (weak) tracial Rokhlin property, and moreover
equivalent to pointwise outerness; see \autoref{thm:outertRp}.

\begin{cor}\label{cor:RdimTAF} Finite group actions with finite Rokhlin dimension with commuting towers preserve the class of simple, nuclear,
unital, separable \ca s with tracial rank zero, as long as the algebra has at most countably many extreme tracial states.\end{cor}
\begin{proof} The result follows immediately from \autoref{cor: RdimTRp},
together with Theorem~2.6 in \cite{Phi_tracial_2011}.\end{proof}

One advantage of \autoref{cor: RdimTRp} is that finite Rokhlin dimension with commuting towers is sometimes
easier to establish than the tracial Rokhlin property. Cyclic group actions on higher dimensional noncommutative
tori provide one instance where this is the case; see \autoref{prop:ZnAThetaRdim1}. We need some preparation first.

Let $d\in\N$ and let $\Theta=(\theta_{jk})_{1\leq j,k\leq d} \in M_d(\R)$ be a skew symmetric matrix. The higher
dimensional noncommutative torus
$A_\Theta$ is the universal \uca\ generated by unitaries $u_j$, for $1\leq j\leq d$, satisfying the commutation
relations $u_ju_k=e^{2\pi i \theta_{jk}}u_ku_j$ for $1\leq j,k\leq d$. Define a rotation map $h_\Theta\colon \Z^d\to \T^d$
by
\[h_\Theta(m)=\left(e^{2\pi i \sum\limits_{j=1}^d \theta_{1j}m_j},\ldots, e^{2\pi i \sum\limits_{j=1}^d \theta_{dj}m_j}\right)\]
for $m=(m_1,\ldots,m_d)\in\Z^d$. It is clear that $h_\Theta$ is a group homomorphism.

Recall that a matrix $\Theta\in M_d(\R)$ as above is said to be \emph{nondegenerate} if whenever $x\in\Z^d$ satisfies
$e^{2\pi i \langle x,\theta y\rangle }=1$ for all $y\in\Z^d$, then $x=0$. (This condition is
equivalent to the rows of $\Theta$ forming a rationally linearly independent set in $\R^d$.)
The case $d=1$ of the following lemma is well known. Recall that given $n\in\Z^d$, there exists a (continuous) group homomorphism $\gamma_n\colon \T^d\to \C$,
given by $\gamma_n(z_1,\ldots,z_d)=z_1^{n_1}\cdots z_d^{n_d}$ for all $z=(z_1,\ldots,z_d)\in \T^d$, and that every continuous homomorphism $\T^d\to \C$
has this form (this is a restatement of Pontryagin duality for the group $\T^d$). Moreover, $\gamma_n$ is nontrivial if and only if $n\neq 0$.

The following lemma is folklore, but since we were unable to find a reference, we include a proof for the convenience of the reader.

\begin{lma}\label{lma:hThetaDense}
Adopt the notation of the discussion above. Then $\Theta$ is nondegenerate if and only if $h_\Theta$ has dense range.
\end{lma}
\begin{proof}
Set $G=\overline{h_{\Theta}(\Z^d)}$, which is a closed subgroup of $\T^d$.
Observe that $G$ is proper if and only if there is $n\in\Z^d\setminus\{0\}$ such that $\gamma_n(G)=\{1\}$. Continuity of $\gamma_n$
implies that this is in turn equivalent to $\gamma_n(h_{\theta}(m))=1$ for all $m\in\Z^d$. By definition, we have
\begin{align*}
\gamma_n(h_\Theta(m))&=\gamma_n\left(e^{2\pi i \sum\limits_{j=1}^d \theta_{1j}m_j},\ldots, e^{2\pi i \sum\limits_{j=1}^d \theta_{dj}m_j}\right)\\
&=e^{2\pi i \left[n_1\sum\limits_{j=1}^d \theta_{1j}m_j+\cdots +n_d\sum\limits_{j=1}^d \theta_{dj}m_j\right]}\\
&=e^{2\pi i \left[m_1\sum\limits_{k=1}^d n_k\theta_{k1}+\cdots +m_d\sum\limits_{k=1}^d n_k\theta_{kd}\right]}.
\end{align*}
For $j=1,\ldots,d$, let $e_j\in\Z^d$ be the canonical $j$-th basis vector.
Since $h_\Theta$ is a group homomorphism, one has $\gamma_n(h_{\theta}(m))=1$ for all $m\in\Z^d$ if and only if
$\gamma_n(h_{\theta}(e_j))=1$ for all $j=1,\ldots,d$. Using the computation above, one sees that this is the case if
and only if $\sum\limits_{k=1}^d n_k\theta_{kj}$ is an integer for all $j=1,\ldots,d$.

Summing up, we argued that $h_\Theta$ has dense range if and only if there does not exist a nonzero $n\in\Z^d$ such that
$n\Theta$ belongs to $\Z^d$, which is precisely the definition of nondegeneracy for $\Theta$. This finishes the proof.
\end{proof}




Next, we use ideas from \cite{Kis_oneparameter_1996} to show that certain quasifree actions of
finite cyclic groups on $A_\Theta$ have Rokhlin dimension one with commuting towers. Our result
contains Example~1.12 in~\cite{HirPhi_rokhlin_2015} as a particular case. For $d\geq 2$, we write
the coordinates of a vector $r\in \Z^d$ as $(r^{(1)},\ldots,r^{(d)})$.

\begin{prop}\label{prop:ZnAThetaRdim1}
Let $\Theta\in M_d(\R)$ be a nondegenerate skew symmetric matrix, let $m_1,\ldots,m_d\in\N$ be mutually coprime positive integers, and set $m=m_1\cdots m_d$.
Suppose that $m>1$. Let $\alpha$ be the automorphism of $A_\Theta$ determined by $\alpha(u_j)=e^{2\pi i/m_j}u_j$ for $j=1,\ldots,d$. Denote
also by $\alpha$ the action of $\Z_m$ on $A_\Theta$ that it determines. Then $\cdimRok(\alpha)=1$.

In particular, $\alpha$ has the tracial Rokhlin property.
\end{prop}
\begin{proof}
We claim that there exists a sequence $(r_n)_{n\in\N}$ in $\Z^d$ such that
\be
\item $\lim\limits_{n\to\I}\mbox{dist}(\Theta\cdot r_n,\Z^d)=0$; and
\item $r_n^{(j)}=1$ mod $m_j$ for all $j=1,\ldots,d$ and for all $n\in\N$.
\ee

As a first step, we show that $h_\Theta\left(m\Z^d+(1,\ldots,1)\right)$ is dense in $\T^d$. Since
$h_\Theta(m\Z^d)$ equals $h_{m\Theta}(\Z^d)$ and $m\Theta$ is also nondegenerate, \autoref{lma:hThetaDense}
implies that this set is dense in $\T^d$. Finally, $h_\Theta\left(m\Z^d+(1,\ldots,1)\right)=h_\Theta(m\Z^d)h_\Theta(1,\ldots,1)$ is
just a translate of $h_\Theta(m\Z^d)$, so it is also dense.

Let $(r_n)_{n\in\N}$ be a sequence in $m\Z^d+(1,\ldots,1)$ such that
$\lim\limits_{n\to\I}h_\Theta(r_n)=(1,\ldots,1)$. It is immediate that $(r_n)_{n\in\N}$ satisfies
conditions (1) and (2) above, and the claim is proved.

Let $(r_n)_{n\in\N}$ be a sequence as in the claim above. Set
\[
v_n=u_1^{r_n^{(1)}}\cdots u_d^{r_n^{(d)}},
\]
which is a unitary in $A_\Theta$. For $1\leq j\leq d$, it is easy to check that
\[v_nu_j= e^{2\pi i \sum\limits_{k=1}^d \theta_{jk}m_k}u_jv_n.\]
Since $\sum\limits_{k=1}^d \theta_{jk}m_k=(\Theta\cdot r_n)^{(j)}$, condition (1) above implies that
$(v_n)_{n\in\N}$ is a central sequence of unitaries in $A_\Theta$.

Denote by $\gamma\in\Aut(C(\T))$ the automorphism induced by rotation by the angle $e^{2\pi i (\frac{1}{m_1}+\cdots+\frac{1}{m_d})}$.
It is clear that $\gamma^m=\id$, so it determines an action of $\Z_m$ on $C(\T)$, which we also denote by $\gamma$. Using that
the integers $m_1,\ldots,m_d$ are coprime, it is easy to verify that $\gamma$ is free.

In the following computation, we use condition (2) at the second step to get
\begin{align*}
\alpha(v_n)&=\left(e^{2\pi i r_n^{(1)}/m_1}u_1^{r_n^{(1)}}\right)\cdots \left(e^{2\pi i r_n^{(d)}/m_d}u_d^{r_n^{(d)}}\right)\\
&=\left(e^{2\pi i /m_1}u_1^{r_n^{(1)}}\right)\cdots \left(e^{2\pi i /m_d}u_d^{r_n^{(d)}}\right)\\
&=e^{2\pi i \left(\frac{1}{m_1}+\cdots+\frac{1}{m_d}\right)}v_n.
\end{align*}

We conclude that $(v_n)_{n\in\N}$ determines a $\Z_m$-equivariant homomorphism
\[(C(\T),\gamma)\to (A_\I\cap A',\alpha_\infty).\]
Since $\gamma$ is free, Lemma~1.9 in~\cite{HirPhi_rokhlin_2015} (see also \autoref{thm:XRpRdim}) implies that $\cdimRok(\alpha)\leq 1$.

To show that $\cdimRok(\alpha)=1$, we must argue that $\alpha$ does not have the Rokhlin property. In fact,
we show that no nontrivial finite group can act on $A_\Theta$ with the Rokhlin property. To this end,
suppose that $G$ is a finite group and $\beta\colon G\to\Aut(A_\Theta)$ is an action with the Rokhlin property.
Find projections $e_g$ in $A_\Theta$, for $g\in G$, satisfying
\[\|\beta_g(e_h)-e_{gh}\|<1 \ \ \mbox{ and } \ \ \sum_{g\in G}e_g=1\]
for all $g,h\in G$. In particular, $\beta_g(e_h)$ is Murray-von Neumann equivalent to $e_{gh}$.
Observe that, by Lemma~3.1 in~\cite{Sla_factor_1972}, $A_\Theta$ has a unique tracial state $\tau$, which
must therefore be invariant under $\beta$. Thus $\tau(e_g)=\tau(e_h)$ for all $g,h\in G$.
We denote by $\tau_* \colon K_0(A_\Theta)\to\R$ the group homomorphism induced by $\tau$.
We deduce that $\tau([1])$ is divisible
by the cardinality of $G$ in $\tau(K_0(A_\Theta))$. However, the range of $\tau$ on projections has been computed by Elliott
in \cite{Ell_ktheory_1984}, and his computation yields that the class of the unit of $A_\Theta$ is not divisible in $K_0(A_\Theta)$.
This forces the cardinality of $G$ to be one, and the assertion follows.

For the last claim, observe that $A_\Theta$ is a simple A$\T$-algebra by the main result of \cite{Phi_every_2006}, and hence it has
tracial rank zero. Since $A_\Theta$ has a unique trace, the result is a consequence of \autoref{cor: RdimTRp}.
\end{proof}

It is not in general true that a finite group action with the tracial Rokhlin property has finite Rokhlin dimension with
commuting towers, since there are $K$-theoretic obstructions to admitting such an action.
For example, it is easy to show that the order two automorphism
$\bigotimes\limits_{n\in\N}\Ad(\diag(1,1,-1))$ of the
UHF-algebra of type $3^\I$, determines a $\Z_2$-action $\alpha\colon\Z_2\to\Aut(M_{3^\I})$ with the tracial Rokhlin property;
see \cite{Phi_tracial_2011}.
It is clear that $\alpha$ does not have the Rokhlin property since the unit of $M_{3^\I}$ is not 2-divisible in $K$-theory.
It then follows from Theorem~4.19 in \cite{Gar_rokhlin_2017} that $\alpha$ does not have finite Rokhlin dimension
with commuting towers. More generally, it is a consequence of part (2) in Corollary~4.8 in \cite{HirPhi_rokhlin_2015}
that there are
\emph{no} $\Z_2$-actions on $M_{3^\I}$ with finite Rokhlin dimension with commuting towers.

There are also more subtle obstructions, related to the equivariant $K$-theory of an action (as opposed to the $K$-theory
of the algebra), as we show in the next example. We refer the reader to Section~2 of~\cite{Phi_equivariant_1987} for the
definitions of equivariant $K$-theory and the augmentation ideal $I_G\subseteq R(G)$ of a compact group $G$.
(A quicker introduction containing the notions that are needed here, is given in Section~III.3 of~\cite{Gar_thesis_2015}.)

\begin{eg}\label{eg: BlackadarExNotXRp} In \cite{Bla_symmetries_1990}, Blackadar constructed an example of a $\Z_2$-action on the UHF-algebra $A$ of type
$2^\I$, whose crossed product is not AF. Denote this action by $\alpha\colon \Z_2\to\Aut(A)$.
Phillips showed in Proposition~3.4 of~\cite{Phi_finite_2015} that $\alpha$ has the tracial Rokhlin property.
Since $A\rtimes_\alpha\Z_2$ is not AF, the action $\alpha$ does not have the Rokhlin property, because this would otherwise
contradict Theorem~2.2 in \cite{Phi_tracial_2011}.
We claim that $\alpha$ has infinite Rokhlin dimension with commuting towers.

In order to show that $\cdimRok(\alpha)=\I$, we will show that $\alpha$ does not have discrete $K$-theory. Once we show this,
the result will then follow from Corollary~4.2 in \cite{HirPhi_rokhlin_2015}. In order to arrive at a contradiction, assume that
there exists $n\in\N$ such that $I_{\Z_2}^n\cdot K_\ast^{\Z_2}(A,\alpha)=0$. Denote by $\widehat{\alpha}\colon\Z_2\to\Aut(A\rtimes_\alpha\Z_2)$
the dual action of $\alpha$, and, by a slight abuse of notation, denote also by $\widehat{\alpha}\in \Aut(A\rtimes_\alpha \Z_2)$
the generating order two automorphism. It is immediate to check that the condition $I_{\Z_2}^n\cdot K_\ast^{\Z_2}(A,\alpha)=0$ is equivalent to
\[ \left(\id_{K_\ast(A\rtimes_\alpha G)}-K_\ast(\widehat{\alpha})\right)^n=0.\]

The proof of Proposition~3.5 in \cite{Phi_finite_2015}
shows that there exists $x\in K_1(A\rtimes_\alpha\Z_2)$ with $x\neq 0$ such that $K_1(\widehat{\alpha})(x)=-x$.
It is shown in Proposition~5.4.1 in \cite{Bla_symmetries_1990} that $A\rtimes_\alpha \Z_2$ is isomorphic to the tensor product of the
Bunce-Deddens algebra of type $2^\I$ with $A$, so in particular $K_1(A\rtimes_\alpha\Z_2)$ does not have any 2-torsion.
We conclude that
\[\left(\id_{K_\ast(A\rtimes_\alpha \Z_2)}-K_\ast(\widehat{\alpha})\right)^m(x)=2^mx\neq 0\]
for all $m\in\N$, contradicting the fact that $\alpha$ has discrete $K$-theory. This contradiction, together with
Corollary~4.2 in \cite{HirPhi_rokhlin_2015}, shows that $\alpha$ does not have finite Rokhlin dimension with commuting towers.
\end{eg}

Next, we show that \autoref{thm: RdimwTRp} and \autoref{cor: RdimTRp} hold in greater generality for finite group
actions on Kirchberg algebras. We emphasize that there are no UCT assumptions in the next result.

\begin{thm}\label{thm:outertRp}
Let $G$ be a finite group, let $A$ be a unital Kirchberg algebra, and let $\alpha\colon G\to\Aut(A)$ be an action. Then
the following are equivalent:
\be
\item $\alpha$ has the tracial Rokhlin property;
\item $\alpha$ has the weak tracial Rokhlin property;
\item $\dimRok(\alpha)\leq 1$;
\item $\alpha$ is pointwise outer (that is, $\alpha_g$ is not inner for all $g\in G\setminus\{1\}$).
\ee
\end{thm}
\begin{proof}
It is clear that (1) implies (2). That (2) implies (4) is well known; see for example
Lemma~VI.8 in~\cite{Arc_crossed_2008} or Proposition~5.3 in~\cite{HirOro_tracially_2013}.
(These implications hold without restrictions on the
algebra $A$.) In turn, the equivalence between (3) and (4) was proved in
Theorem~4.20 of~\cite{Gar_rokhlin_2017}. It remains to show that (3) implies (1).

Assume that $\alpha$ is pointwise outer. By Theorem~5.1 in~\cite{GolIzu_quasifree_2011}, there are a pointwise outer action $\gamma\colon G\to\Aut(\OI)$ and
an equivariant isomorphism
\[(A,\alpha)\cong \left(A\otimes\bigotimes\limits_{n=1}^\I \OI,\alpha\otimes\bigotimes\limits_{n=1}^\I \gamma\right).\]
(In fact, any outer quasifree action $\gamma$ will do.) We identify these two $G$-algebras in the sequel. Also, for $k\in\N$, we
regard $A\otimes\bigotimes_{n=1}^k \OI$ canonically as a subalgebra of $A\otimes\bigotimes_{n=1}^\I \OI$.

Let $\ep>0$, let $F\subseteq A\otimes\bigotimes\limits_{n=1}^\I \OI$ be a finite set, and let $x\in A\otimes\bigotimes\limits_{n=1}^\I \OI$
be a nonzero positive contraction.
Find $m\in\N$, a finite set $\widetilde{F}\subseteq A\otimes \bigotimes\limits_{n=1}^m\OI$ and a positive contraction $y\in
A\otimes \bigotimes\limits_{n=1}^m\OI$
such that:
\bi\item for every $a\in F$ there exists $b\in \widetilde{F}$ with $\|a-b\|<\ep$, and
\item $\|x-y\|<\ep$.\ei

Since $\gamma_g$ is outer for all $g\in G\setminus\{1\}$, Lemma~1.9 in~\cite{Kis_outer_1981} implies that there exists a positive contraction $f\in \OI$
with $\|f\gamma_g(f)\|<\ep$ for all $g\in G\setminus\{1\}$. Since $\OI$ has real rank zero, we may assume that $f$ has finite spectrum,
and that $1\in \spec(f)$. Let $q\in \OI$ be the spectral projection corresponding to $1\in \spec(f)$. Then $qfq=q$ and thus
\[\|q\gamma_g(q)\|=\|fqf\gamma_g(f)\gamma_g(q)\gamma_g(f)\|<\ep\]
for all $g\in G\setminus\{1\}$. By slightly perturbing $q$, we may assume without loss of generality that $q\gamma_g(q)=0$
It follows that $\|\gamma_h(q)\gamma_g(q)\|<\ep$ for all $g,h\in G$ with $g\neq h$.

Given $g\in G$, set
\[e_g=1_A\otimes \underbrace{1_{\OI}\otimes\cdots\otimes 1_{\OI}}_\text{$m$ times}\otimes \gamma_g(q)\in A\otimes\bigotimes\limits_{n=1}^{m+1} \OI\subseteq A\otimes\bigotimes\limits_{n=1}^\I \OI.\]
We claim that the projections $e_g$, for $g\in G$, satisfy the conditions in part~(2) of~\autoref{df: tRp wtRp} for $\ep$,
$F$ and $x$. It is clear that $\alpha_g(e_h)=e_{gh}$ for all $g,h\in G$, so the first condition
is satisfied. Clearly $\|e_ge_h\|<\ep$ for all $g,h\in G$, so condition (b) follows.
Given $a\in F$, choose $b\in \widetilde{F}$ satisfying $\|a-b\|<\ep$. For $g\in G$, we have
\[\|e_ga-ae_g\|\leq \|a-b\|+\|e_gb-be_g\|<\ep,\]
because $e_g$ commutes with $b$, so condition (c) is also satisfied.
Condition (d) is automatic since any two nonzero positive elements in a purely infinite simple \ca\ are
Cuntz equivalent. Finally, to check condition (e), recall that $y\in A\otimes \bigotimes\limits_{n=1}^m\OI$ satisfies $\|x-y\|<\ep$. Set $e=\sum\limits_{g\in G}e_g$. Then
$e$ is a nonzero projection in $\OI$, and $eye=ey$ can be represented as
\[\sum\limits_{g\in G}y\otimes q_g\in A\otimes\bigotimes\limits_{n=1}^{m} \OI\otimes \OI\subseteq A.\]
We deduce that
\[\|exe\|\geq \left|\|eye\|-\|x-y\|\right|>\|ey\|-\ep=\|e\|\|y\|-\ep=1-\ep,\]
as desired. This shows that $\alpha$ has the tracial Rokhlin property, and finishes the proof.
\end{proof}

\section{Structure of the crossed product}

In this section, we study crossed products by actions of compact groups with finite
Rokhlin dimension with commuting towers. The main result of this section, \autoref{thm:preservationCP},
shows that a number of structural properties are inherited under formation of crossed products by
such actions. We also point out, in \autoref{rmk:nonCommTws}, that the formulation of Rokhlin dimension
without commuting towers is not enough to obtain most of these conclusions. Strict comparison is also inherited,
as long as the action has Rokhlin dimension with commuting towers at most one, and the $C^*$-algebra has ``no $K_1$-obstructions'';
see \autoref{thm:StrCompRdim1} for the precise formulation.
Our results are used to construct an action $\alpha$ of $\Z_2$ on a UCT Kirchberg algebra with
$\cdimRok(\alpha)=2$ and $\dimRok(\alpha)=1$; see \autoref{eg:cdimRokdimRokDifferent}. To our best knowledge, this is
the first example of a group action on a simple \ca\ with Rokhlin dimension other than 0, 1 or $\infty$.

At the core of these results is the fact that if $\alpha\colon G\to\Aut(A)$ has finite Rokhlin
dimension with commuting towers, then $A\rtimes_\alpha G$ can be locally approximated by a certain
continuous $C(X)$-algebra with fibers isomorphic to $A\otimes\K(L^2(G))$; see \autoref{prop:ApproxCP}.
Moreover, the space $X$ can be chosen to satisfy $\dim(X)<\I$ whenever $\dim(G)<\I$; see \autoref{thm:XRpRdim}.


\vspace{0.3cm}

The following is well known, but we include a proof for the convenience of the reader.

\begin{prop}\label{prop:CGAconjugate}
Let $A$ be a \ca, let $G$ be a compact group, and let $\alpha\colon G\to\Aut(A)$ be a continuous action.
Then there is an equivariant isomorphism
\[\theta\colon (C(G,A),\texttt{Lt}\otimes\alpha)\to (C(G,A),\texttt{Lt}\otimes\id_A)\]
given by $\theta(f)(g)=\alpha_{g^{-1}}(f(g))$ for all $f\in C(G,A)$ and all $g\in G$.

In particular, there is a natural identification
\[C(G,A)\rtimes_{\texttt{Lt}\otimes\alpha}G\cong A\otimes\K(L^2(G)).\]\end{prop}
\begin{proof}
It is clear that $\theta$ is an isomorphism. To check that it is equivariant, let
$f\in C(G,A)$ and $g,h\in G$ be given. Then
\begin{align*} \theta\left((\texttt{Lt}_g\otimes\alpha_g)(f)\right)(h)&=\alpha_{h^{-1}}\left((\texttt{Lt}_g\otimes\alpha_g)(f)(h)\right)\\
&=\alpha_{h^{-1}}\left(\alpha_g(f(g^{-1}h))\right)\\
&=\theta(\xi)(g^{-1}h)\\
&= (\texttt{Lt}_g\otimes\id_A)(f)(h).
\end{align*}

The second claim is immediate since there are natural isomorphisms
\[(C(G)\otimes A)\rtimes_{\texttt{Lt}\otimes\alpha}G\cong (C(G)\rtimes_{\texttt{Lt}}G)\otimes A\cong  \K(L^2(G))\otimes A.\]
\end{proof}

With the notation from the above proposition, it follows that the canonical equivariant inclusion
$A\to C(G)\otimes A$ induces an injective homomorphism
\[\iota\colon A\rtimes_\alpha G\to A\otimes \K(L^2(G)).\]

Denote by $\lambda\colon G\to \U(L^2(G))$ the left regular representation, and identify
$A\rtimes_\alpha G$ with its image under $\iota$. It is then a consequence of noncommutative
duality for crossed products that
\[A\rtimes_\alpha G= \left(A\otimes\K(L^2(G))\right)^{\alpha\otimes\Ad(\lambda)}.\]

The following definitions are standard, and have been independently defined
by Kasparov and Phillips.
For a \ca\ $A$, we denote its center by $Z(A)$.

\begin{df}\label{df:CX-alg}
Let $X$ be a locally compact Hausdorff space.

\be
\item A \emph{$C_0(X)$-algebra} is a pair $(A,\mu)$ consisting of a \ca\ $A$ and
a homomorphism $\mu\colon C_0(X)\to Z(M(A))$ satisfying $\mu(C_0(X))A=A$.
We usually suppress $\mu$ from the notation, and simply say that $A$ is a $C_0(X)$-algebra.
\ee

If $A$ is a $C_0(X)$-algebra and $U\subseteq X$ is open, then $\mu(C_0(U))A$
is an ideal in $A$. Given $x\in X$, the \emph{fiber over $x$} is the quotient
\[A(x)=A/C_0(X\setminus\{x\})A.\]
(In principle, the fiber $A(x)$ depends on $\mu$, but we do not incorporate it into
the notation.) For $a\in A$, we write $a(x)$ for its image in $A(x)$.

\be\item[(2)] We say that $A$ is a \emph{continuous $C_0(X)$-algebra} if it is a $C_0(X)$-algebra
and for every $a\in A$, the map $X\to \R$, given by $x\mapsto \|a(x)\|$, is continuous.
\item[(3)] If $A$ is a continuous $C_0(X)$-algebra, we say that it is \emph{locally trivial} if for
every $x\in X$, there exist an open set $U\subseteq X$ containing $x$ and an isomorphism
$\mu(C_0(U))A \cong C_0(U,A(x))$. (In particular, $A(x)\cong A(y)$ for every $y\in U$.)
\ee
\end{df}

Let $G\curvearrowright X$ be a continuous action of a group $G$ on a space $X$. We write
$\pi\colon X\to X/G$ for the canonical quotient map. Given $x\in X$,
we denote by $G\cdot x$ the set $\{g\cdot x\colon g\in G\}$.
The \emph{stabilizer} of $x\in X$ is the closed subgroup $G_x=\{g\in G\colon g\cdot x=x\}$, and there
is a canonical homeomorphism $G\cdot x\cong G/G_x$.

\begin{prop}\label{prop:CXX/Galg}
Let $G$ be a compact group, let $X$ be a locally compact Hausdorff space, and let $G$ act continuously on $X$.
Then $C_0(X)$ is naturally a continuous $C_0(X/G)$-algebra.
\end{prop}
\begin{proof}
We denote by $\pi\colon X\to X/G$ the quotient map. Observe that, since $G$ is compact, $C_0(X/G)$ can be identified
with the fixed point algebra of $C_0(X)$ via the map $\pi^*\colon C_0(X/G)\to C_0(X)$. Hence $C_0(X)$ is a $C_0(X/G)$-algebra,
and we only need to show continuity of the bundle. Let $U\subseteq X/G$ be an open subset, and set $V=\pi^{-1}(U)$. Using the
notation from \autoref{df:CX-alg}, there is a natural identification of the ideal $C_0(X)C_0(U)$ with $C_0(V)$. In particular,
for $x\in X$, the fiber of $C_0(X)$ over $\pi(x)$ is naturally identified with $C(G\cdot x)$. Thus, in order to check continuity,
we must show that for all $f\in C_0(X)$, the assignment $\pi(x)\mapsto \sup\limits_{g\in G}|f(g\cdot x)|$ is a continuous map $X/G\to\R$.

Let $f\in C_0(X)$, let $\ep>0$ and let $x\in X$. Use continuity of $f$ to find, for every $g\in G$, an open neighborhood
$W_g$ of $g\cdot x$ such that $|f(g\cdot x)-f(y)|<\ep$ for all $y\in W_g$. By compactness of $G$, there exists an
open neighborhood $W$ of $x$ such that $g\cdot W\subseteq W_g$ for all $g\in G$. Set $U=\pi(W)$, which is an open neighborhood
of $\pi(x)$ in $X/G$. Let $y\in W$. Then
\[\sup_{h\in G} |f(h\cdot y)|-\ep \leq \sup_{g\in G} |f(g\cdot x)|\leq \sup_{h\in G} |f(h\cdot y)|+\ep.\]

We conclude that $\sup\limits_{\pi(y)\in U}\left|\sup\limits_{h\in G} |f(h\cdot y)|-\sup\limits_{g\in G} |f(g\cdot x)| \right|<\ep$, and hence $C_0(X)$ is a continuous $C_0(X/G)$-algebra.
\end{proof}

\begin{rem} Note that if $X$ is a compact Hausdorff space and $Y$ is some quotient of $X$, then
 $C(X)$ is not necessarily a continuous $C(Y)$-algebra.
\end{rem}

We will need the following fact about compact group actions on algebras of the form $C_0(X,A)$. We
will only need it when $X$ is compact and the action is free, but the proof in the
general case does not take more work.

\begin{prop}\label{thm:CXGalgebra}
Let $G$ be a compact group, let $X$ be a locally compact space, let $A$ be a \ca, let $G\curvearrowright X$
be a continuous action, and let $\alpha\colon G\to \Aut(A)$ be another action.
Endow $C_0(X,A)$ with the diagonal $G$-action $\gamma$. Then
$C_0(X,A)\rtimes_\gamma G$ is a continuous $C_0(X/G)$-algebra. Moreover, with $\pi\colon X\to X/G$ denoting the
quotient map, the fiber over $\pi(x) \in X/G$ is canonically isomorphic to
\[(A\rtimes_\alpha G_x)\otimes\K(L^2(G/G_x)).\]
\end{prop}
\begin{proof}
Observe that $X/G$ is Hausdorff because $X$ is Hausdorff and $G$ is compact. Also, for $x,y\in X$,
the algebras $(A\rtimes_\alpha G_x)\otimes\K(L^2(G/G_x))$ and $(A\rtimes_\alpha G_y)\otimes\K(L^2(G/G_y))$
are isomorphic whenever $\pi(x)=\pi(y)$.

Regard $C_0(X)$ as a continuous
$C_0(X/G)$-algebra as in \autoref{prop:CXX/Galg}. By Theorem~B in~\cite{KirWas_operations_1995}, the tensor product
$C_0(X,A)$ is a continuous $C_0(X/G)$-algebra, and the fiber of $C_0(X,A)$ over $\pi(x)\in X/G$ can be identified with
$C(G\cdot x,A)\cong C(G/G_x,A)$.
Moreover, the action $\gamma$ is a fiber-wise action, in the sense of page~194 of~\cite{KirWas_exact_1999}, and the
induced action on the fiber $C(G/G_x,A)$ over $\pi(x)$ is $\texttt{Lt}\otimes\alpha$. The discussion there shows that $C_0(X,A)\rtimes_\gamma G$ is a
$C_0(X/G)$-algebra with fibers isomorphic to $C(G/G_x,A)\rtimes_{\texttt{Lt}\otimes\alpha}G$. By Green's Imprimitivity Theorem,
the fiber over $\pi(x)$ is thus isomorphic to $(A\rtimes_\alpha G_x)\otimes\K(L^2(G/G_x))$.
Finally, continuity of this $C(X/G)$-algebra follows from Theorem~4.1 in~\cite{KirWas_exact_1999}, since $G$ is amenable.
\end{proof}

We now specialize to the case of free actions, and show that when local cross sections for the action exist, then the $C^*$-bundle from
the theorem above is locally trivial. Recall that for an action $\alpha\colon G\to\Aut(A)$ of
a compact group, $\iota\colon A\rtimes_\alpha G\to A\otimes\K(L^2(G))$ denotes the canonical inclusion.

\begin{cor}\label{cor:freeCXGalgebra}
Let $G$ be a compact group, let $X$ be a compact Hausdorff space, let $A$ be a \ca, let $G\curvearrowright X$
be a continuous free action, and let $\alpha\colon G\to \Aut(A)$ be another action.
Endow $C(X,A)$ with the diagonal $G$-action $\gamma$.
Then $C(X,A)\rtimes_\gamma G$ is a continuous $C(X/G)$-algebra with fibers canonically isomorphic to
$A\otimes\K(L^2(G))$.

Let $\theta\colon A\rtimes_\alpha G\to C(X,A)\rtimes_\gamma G$ be the map induced by the canonical
inclusion $A\to C(X,A)$.
For $z\in X/G$, denote by
\[\pi_z\colon C(X,A)\rtimes_\gamma G\to (C(X,A)\rtimes_\gamma G)_z\cong A\otimes\K(L^2(G))\]
the canonical quotient map onto the fiber over $z$. Then the following diagram is commutative
\beqa\xymatrix{
A\rtimes_\alpha G\ar[r]^-{\theta}\ar[dr]_-{\iota} & C(X,A)\rtimes_\gamma G\ar[d]^-{\pi_z}\\
&  A\otimes\K(L^2(G)).}\eeqa

Finally, if there exist local cross sections for the canonical quotient map $\pi\colon X\to X/G$
(for example, if $G$ is a Lie group; see Theorem~8 in \cite{Mos_sections_1956}),
then $C(X,A)\rtimes_\gamma G$ is a locally trivial bundle over~$X/G$.
\end{cor}
\begin{proof}
The first assertion is a consequence of \autoref{thm:CXGalgebra}, since $G_x=\{1\}$ for all $x\in X$ because the action is free.
(Also, in this case, the computation of the fibers can be performed using \autoref{prop:CGAconjugate}.)

In order to show that the diagram in the statement is commutative, observe first
that the map $\iota$ is induced by the map $\widetilde\iota\colon C(G,A)\to C(G\times X,A)$ given
by $\widetilde{\iota}(f)(g,x)=f(g)$ for $f\in C(G,A)$, for $g\in G$ and for $x\in X$. It is then easy to see
that for $z\in X/G$, the continuous function $\pi_z(\iota(f))\colon G\to C(G,A)$ is given by
\[\pi_z(\iota(f))(g)(h)=f(g)\]
for $g,h\in G$. It follows that the restrictions of $\pi_z\circ \iota$ and $\theta$ to $C(G,A)$ agree.
By density in the crossed product, we deduce that $\pi_z\circ\iota=\theta$, as desired.

We prove the last statement about local triviality.
Let $x\in X$. Using the existence of local cross sections, find an open set
$U\subseteq X/G$ containing $G\cdot x$, and a local section $s\colon U\to X$. In particular, the open
subset
\[\{g\cdot s(U)\colon g\in G\}\subseteq X\]
is equivariantly homeomorphic to $U\times G$ with the trivial action on $U$ and translation on $G$.
The crossed product of this invariant ideal, which is the ideal of the bundle associated to the open
set $U$, is isomorphic to $C_0(U\times G ,A)\rtimes_{\id\otimes\texttt{Lt}\otimes\alpha} G$, which is itself
isomorphic to $C_0(U)\otimes A\otimes\K(L^2(G))$ by \autoref{prop:CGAconjugate}. The claim follows.
%
\end{proof}

The following is a slight variant of Theorem~4.4 of~\cite{Gar_rokhlin_2017}, extended to non-Lie groups.
(The case of unital $A$ and finite $G$ is implicit in Lemma~1.9 of~\cite{HirPhi_rokhlin_2015}.)
The main difference between the theorem below and Theorem~4.4 of~\cite{Gar_rokhlin_2017}, is that the free
$G$-space is independent of the algebra and the action, and the cost is that the map $\varphi$ is not
necessarily injective. We denote by $\dim(G)$ the covering dimension of the group $G$.

\begin{thm}\label{thm:XRpRdim}
Let $G$ be a compact group. For every $d\in\N$, there exists
a free $G$-space $X$ with local cross sections with the following property.
Given an action $\alpha\colon G\to\Aut(A)$ of $G$ on a \ca\ $A$, we have $\cdimRok(\alpha)\leq d$
\ifo for every $\sigma$-unital $\alpha$-invariant subalgebra $D\subseteq A$ there exists an equivariant unital homomorphism
\[\varphi\colon C(X)\to F_\alpha(D,A).\]
Moreover, we have the following relations between the covering dimension of $X$ and the Rokhlin dimension of $\alpha$:
\begin{align*} \dim(X)&\leq (\cdimRok(\alpha)+1)(\dim(G)+1)-1 \\
\cdimRok(\alpha)&\leq \dim(X)-\dim(G).
\end{align*}

In particular, we can always take the space $X$ above to be finite dimensional when $G$ is finite dimensional.
\end{thm}
\begin{proof}
For the ``only if'' implication, set $d=\cdimRok(\alpha)<\infty$.
An inspection of the proof of Lemma~4.3 of~\cite{Gar_rokhlin_2017} shows that the space
$X$ can be chosen to be the $d$-join $G\ast \cdots\ast G$ with
diagonal action (see definition in~\cite{Mil_construction_1956}). This action is free, and has local cross sections regardless of whether $G$ is Lie or not; see
Section~3 in~\cite{Mil_construction_1956}. For the converse implication, an inspection of the proof of Theorem~4.4 of~\cite{Gar_rokhlin_2017}
reveals that existence of local cross sections is all that is needed to obtain $\cdimRok(\alpha)<\infty$.
\end{proof}

We use the above result to show that compact group actions
with finite Rokhlin dimension with commuting towers preserve a number of properties upon
forming crossed products; see \autoref{thm:preservationCP}. The way it is used is
through \autoref{df:criterionCP}. We need an auxuliary notion first.

\begin{df}\label{df:posexistemb}
Let $A$ and $B$ be \ca s, and let $\iota\colon A\to B$ be an injective homomorphism. We say
that $\iota$ is a \emph{positively existential embedding} if for all separable subalgebras
$A_0\subseteq A$ and $B_0\subseteq B$ with $\iota(A_0)\subseteq B_0$, there exists a
homomorphism $\varphi\colon B_0\to A_\I$ making the following diagram commute:
\beqa
\xymatrix{
A\ar@{^{(}->}[rr]\ar[dr]_-{\iota} && A_\infty,\\
A_0\ar@{^{(}->}[u] \ar[dr]_-{\iota} &B & \\
& B_0\ar@{^{(}->}[u]\ar[uur]_-\varphi &
}
\eeqa
where all the inclusions are the canonical ones.
\end{df}

When $A$ and $B$ are separable, this notion agrees with that of a 
\emph{sequential splitting}; see~\cite{BarSza_sequentially_2016}. In general, however,
being a positively existential embedding is a weaker condition; 
see Subsection~4.3 in~\cite{GarLup_equivariant_2016}. This terminology agrees with the 
usual one in model theory for $C^*$-algebras; see, for example, 
Definition~2.11 in~\cite{GarLup_equivariant_2016} and the comments after it, although we do not use any results from model theory for $C^*$-algebras in this paper.


\begin{rem} It is clear that if $\iota\colon A\to B$ is a positively existential embedding, then so are all its amplifications
 $\iota\otimes\id_{M_n}\colon M_n(A)\to M_n(B)$, for $n\in\N$, as well as its unitization $\widetilde{\iota}\colon\widetilde{A}\to\widetilde{B}$.
\end{rem}

\begin{rem}\label{rem:subalgCP}
We need two easy observations about separable subalgebras.
\bi\item Let $X$ be a compact space and let $A$ be a \ca. Then any separable subalgebra of $C(X,A)$
is contained in an algebra of the form $C(X,D)$ for some separable subalgebra $D\subseteq A$.
\item Let $E$ be a \ca, let $G$ be a second countable locally compact group, and let $\beta\colon G \to \Aut(E)$ be an action. If $B_0\subseteq E\rtimes_\beta G$ is a separable subalgebra, then
there exists a $\beta$-invariant separable subalgebra $E_0\subseteq E$ such that $B_0\subseteq E_0\rtimes_\beta G$. For example, find a countably family
$\{f_n\colon n\in\N\}$ of elements in $C_c(G,E,\beta)$ whose closure in $E\rtimes_\beta G$ contains $B_0$. Since the image of each $f_n$
is a compact subset of $E$ and $G$ is second countable, the $\beta$-invariant subalgebra $E_0$ of $E$ generated by $\{f_n(G)\colon n\in\N\}$ is separable, and clearly
$B_0\subseteq E_0\rtimes_\beta G$.
\ei
\end{rem}

The combination of \autoref{thm:XRpRdim} with the following proposition constitutes our main technical tool in the study
of crossed product by actions with finite Rokhlin dimension with commuting towers.

\begin{prop} \label{prop:ApproxCP}
Let $A$ be a \ca, let $G$ be a second countable compact group,
let $\alpha\colon G\to\Aut(A)$ be an
action, and let $X$ be a compact free $G$-space. Suppose that for every $\sigma$-unital $\alpha$-invariant
subalgebra $D\subseteq A$,
there exists a unital equivariant homomorphism $\psi\colon C(X)\to F_{\alpha}(D,A)$.
Denote by $\iota\colon A\rtimes_\alpha G\to C(X,A)\rtimes G$ the canonical embedding
induced by the equivariant inclusion $A\hookrightarrow C(X,A)$, where $C(X,A)$ carries the diagonal $G$-action.
Then $\iota$ is a positively existential embedding.
\end{prop}
\begin{proof}
Let $B_0\subseteq C(X,A)\rtimes G$ and $A_0\subseteq A\rtimes G$ be separable subalgebras satisfying
$\iota(A_0)\subseteq B_0$. Use \autoref{rem:subalgCP} to find an $\alpha$-invariant separable subalgebra
$D\subseteq A$ satisfying $B_0\subseteq C(X,D)\rtimes G$ and $A_0 \subseteq D\rtimes G$.
Let $\psi\colon C(X)\to F_{\alpha}(D,A)$ be a unital equivariant homomorphism
as in the statement. Upon tensoring with $\id_D$ and using Lemma~2.3 in~\cite{Gar_crossed_2017}, we obtain an equivariant
homomorphism
\[\theta \colon C(X)\otimes D \to A_{\I,\alpha},\]
which is the identity map on $D$. We obtain a unital homomorphism
\[\varphi\colon C(X,D)\rtimes G \to A_{\I,\alpha}\rtimes_{\alpha_\I}G\hookrightarrow (A\rtimes_\alpha G)_\I\]
which makes the following diagram commute:
\beqa
\xymatrix{
A\rtimes_\alpha G\ar[rrr]\ar[dr]_-{\iota} &&& (A\rtimes_\alpha G)_\infty.\\
D\rtimes_\alpha G\ar@{^{(}->}[u] \ar[dr]_-{\iota} &C(X,A)\rtimes G && \\
A_0  \ar[dr]_-{\iota} \ar@{^{(}->}[u]& C(X,D)\rtimes G\ar@{^{(}->}[u]\ar[uurr]_-{\varphi}&& \\
& B_0 \ar@{^{(}->}[u]&
}
\eeqa
This finishes the proof. \end{proof}

\begin{df}\label{df:criterionCP}
Let $\mathcal{C}$ be a class of \ca s. Consider the following permanence properties:
\be\item[(P1)] \emph{Passage to section algebras:} If $Y$ is a finite-dimensional compact
metrizable space and $A$ is a continuous
$C(Y)$-algebra with $A_y\in\mathcal{C}$ for all $y\in Y$, then $A\in \mathcal{C}$.
\item[(P2)] \emph{Passage to positively existential subalgebras:}
If $A$ is a \ca\ such that there exist another \ca\ $B\in\mathcal{C}$, and a positively existential embedding
$A\hookrightarrow B$, then $A\in\mathcal{C}$.\ee
\end{df}

Recall (\cite{TomWin_strongly_2007}) that a unital, infinite-dimensional, separable $C^*$-algebra $\mathcal{D}$
is said to be \emph{strongly self-absorbing} if there exists an isomorphism 
$\mathcal{D}\cong \mathcal{D}\otimes_{\mathrm{min}}\mathcal{D}$ that is approximately unitarily
equivalent to the first factor embedding $d\mapsto d\otimes 1_{\mathcal{D}}$. A \ca\ $A$ is said to be 
\emph{$\mathcal{D}$-stable} if $A\cong A\otimes \mathcal{D}$.





We begin by recording the fact that a number of classes of \ca s are closed under passing to section algebras with finite
dimensional bases (property (P1)). Most of these
are well known: \autoref{thm:P1} just collects the relevant references, and provides a proof where one is needed.
The only original results are (7), (8), and (9). Recall that a \ca\ $A$ is said to be \emph{stable} if
$A\cong A\otimes\K$.

\begin{thm}\label{thm:P1}
Let $Y$ be a compact metrizable space of finite covering dimension, and let $A$ be a $C(Y)$-algebra.
\be
\item Let $\mathcal{D}$ be a strongly self-absorbing \ca. If $A_y$ is separable and $\mathcal{D}$-absorbing
for all $y\in Y$, then the same is true for $A$.
\item If $A_y$ is stable and separable for all $y\in Y$, then so is $A$.
\item We have
\[\dimnuc(A)+1\leq(\dim(Y)+1)\left(\sup\limits_{y\in Y}\dimnuc(A_y)+1\right).\]
\item We have
\[\dr(A)+1\leq(\dim(Y)+1)\left(\sup\limits_{y\in Y}\dr(A_y)+1\right).\]
\item We have
\[\mathrm{sr}(A)\leq \dim(Y)+ \sup_{y\in Y}\mathrm{sr}(A_y).\]
\item We have
\[\mathrm{RR}(A)\leq \dim(Y)+ \sup_{y\in Y}\mathrm{RR}(A_y).\]
\item Let $n\in\N$. If $A_y$ is nuclear, satisfies the UCT, and $K_\ast(A_y)$ is uniquely $n$-divisible for all $y\in Y$, then the same
is true for $A$.
\item If $A_y$ is nuclear, satisfies the UCT, and $K_\ast(A_y)$ is rational for all $y\in Y$, then the same
is true for $A$.
\item If $A_y$ is nuclear, satisfies the UCT, and $K_\ast(A_y)=\{0\}$ for all $y\in Y$, then the same
is true for $A$.
\item If $A_y$ is nuclear, separable and satisfies the UCT for all $y\in Y$, then the same
is true for $A$.
\ee
\end{thm}
\begin{proof}
(1) and (2) are proved in Theorem~4.6 of~\cite{HirRorWin_algebras_2007}.

For (4), this is Lemma~3.1 in~\cite{Car_classification_2011}, while the proof of the said lemma is easily adapted to show (3).
For (5), this is a consequence of Theorem~1.3 in~\cite{Sud_stable_2005} and the fact that $\mathrm{sr}(C(Y,A))\leq \dim(Y)+\mathrm{sr}(A)$
for every compact space $Y$ and every \ca\ $A$, by Theorem~1.13 in~\cite{NagOsaPhi_ranks_2001}.
For (6), the same argument in the proof of Theorem~1.3 in~\cite{Sud_stable_2005} applies to real rank, showing that
\[\RR(A)\leq \sup_{y\in Y}\RR(C(Y,A_y).\]
(In the proof, one has to use the fact that if $I$ is an ideal in a \ca\ $B$, then
$\RR(B/I)\leq \RR(B)$ as a replacement of the inequality (F2) from~\cite{Sud_stable_2005}. This is
Theorem~1.4 in~\cite{Elh_rang_1995}.)
Since $\RR(C(Y,A_y))\leq \dim(Y)+\RR(A_y)$ for all $y\in Y$ by Corollary~1.12 in \cite{NagOsaPhi_ranks_2001},
the result follows.

We prove (9). By passing to separable sub-$C(Y)$-algebras
as in the proof of (2), we may assume that each $A_y$ is separable. Then $A$ is separable and nuclear.
We will prove that $K_\ast(A)=\{0\}$ by showing that
$A$ is $KK_Y$-equivalent to $A\otimes\Ot$. We regard $A\otimes\Ot$ as a $C(Y)$-algebra in the obvious way, so that the
fiber over $y\in Y$ is $A_y\otimes\Ot$. Since $A$ is exact, this is a \emph{continuous} $C(Y)$-algebra, by Theorem~B in~\cite{KirWas_operations_1995}.
Let $\rho\colon A\to A\otimes\Ot$ denote the first tensor factor
embedding. Then $\rho$ is a $C(Y)$-homomorphism whose fiber map $\rho_y\colon A_y\to A_y\otimes\Ot$ is also the first tensor
factor embedding. Since $A_y$ has trivial $K$-theory and satisfies the UCT, we deduce that $\rho_y$ induces a $KK$-equivalence
$A_y\sim_{KK} A_y\otimes\Ot$. Since $Y$ is finite dimensional and $A$ is separable and nuclear, it follows from
Theorem~1.1 in~\cite{Dad_fiberwise_2009} that $\xi=KK(\rho) \in KK_Y(A, A\otimes\Ot)$ is invertible. The result follows.

For (7), one uses an identical argument as for (9), replacing $\Ot$ by $M_{n^\infty}$ and noting that unique
$n$-divisibility of an abelian group is equivalent to absorbing $\Z[1/n]$ tensorially. Also, (8) follows by
applying (7) for all $n\in\N$.

Finally, for (10), this is Theorem~1.4 in~\cite{Dad_fiberwise_2009}.
\end{proof}

The next result lists a number of classes that are closed under positive existential embeddings (property (P2)).
Some of these observations
were made, in the separable case, in~\cite{Gar_crossed_2017}, and later independently in~\cite{BarSza_sequentially_2016}, and are also implicit
in various earlier papers. 
For the classes in (6), (7) and (8), the result is new even in the separable case.

\begin{prop}\label{prop:P2}
Let $A$ and $B$ be \ca s, and let $\iota\colon A\hookrightarrow B$ be a positive existential embedding.
\be
\item Let $\mathcal{D}$ be a strongly self-absorbing \ca. If $B$ is separable and $\mathcal{D}$-absorbing,
then the same holds for $A$.
\item If $B$ is separable and stable, then so is $A$.
\item We have
\[\dimnuc(A)\leq \dimnuc(B).\]
\item We have
\[\dr(A)\leq \dr(B).\]
\item We have
\[\mathrm{sr}(A)\leq \mathrm{sr}(B).\]
\item We have
\[\RR(A)\leq \RR(B).\]
\item The map $K_\ast(\iota)\colon K_\ast(A)\to K_\ast(B)$ is injective.
In particular, if either $K_0(B)$ or $K_1(B)$ has any of the following
properties, then so does $K_0(A)$ or $K_1(A)$: being trivial; being free; being torsion-free.
\item Let $n\in \N$. If either $K_0(B)$ or $K_1(B)$ is uniquely $n$-divisible, then so is
$K_0(A)$ or $K_1(A)$. In particular, if either $K_0(B)$ or $K_1(B)$ is rational, then so is
$K_0(A)$ or $K_1(A)$.
 \ee\end{prop}
\begin{proof}
The proof for (1) is contained in the proofs of Theorem~4.3 of~\cite{Gar_crossed_2017} (see also Remark~2.12 there),
and we briefly sketch the proof. Let $\iota\colon A\to B$ be a positively existential embedding, and let $\varphi\colon B\to A_\infty$
be a left inverse.
Find a sequence
$(\varphi_n)_{n\in\N}$ of contractive linear maps $\varphi_n\colon B\to A$ which are asymptotically $\ast$-multiplicative,
and such that $\lim\limits_{n\to\I}\varphi_n(\iota(a))=a$ for all $a\in A$.

Let $\ep>0$, and let $F\subseteq A$ and $E\subseteq \D$
be finite subsets. Use Theorem~4.2 in~\cite{Gar_crossed_2017} to choose a completely
positive map $\psi\colon D\to B$ satisfying
\be
\item $\|\iota(a)\psi(d) -\psi(d)\iota(a)\| <\ep$ for all $a\in F$ and for all $d\in E$;
\item $\|\psi(de)\iota(a) - \psi(d)\psi(e)\iota(a)\|<\ep$ for every $d, e \in E$ and every $a \in F$;
\item $\|\psi(1)\iota(a) - \iota(a)\| <\ep$ for all $a \in F$.
\ee

Consider the completely positive contractive map $\theta=\varphi\circ\psi \colon \D\to A_\infty$. Using the Choi-Effros Lifting Theorem,
we can find a completely positive contractive lift $\pi\colon D\to A$ satisfying
\be
\item $\|a\pi(d) -\pi(d)a\| <\ep$ for all $a\in F$ and for all $d\in E$;
\item $\|\pi(de)a - \pi(d)\psi(e)a\|<\ep$ for every $d, e \in E$ and every $a \in F$;
\item $\|\pi(1)a - a\| <\ep$ for all $a \in F$.
\ee
Another application of Theorem~4.2 in~\cite{Gar_crossed_2017} implies that $A$ is $\D$-stable, as desired.

We prove (2). By the equivalence between (a) and (c) in Theorem~2.1 of~\cite{HjeRor_stability_1998}, and the equivalence
between (b) and (c) in ~2.1 of Proposition~\cite{HjeRor_stability_1998}, a separable \ca\ $C$ is stable if and only if for 
every positive contraction $c\in C$ and every $\ep>0$ there exists a contraction $x\in C$ with $\|cxx^*\|<\ep$ and $\|c-xx^*\|<\ep$. 
Now let $a\in A$ be a positive contraction and let $\ep>0$. Since $B$ is stable, there exists a contraction $x\in B$ with 
$\|\iota(a)xx^*\|<\ep$ and $\|\iota(a)-xx^*\|<\ep$. Find a sequence
$(\varphi_n)_{n\in\N}$ of contractive linear maps $\varphi_n\colon B\to A$ which are asymptotically $\ast$-multiplicative,
and such that $\lim\limits_{n\to\I}\varphi_n(\iota(a))=a$ for all $a\in A$. Then there exists $n_0\in\N$ such that
$\|a\varphi_n(x)\varphi_n(x)^*\|<\ep$ and $\|a-\varphi_n(x)\varphi_n(x)^*\|<\ep$ for all $n\geq n_0$, so
$A$ is stable. 

Parts (3) and (4) are almost identical, so we only prove (3). (The definition of nuclear dimension
is given on page 3 of the introduction.) Set $n=\dimnuc(B)$. Let $F\subseteq A$ be a finite set consisting of
contractions, and let $\ep>0$. Find finite-dimensional $C^*$-algebras $C_0,\ldots,C_n$, a completely positive contractive map
$\varphi\colon B\to C=\oplus_{j=0}^n$ and completely positive contractive order zero maps 
$\psi_j\colon C_j\to B$, for $j=0,\ldots,n$, satisfying $\left\|\sum_{j=0}^n\psi_j(\varphi(\iota(a)))-\iota(a)\right\|<\ep$ for all
$a\in F$. Set 
\[\ep_0=\ep-\max_{a\in F}\left\|\sum_{j=0}^n\psi_j(\varphi(\iota(a)))-\iota(a)\right\|>0.\]
Using that the cones over the finite-dimensional \ca s $C_0,\ldots, C_n$ are projective and the correspondence between order zero maps from
$C_j$ and homomomorphisms from its cone, find $\delta>0$ such that for all $j=0,\ldots,n$, whenever 
$\kappa\colon C_j\to A$ is a completely positive contractive map which is $\delta$-order zero (meaning that $\|\kappa(c)\kappa(d)\|<\delta$ for all 
positive orthogonal contractions $c,d\in C_j$), then there exists a completely positive contractive order zero map 
$\rho\colon C_j\to A$ satisfying $\|\rho-\kappa\|<\ep_0/2(n+1)$. 

Let $A_0\subseteq A$ be the separable subalgebra generated by $F$ and let $B_0\subseteq B$ be the separable subalgebra
generated by $F$ together with the images of $\psi_0,\ldots,\psi_n$. 
Denote by $C_j^{(1)}$ the unit ball of $C_j$, for $j=0,\ldots,n$.
Since $\iota\colon A\to B$ is positively existential, and arguing as in part (1) above, we can find a completely positive
contractive map $\theta\colon B_0\to A$ which is $\delta$-multiplicative on $\bigcup_{j=0}^n\psi_j(C_j^{(1)})$ and satisfies
$\|\theta(\iota(a))-a\|<\ep_0/2$. The maps $\theta\circ\psi_0,\ldots,\theta\circ\psi_n$ are thus $\delta$-order zero in the 
sense above, and hence there exist completely positive contractive order zero maps 
$\rho_j\colon C_j\to A$, for $j=0,\ldots,n$, satisfying $\|\rho_j-\theta\circ\psi_j\|<\ep_0/2(n+1)$. 
For $a\in F$, we have 
\begin{align*}
 \left\|\sum_{j=0}^n\rho_j(\varphi(\iota(a)))-a\right\|&\leq \ep_0/2+ \left\|\sum_{j=0}^n\theta(\psi_j(\varphi(\iota(a))))-a\right\| \\
 &\leq \ep_0/2+\ep_0/2 +\left\|\sum_{j=0}^n\theta(\psi_j(\varphi(\iota(a))))-\theta(a)\right\|\\
 &\leq \ep_0 +\left\|\sum_{j=0}^n\psi_j(\varphi(\iota(a)))-a\right\|<\ep,
\end{align*}
thus showing that $\dimnuc(A)\leq n$, as desired.

The proofs of (5) and (6) are almost identical, and for the sake of brevity we only prove the result for the real rank. 
Let $n\in\N$ and suppose that $\RR(B)\leq n$.
By passing to unitizations, we may assume that $A$, $B$ and $\iota$ are unital. Given
self-adjoint elements $a_0,\ldots,a_n$ in $A$ and $\ep>0$, find self-adjoint elements $b_0,\ldots,b_n\in B$ such that $\sum\limits_{j=1}^n b_j^2$ is invertible in $B$ and
$\left\|\sum\limits_{j=1}^n(\iota(a_j)-b_j)^2\right\|<\ep$. Denote by $A_0\subseteq A$ the (separable) subalgebra generated by $\{a_0,\ldots,a_n\}$, and denote by
$B_0\subseteq B$ the (separable) subalgebra generated by $\iota(A_0)\cup \{b_0,\ldots,b_n\}$. Let $\varphi\colon B_0\to A_\I$ be a homomorphism as
in \autoref{df:posexistemb}. For $j=0,\ldots,n$, set $x_j=\varphi(b_j)$. Then $x_j$ is self-adjoint, satisfies $\|\iota(a_j)-x_j\|<\ep$, and $\sum\limits_{j=1}^n x_j^2$ is
invertible in $A_\I$. By choosing representing sequences for the $x_j$ and picking suitable elements in these sequences, we obtain self-adjoint elements
$y_0,\ldots,y_n$ in $A$ with $\sum\limits_{j=0}^n y_j^2$ invertible and satisfying
$\left\|\sum\limits_{j=1}^n(a_j-y_j)^2\right\|<\ep$.
We conclude that $\RR(A)\leq n$, and hence $\RR(A)\leq \RR(B)$.

We prove (7). Observe that $\iota\otimes\id_{M_n}\colon M_n(A)\to M_n(B)$ is also positive existential, and so is its unitization.
Thus, in order to show that $K_0(\iota)$ is injective, it is enough to show that if $p$ and $q$ are projections in $A$
such that $\iota(p)$ and $\iota(q)$ are Murray-von Neumann equivalent in $B$, then $p$ and $q$ are
Murray-von Neumann equivalent in $A$.

Let $p,q\in A$ satisfy $\iota(p)\sim_{\mathrm{M-vN}}\iota(q)$ in $B$, and choose a partial isometry
$v\in B$ satisfying $v^*v=\iota(p)$ and $vv^*=\iota(q)$. Set $A_0=C^*(p,q)$ and $B_0=C^*(\iota(p),\iota(q),v)$.
Find a homomorphism $\varphi\colon B_0\to A_\infty$ as in \autoref{df:posexistemb}, and
choose linear maps $\varphi_n\colon B_0\to A$ representing $\varphi$. These maps are approximately
multiplicative and approximately $\ast$-preserving, and are approximately the identity on $\iota(A)$.
Thus, for $n$ large enough we have
\bi\item
$\|\varphi_n(v)^*\varphi_n(v)-p\|<\frac{1}{2}$;
\item $\|\varphi_n(v)\varphi_n(v)^*-q\|<\frac{1}{2}$;
\item $\|q\varphi_n(v)p-\varphi_n(v)\|<\frac{1}{2}$.\ei
For $s=q \varphi_n(v) p$, we have $\|s^*s-p\|<1$ and $\|ss^*-q\|<1$. By Lemma~2.5.3 in~\cite{Lin_introduction_2001},
we conclude that $p$ and $q$ are Murray-von Neumann equivalent in $A$. This shows that $K_0(\iota)$ is injective.
The argument for $K_1(\iota)$ is analogous, or can be deduced from the statement for $K_0$ by taking suspensions.

We show (8). Without loss of generality, we assume that $K_0(B)$ is uniquely $n$-divisible, and will show that
so is $K_0(A)$. For this, it suffices to show that $K_0(A)$ is $n$-divisible, since uniqueness follows from injectivity
of $K_0(\iota)$ (by part~(7)). By taking unitizations, we may assume that $A$, $B$ and $\iota$ are unital. As before,
upon taking matrix amplifications, it suffices to show that if $p\in A$ is a projection, and there exists a projection
$q\in B$ such that $\diag(q,\ldots,q)$ is Murray-von Neumann equivalent to $\diag(\iota(p),0,\ldots,0)$ in $M_n(B)$, then there
exists a projection $\widetilde{q}\in M_n(A)$ which is Murray-von Neumann equivalent to $\diag(p,0,\ldots,0)$ in $M_n(A)$
and such that $\iota(\widetilde{q})$ is Murray-von Neumann equivalent to
$\diag(q,\ldots,q)$ in $M_n(B)$.
We fix a partial isometry $s\in M_n(B)$ implementing the equivalence $\diag(q,\ldots,q)\sim_{\mathrm{M-vN}}\diag(\iota(p),0,\ldots,0)$.

Let $A_0$ be the subalgebra of $M_n(A)$ generated $\diag(p,0,\ldots,0)$ and the unit, and let $B_0$ be the subalgebra of
$M_n(B)$ generated by $\diag(q,\ldots,q)$ and $s$. Observe that $\iota\otimes\id_{M_n}\colon M_n(A)\to M_n(A)$ is also
positive existential, and let $\varphi\colon B_0\to (M_n(A))_\I$ be a homomorphism as in
\autoref{df:posexistemb}. Then $\varphi(s)$ implements a Murray-von Neumann equivalence between $\diag(p,0,\ldots,0)$ and
$\varphi(q,\ldots,q)$. If $(q_k)_{k\in\N}\in \ell^\I(M_n(A))$ is a lift of $\varphi(q,\ldots,q)$ consisting of projections, and
$(s_k)_{k\in\N}\in \ell^\I(M_n(A))$ is a lift of $\varphi(s)$ consisting of partial isometries, an argument similar to the one used in
part~(7) shows that for $k$ large enough, $\widetilde{q}=q_k$ is Murray-von Neumann equivalent, in $M_n(A)$, to $\diag(p,0,\ldots,0)$,
as witnessed by $s_k$.
\end{proof}


Our next goal is to show that nuclear $C^*$-algebras satisfying the Universal Coefficient Theorem (UCT)
are closed under existential embeddings and thus satisfy (P2). This was observed, in the simple case, in Theorem~3.13 of~\cite{Gar_crossed_2017},
and was then generalized in Theorem~2.12 of~\cite{BarSza_sequentially_2016}. We give an alternative proof below, using
the following well-known consequence of the Arveson Extension Theorem. Since we were not able to find this formulation in
the literature, we sketch a proof here.

\begin{thm}\label{thm:ArvExt} (Arveson).
Let $A$ and $B$ be \ca s, with $A$ unital and $B$ nuclear, let $E\subseteq A$ be an operator subsystem, and let
$\rho\colon E\to B$ be a completely positive contractive map. Let $F\subseteq E$ be a finite set and let $\ep>0$.
Then there exists a completely positive contractive map $\theta\colon A\to B$ making the following
diagram commute on $F$ up to $\epsilon$:
\begin{align*}
\xymatrix{
E\ar@{^{(}->}[d]\ar[r]^-\rho & B\\
A.\ar[ur]_-{\theta} &}
\end{align*}
\end{thm}
\begin{proof} Using nuclearity of $B$, find a sufficiently large $n\in\N$, and
completely positive contractive maps $\varphi\colon B\to M_n$ and $\psi\colon M_n\to B$ such that
$\|\psi(\varphi(b))-b\|<\ep$ for all $b\in \rho(F)$. By Arveson's Extension Theorem (see, for example,
Theorem~1.6.1 in~\cite{BroOza_algebras_2008}), there exists a
completely positive contractive map $\widetilde{\theta}\colon A\to M_n$ making the following diagram commute:
\begin{align*}
\xymatrix{
E\ar@{^{(}->}[d]\ar[r]^-\rho & B\ar[dr]_-{\varphi} \ar[rr]^{\id} & & B\\
A\ar[rr]_-{\widetilde{\theta}} & & M_n.\ar[ur]_-{\psi} &}
\end{align*}

The proof is completed by taking $\theta=\psi\circ\widetilde{\theta}$.
\end{proof}



In the following, we denote by $\mathcal{O}_\infty^{\mathrm{st}}$ the unique unital Kirchberg
algebra satisfying the UCT with scaled $K$-theory given by $(\Z,0,\{0\})$.

\begin{thm}\label{thm:UCTP2}
The class of separable, nuclear $C^*$-algebras satisfying the UCT is closed under positively existential embeddings (property (P2)).
That is, if $B$ is separable, nuclear and satisfies the UCT, and $A\hookrightarrow B$ is a positively existential embedding,
then $A$ is separable, nuclear and satisfies the UCT.
\end{thm}
\begin{proof}
Let $A$ be a \ca, let $B$ be a nuclear \ca\ satisfying the UCT, and
let $\iota\colon A\to B$ be a positively existential embedding.
It is clear that $A$ is separable. Nuclearity of $A$ can be deduced as follows. Take $A_0=A$ and $B_0=B$ in \autoref{df:posexistemb}
and apply Choi-Effros to $\varphi\colon B\to A_\I$ to obtain a sequence $(\varphi_n)_{n\in\N}$ of completely positive contractive
linear maps $\varphi_n\colon B\to A$ satisfying $\lim\limits_{n\to\I}\varphi_n(a)=a$ for all $a\in A$. Let $\ep>0$ and let $F\subseteq A$
be a finite subset. Since $B$ is nuclear, there exist a finite dimensional \ca\ $E$ and completely positive contractive maps
$\xi\colon B\to E$ and $\eta\colon E\to B$ satisfying $\|(\eta\circ\xi)(\iota(a))-\iota(a)\|<\ep$ for all $a\in F$. Thus, for large enough
$n$, we have $\|(\varphi_n\circ \eta\circ\xi\circ\iota)(a)-a\|<\ep$ for all $a\in F$, so $A$ is nuclear as well.

For the statement about the UCT, we make some standard reductions first,
so as to fit within the framework of Theorem~5.2 in~\cite{Dad_fiberwise_2009}. We
identify $A$ with its image under $\iota$. By taking unitizations, we may assume that $A$ is a unital
subalgebra of $B$. Since for a \ca\ $C$, the tensor product $C\otimes \mathcal{O}_\infty^{\mathrm{st}}$ is $KK$-equivalent to $C$, it follows
that $C$ satisfies the UCT if and only if $C\otimes \mathcal{O}_\infty^{\mathrm{st}}$ does.
Thus, by tensoring $A$ and $B$ with $\mathcal{O}_\infty^{\mathrm{st}}$,
we may assume that $A$ and $B$ have unital embeddings of $\Ot$.
By the proof of Theorem~2.5 of~\cite{Dad_fiberwise_2009} applied to the one point space $X$, there
exists a homomorphism $\varphi \colon B\to B$
satisfying $\varphi(A)\subseteq A$, such that the stationary inductive limits
\[B^\sharp=\varinjlim (B,\varphi) \ \mbox{ and } \ A^\sharp=\varinjlim (A,\varphi|_A)\]
are unital Kirchberg algebras which are $KK$-equivalent to $B$ and $A$, respectively. Thus, $B^\sharp$ satisfies
the UCT because $B$ does, and it is enough to show that $A^\sharp$ satisfies the UCT.

We claim that the induced inclusion $A^\sharp \hookrightarrow B^\sharp$ is a positively existential embedding.
To show this, we will construct a homomorphism $\widetilde{\psi}\colon B^\sharp \to (A^\sharp)_\infty$ which restricts to the
canonical embedding of $A^\sharp$. Since $A^\sharp$ and $B^\sharp$ are separable, we may choose increasing families
$(F_n)_{n\in\N}$ and $(G_n)_{n\in\N}$ of finite subsets of $A^\sharp$ and $B^\sharp$ with dense union.
For $n\in\N$, we write $A_n$ for $A$ when we regard it as the $n$-th stage of the
stationary inductive limit defining $A^\sharp$, and similarly with $B$. Without loss of generality, we may assume
that $F_n\subseteq A_n$ and $G_n\subseteq B_n$ for all $n\in\N$.
It is enough to construct a sequence $\theta_n\colon B^\sharp \to A^\sharp$ of unital completely positive maps
satisfying $\|\theta_n(b_1b_2)-\theta_n(b_1)\theta_n(b_2)\|<1/n$ for all $b_1,b_2\in G_n$ and $\|\theta_n(a)-a\|<1/n$
for all $a\in F_n$.

Fix $n\in\N$. Using that $A\hookrightarrow B$ is a positively existential embedding, find a homomorphism
$\psi\colon B\to A_\infty$ as in \autoref{df:criterionCP}. Using nuclearity of $A$ and the Choi-Effros Lifting Theorem,
find a unital completely positive map $\rho\colon B\to A$ which satisfies
\bi
\item $\|\rho(b_1b_2)-\rho(b_1)\rho(b_2)\|<\frac{1}{n}$ for all $b_1,b_2\in G_n$, and
\item $\|\rho(a)-a\|<\frac{1}{n}$ for all $a\in F_n$.\ei

Regard $\rho$ as a unital completely positive map $B_n\to A_n\hookrightarrow A^\sharp$. Since $A^\sharp$ is nuclear,
Arveson's Extension Theorem (in the form of \autoref{thm:ArvExt}) implies that there is a
unital completely positive map $\theta_n\colon B^\sharp\to A^\sharp$ which agrees with $\rho$ on $G_n$ up to
$1/n$. The induced map
\[\theta=(\theta_n)_{n\in\N}\colon B^\sharp \to (A^\sharp)_\infty\]
is clearly a homomorphism which restricts to the standard inclusion of $A^\sharp$ into its sequence algebra. This
proves the claim.

The rest of the proof is identical to that of Theorem~3.13 in~\cite{Gar_crossed_2017}.
Denote by $\mathcal{C}$ the class of weakly semiprojective unital Kirchberg algebras satisfying the UCT.
Use Lemma~3.12 in~\cite{Gar_crossed_2017} to write $B^\sharp$ as a direct limit of algebras in $\mathcal{C}$.
Then $A^\sharp$ is a generalized local $\mathcal{C}$-algebra (Definition~3.1 in~\cite{Gar_crossed_2017}).
By Proposition~3.9 (which is a straightforward consequence of Theorem~3.9 in~\cite{Thi_inductive_2018}),
$A^\sharp$ is a unital approximate $\mathcal{C}$-algebra, that is, a direct limit of \ca s in $\mathcal{C}$.
It follows that $A^\sharp$ satisfies the UCT, and the proof is finished.
\end{proof}

The following combines the results obtained in this section, proving Theorem~D in the introduction.

\begin{thm}\label{thm:preservationCP}
Let $G$ be a compact group of finite covering dimension, let $A$ be a \ca, and let $\alpha\colon G\to\Aut(A)$ be an action
with $\cdimRok(\alpha)<\infty$. If $A$ belongs to any of the following classes, then so do $A\rtimes_\alpha G$ and $A^\alpha$:
\be
\item For $\mathcal{D}$ being a strongly self-absorbing \ca, if $A$ is separable and $\mathcal{D}$-absorbing,
then so are $A\rtimes_\alpha G$ and $A^\alpha$.
\item If $A$ is separable and stable, then so are $A\rtimes_\alpha G$ and $A^\alpha$.
\item We have
\[\dimnuc(A^\alpha)=\dimnuc(A\rtimes_\alpha G)\leq(\cdimRok(\alpha)+1)(\dimnuc(A)+1)-1.\]
\item We have
\[\dr(A^\alpha)=\dr(A\rtimes_\alpha G)\leq(\cdimRok(\alpha)+1)(\dr(A)+1)-1.\]
\item We have
\[
\mathrm{sr}(A\rtimes_\alpha G)\leq \mathrm{sr}(A^\alpha)\leq (\cdimRok(\alpha)+1)(\dim(G)+1)+\mathrm{sr}(A)-1.
\]
\item We have
\[\RR(A\rtimes_\alpha G)\leq \RR(A^\alpha)\leq (\cdimRok(\alpha)+1)(\dim(G)+1)+\RR(A)-1.\]
\item Let $n\in\N$. If $A$ is nuclear, satisfies the UCT, and has uniquely $n$-divisible $K$-theory, then the same is true for $A\rtimes_\alpha G$ and
$A^\alpha$.
\item If $A$ is nuclear, satisfies the UCT, and has rational $K$-theory, then the same is true for $A\rtimes_\alpha G$ and
$A^\alpha$.
\item If $A$ is nuclear, satisfies the UCT, and has trivial $K$-theory, then the same is true for $A\rtimes_\alpha G$ and
$A^\alpha$.
\item If $A$ is separable, nuclear, and satisfies the UCT, then the same is true for $A\rtimes_\alpha G$ and
$A^\alpha$.
\ee
\end{thm}
\begin{proof}
We argue for the fixed point algebra $A^\alpha$ first.
By \autoref{prop:ApproxCP}, it is enough to check that the classes in the statement
satisfy conditions (P1) and (P2) of \autoref{df:criterionCP}.
For the former, this is the content of \autoref{thm:P1}, and for the latter, this is the content of
\autoref{prop:P2} and \autoref{thm:UCTP2}.

Recall that the crossed product $A\rtimes_\alpha G$ is Morita equivalent to $A^\alpha\otimes\K(L^2(G))$.
Now, the properties in (1)--(2) and (7)--(10) are invariant under
stable isomorphism, and the nuclear dimension and decomposition rank of a \ca\ and its stabilization
are identical, so the result for the crossed product follows.
It remains to observe that for any \ca\ $B$, one has $\mathrm{sr}(B\otimes \K)\leq \mathrm{sr}(B)$
and $\RR(B\otimes\K)\leq \RR(B)$, so the proof is finished.
\end{proof}

\begin{rem} \label{rmk:nonCommTws}
We make some comments about preservation of the classes in \autoref{thm:preservationCP}
by actions with finite Rokhlin dimension (not necessarily with commuting towers).
\be
\item[(a)] For a strongly self-absorbing \ca\ $\mathcal{D}$, denote by $\mathcal{A}_\mathcal{D}$ the class
of $\mathcal{D}$-absorbing separable \ca s. Then $\mathcal{A}_{\mathcal{O}_2}$ and $\mathcal{A}_{M_{p^\infty}}$,
for a prime $p\geq 2$, are not preserved by actions with finite Rokhlin dimension. Indeed, Example~4.8
in~\cite{Gar_rokhlin_2017} shows that $\Ot$-absorption is not preserved (this action was originally constructed
in \cite{Izu_finiteI_2004}). It also shows that absorption of UHF-algebras
other than the CAR algebra $M_{2^\infty}$ is not preserved. In order to rule out preservation of absorption of $M_{2^\infty}$,
one may adapt Izumi's construction to produce an approximately representable action of $\Z_3$ on $\Ot$ whose
equivariant $K$-theory is not uniquely 2-divisible. (This is done, for example, in \cite{GarLup_conjugacy_2016}.)
Whether the classes $\mathcal{A}_{\mathcal{O}_\infty}$ or $\mathcal{A}_{\mathcal{Z}}$ are preserved, is still unknown.
\item[(b)] The examples mentioned above also show that the classes in (7), (8) and (9) are not preserved.
\item[(c)] Whether the class of nuclear algebras satisfying the UCT is preserved by actions with finite Rokhlin dimension,
is in fact equivalent to the UCT problem for separable, nuclear \ca s. Indeed, Barlak has observed that the UCT
problem for separable, nuclear \ca s can be reduced to the question of whether the crossed products of certain
pointwise outer actions of finite groups on $\mathcal{O}_2$ satisfy the UCT; see Theorem~4.17 in~\cite{BarSza_rokhlin_2017}
for the precise statement. By Theorem~4.19 in~\cite{Gar_rokhlin_2017},
any such action has Rokhlin dimension at most one, which proves our claim. \ee\end{rem}

In reference to part~(1) of \autoref{thm:preservationCP},
we show next that stability (that is, $\K$-absorption) of $A\rtimes_\alpha G$ is automatic whenever $G$ is infinite,
and regardless of $A$. A similar result for $\R$-actions has been obtained in Theorem~D of~\cite{HirSzaWinWu_rokhlin_2017}.

\begin{cor}\label{cor:GinfRRsr} Let $G$ be a second-countable compact group, let $A$ be a separable \ca, and let $\alpha\colon
 G\to \Aut(A)$ be an action with $\cdimRok(\alpha)<\infty$. If $G$ is infinite, then $A\rtimes_\alpha G$ is stable.
\end{cor}
\begin{proof} When $G$ is infinite, $L^2(G)$ is the infinite dimensional separable Hilbert space. Let $X$ be the
finite dimensional compact free $G$-space from \autoref{thm:XRpRdim}. By \autoref{cor:freeCXGalgebra}, the $C(X/G)$-algebra
$C(X,A)\rtimes_\gamma G$ has stable fibers. Since $X/G$ is finite dimensional, Proposition~3.4 in~\cite{HirRorWin_algebras_2007} implies
that $C(X,A)\rtimes_\gamma G$ is stable. Since the canonical inclusion
\[A\rtimes_\alpha G\hookrightarrow C(X,A)\rtimes_\gamma G\]
is positively existential by \autoref{prop:ApproxCP}, it follows from part~(2) in~\autoref{prop:P2} that $A\rtimes_\alpha G$ is stable.
\end{proof}

An immediate consequence is the following.

\begin{cor}
Let $G$ be an infinite compact group with finite dimension, let $A$ be a \ca, and let $\alpha\colon G\to\Aut(A)$ be an action with
$\cdimRok(\alpha)<\infty$.
Then $\sr(A\rtimes_\alpha G)\leq 2$ and $\RR(A\rtimes_\alpha G)\leq 1$, regardless of $A$.
\end{cor}



Adopt the notation of \autoref{prop:ApproxCP}. If one moreover knew that $C(X,A)\rtimes G$ is Morita equivalent to
$(C(X)\rtimes G)\otimes A$, then it would follow from Situation 2 in \cite{Rie_applications_1982} that $C(X,A)\rtimes G$ is
itself Morita equivalent to $C(X/G)\otimes A$, since the action is free and the group is compact.
This is, however, not the case in general, as the next example shows.

\begin{eg} Let the non-trivial element of $G=\Z_2=\{-1,1\}$ act on $X=S^1$ via multiplication by $-1$, and let it act on
$A=\OIst$ by an order two automorphism whose induced map on $K_0(\mathcal{O}_\infty^{\mathrm{st}})\cong\Z$ is multiplication
by $-1$. Denote by $\gamma\colon\Z_2\to\Aut(C(S^1))$ and
$\alpha\colon\Z_2\to\Aut(\OIst)$ the corresponding actions.
Set $B=C(S^1,\OIst)$ and denote by $\beta\colon\Z_2\to \Aut(B)$
the diagonal action $\beta=\gamma\otimes\alpha$, where we, as usual, identify $C(S^1,\mathcal{O}_\infty^{\mathrm{st}})$ with
$C(S^1)\otimes \mathcal{O}_\infty^{\mathrm{st}}$.

Let $z\in C(S^1)$ be the canonical generating unitary, and set
$v=z\otimes 1\in C(S^1)\otimes \OIst\cong B$. Regard $v$ as a unitary in $B\rtimes_\beta\Z_2$ under the canonical unital inclusion $B\to B\rtimes_\beta\Z_2$.

\textbf{Claim:} $\widehat{\beta}_{-1} \in\Aut( B\rtimes_\beta\Z_2)$ is inner, and it is implemented by $v$. (In particular,
$\widehat{\beta}_{-1}$ acts trivially on $K$-theory.) Note first that $\beta_{-1}(v)=-v$. Denote by
$u\in B\rtimes_\beta\Z_2$ the canonical unitary implementing $\beta$. To check that
$\widehat{\beta}_{-1}=\Ad(v)$, it suffices to show that these automorphisms agree on the generating set
$F=\{u\}\cup B$. Since $v$ commutes with the elements of $B$
(and $\widehat{\beta}_{-1}$ fixes these elements), it is enough to show that $\widehat{\beta}_{-1}(u)=v^*uv$.
Since $\widehat{\beta}_{-1}(u)=-u$,
this follows from the fact that $\beta_{-1}(v)=uvu^*=-v$, and the claim follows.

Observe that $\beta_{-1}$ induces the automorphism of multiplication by $-1$ on both $K_0(B)$ and $K_1(B)$ (essentially
because this is the case for $\alpha_{-1}$). We compute the $K$-theory of $B\rtimes_\beta\Z_2$ using its Pimsner-Voiculescu
6-term exact sequence. The $K$-groups of $B\rtimes_\beta\Z$ can be obtained from the following sequence:
\beqa\xymatrix{K_0(B)\ar[rr]^-{1-K_0(\beta_{-1})} &&K_0(B)\ar[rr] &&K_0(B\rtimes_\beta\Z)\ar[d]\\
K_1(B\rtimes_\beta\Z)\ar[u]&& K_1(B)\ar[ll] &&K_1(B)\ar[ll]^-{1-K_1(\beta_{-1})}.}\eeqa
Since $1-K_j(\beta_{-1})$ is multiplication by 2 and $K_j(B)\cong \Z$ for $j=0,1$, it follows that the vertical maps are zero and that
$K_0(B\rtimes_\beta\Z)\cong K_1(B\rtimes_\beta\Z)\cong\Z_2$. To compute the $K$-groups of $B\rtimes_\beta\Z_2$, we
use the following sequence:
\beqa\xymatrix{K_0(B\rtimes_\beta\Z_2)\ar[rr]^-{1-K_0(\widehat{\beta})} &&K_0(B\rtimes_\beta\Z_2)\ar[rr] &&K_1(B\rtimes_\beta\Z)\ar[d]\\
K_0(B\rtimes_\beta\Z)\ar[u]&& K_1(B\rtimes_\beta\Z_2)\ar[ll] &&K_1(B\rtimes_\beta\Z_2)\ar[ll]^-{1-K_0(\widehat{\beta})},}\eeqa
where the vertical maps are the canonical ones induced by the quotient map $B\rtimes_\beta\Z\to B\rtimes_\beta\Z_2$.
The map $K_j(B\rtimes_\beta\Z_2)\to K_j(B\rtimes_\beta\Z_2)$, for $j=0,1$, is zero because
$\widehat{\beta}$ acts by inner automorphisms. Since the $K$-groups of $B\rtimes_\beta\Z$ are both $\Z_2$,
it follows that either
\[K_0(B\rtimes_\beta\Z_2)\cong \Z_2 \ \mbox{ and } \  K_1(B\rtimes_\beta\Z_2)\cong \{0\}\]
or
\[K_0(B\rtimes_\beta\Z_2)\cong \{0\} \ \mbox{ and } \  K_1(B\rtimes_\beta\Z_2)\cong \Z_2.\]
In either case, it follows that
$C(S^1,\mathcal{O}_\infty^{\mathrm{st}})\rtimes_\beta \Z_2$ and $C(S^1/\Z_2)\otimes \mathcal{O}_\infty^{\mathrm{st}}$ are not
Morita equivalent, since $K_0(C(S^1/\Z_2)\otimes \OIst)\cong K_1(C(S^1/\Z_2)\otimes \OIst)\cong\Z$ by the
K\"unneth formula. \end{eg}



Recall that the Toms-Winter conjecture predicts that finiteness of the nuclear dimension, $\mathcal{Z}$-absorption and strict
comparison are equivalent for all simple, separable, infinite dimensional, nuclear \uca s.
We make some connections between Rokhlin dimension and preservation of the properties in the Toms-Winter conjecture.
\autoref{thm:preservationCP} shows that $\mathcal{Z}$-absorption,
finiteness of the nuclear dimension and finiteness of the decomposition rank are preserved under formation of crossed products
(and fixed point algebras) by actions with finite Rokhlin dimension with commuting towers.
A similar result for strict comparison is false in general, although it holds with further assumptions, as we
show in \autoref{thm:StrCompRdim1}. We need an easy lemma first.

\begin{lma}\label{lma:PosExEmbOrdEmbCu}
Let $A$ and $B$ be \ca s, and let $\iota\colon A\to B$ be a positively existential embedding. Then
$\Cu(\iota)\colon \Cu(A)\to \Cu(B)$ is an order embedding. In other words, if $s,t\in\Cu(A)$ satisfy $\Cu(\iota)(s)\leq \Cu(\iota)(t)$ in
$\Cu(B)$, then $s\leq t$ in $\Cu(A)$.
\end{lma}
\begin{proof}
Since $\iota\otimes\id_\K\colon A\otimes\K\to B\otimes\K$ is an existential embedding, we may assume that $A$ and $B$ are stable.
Let $s,t\in\Cu(A)$ satisfy $\Cu(\iota)(s)\leq \Cu(\iota)(t)$ in $\Cu(B)$, and choose positive elements $a,b\in A$ with $[a]=s$ and $[b]=t$.
Since $\iota(a)$ is Cuntz subequivalent to $\iota(b)$, there exists a sequence $(c_n)_{n\in\N}$ in $B$ satisfying
$\lim\limits_{n\to\I}c_nbc_n^*=a$. Let $A_0$ and $B_0$ be, respectively, the separable subalgebras of $A$ and $B$ generated by
$\{a,b\}$ and $\{\iota(a),\iota(b),c_n\colon n\in\N\}$. Let $\varphi\colon B_0\to A_\I$ be a homomorphism as in \autoref{df:posexistemb}.
By choosing suitable elements in representative sequences of the $\varphi(c_n)$, we can find a sequence $(d_n)_{n\in\N}$ in $A$
satisfying $\lim\limits_{n\to\I}d_nbd_n^*=a$. Hence $a\precsim b$ in $A$ and thus $s\leq t$.
\end{proof}

We need a convenient definition from \cite{AntBosPer_cuntz_2013}.

\begin{df}\label{df:noK1obstr}
We say that a \ca\ $A$ has \emph{no $K_1$-obstructions} if it has stable rank one and
$K_1(I)=0$ for all ideals $I$ in $A$.\end{df}

The following is our result on preservation of strict comparison.

\begin{thm} \label{thm:StrCompRdim1}
Let $G$ be a compact group, let $A$ be a \ca, and let $\alpha\colon G\to\Aut(A)$ be an action with
$\cdimRok(\alpha)\leq 1$. Suppose further that $A$ has stable rank one and that $K_1(I)=0$ for all ideals $I$ in $A$.
Then the canonical inclusions
\[A^\alpha\hookrightarrow A \ \ \mbox{ and } \ \ A\rtimes_\alpha G\hookrightarrow A\otimes\K(L^2(G))\]
induce order embeddings at the level of the Cuntz semigroup. In particular, strict comparison passes
from $A$ to $A^\alpha$ and to $A\rtimes_\alpha G$.
\end{thm}
\begin{proof}
By tensoring $A$ with $\K$ and $\alpha$ with $\id_\K$, we may assume that $A$, $A^\alpha$, and $A\rtimes_\alpha G$ are stable.
We prove the statement for the second inclusion, since the proof for the first one is identical.
Denote the canonical map $A\rtimes_\alpha G\hookrightarrow A\otimes\K(L^2(G))$ by $\theta$. Let
$a$ and $b$ be positive elements in $A\rtimes_\alpha G$, and assume that $\theta(a)\precsim \theta(b)$ in $A\otimes\K(L^2(G))$.
We want to show that $a\precsim b$ in $A\rtimes_\alpha G$.

Let $X$ be the compact Hausdorff free $G$-space from \autoref{thm:XRpRdim}, and observe that $\dim(X/G)\leq 1$. Denote by
\[\iota\colon A\rtimes_\alpha G \to C(X,A)\rtimes_\gamma G\]
the inclusion induced by the equivariant factor embedding $A\to C(X,A)$. For $z\in X/G$, denote by
\[\pi_z\colon C(X,A)\rtimes_\gamma G\to A\otimes\K(L^2(G))\]
the quotient map onto the fiber over $z$. Then $\pi_z\circ\iota=\theta$ by \autoref{cor:freeCXGalgebra}. It
follows that $\pi_z(\iota(a))\precsim \pi_z(\iota(b))$ for all $z\in X/G$. By Theorem~2.6 in~\cite{AntBosPer_cuntz_2013},
it follows that $\iota(a)\precsim \iota(b)$ in $C(X,A)\rtimes_\gamma G$. Since $\iota$ is an existential embedding
by \autoref{prop:ApproxCP}, it follows from \autoref{lma:PosExEmbOrdEmbCu} that $a\precsim b$ in $A\rtimes_\alpha G$, as desired.

That strict comparison passes from  $A$ to $A^\alpha$ and $A\rtimes_\alpha G$ is an immediate consequence of the canonical inclusions
being order embeddings.
\end{proof}

Obtaining more results may depend on computing $X$ in some concrete cases. We examine the smallest non-trivial case, namely, that of actions of $\Z_2$
with Rokhlin dimension at most one. In the following lemma, we denote by
$C^*(\mathcal{G},\mathcal{R})$ the universal \ca\ generated by a set of generators $\mathcal{G}$, subject to a set
of relations $\mathcal{R}$.

\begin{lma} \label{lma:S1rotation}
Let $d\in\N$. Consider the set of generators
\[\mathcal{G}=\left\{1,f_0^{(0)},f_{1}^{(0)},\ldots,f_0^{(d)},f_{1}^{(d)}\right\}\]
and the set of relations

\[\mathcal{R}=\left\{
\begin{aligned}
& 0\leq f_g^{(j)}\leq 1, \ \left[f_g^{(j)},f_k^{(\ell)}\right]=0,\\
& f_g^{(j)}f_h^{(j)}=0, \sum_{g\in\Z_2}\sum_{j=0}^d f_g^{(j)}=1
\end{aligned}
\colon g,h,k\in\Z_2, g\neq h, j,\ell=0,\ldots,d\right\}.\]
Define an action $\alpha\colon \Z_2\to C^*(\mathcal{G},\mathcal{R})$ by $\alpha_g(f_h^{(j)})=f_{gh}^{(j)}$ for $g,h\in \Z_2$ and $j=0,\ldots,d$.
Then the $\Z_2$-$C^*$-algebra $(C^*(\mathcal{G},\mathcal{R}),\alpha)$ is canonically equivariantly
isomorphic to $C(S^d)$ with the antipodal action.
\end{lma}
\begin{proof}
We give a sketch of the proof. Each $f_g^{(j)}$ determines a 0-simplex. Moreover, there is a 1-simplex between
$f_g^{(j)}$ and $f_h^{(k)}$ whenever $f_g^{(j)}f_h^{(k)}\neq 0$. More generally, for $n\in\N$, there is an $n$-simplex connecting
$f_{g_0}^{(j_0)},\ldots,f_{g_n}^{(j_n)}$ whenever $f_{g_0}^{(j_0)}\cdots f_{g_n}^{(j_n)}\neq 0$. In particular, $X$ is a $d$-dimensional
simplicial complex. This description also allows one to see that $X$ is homeomorphic to $S^d$, and that $\alpha$ corresponds to the
antipodal action. We omit the details. \end{proof}

\begin{rem}
For cyclic groups of higher order, the space $X$ from \autoref{lma:S1rotation} is a $d$-dimensional simplicial complex, but its
explicit structure is harder to describe. In particular, it may fail to be a manifold.
\end{rem}

We use \autoref{lma:S1rotation} to provide an explicit computation of the free $G$-space $X$ in the conclusion
of \autoref{thm:XRpRdim}.

\begin{cor}\label{cor:ZnRdim1}
Let $A$ be a unital \ca\ and let $\alpha\colon \Z_2\to\Aut(A)$ be an action. Then $\cdimRok(\alpha)\leq 1$
if and only if there exists a unital equivariant homomorphism
\[\varphi\colon (C(S^1),\texttt{Lt})\to (A_\I\cap A',\alpha_\I).\]

Moreover, if $\cdimRok(\alpha)=1$, then every such map $\varphi$ is automatically injective.
\end{cor}
\begin{proof}
Since $\texttt{Lt}\colon \Z_2 \to \Aut(C(S^1))$ has Rokhlin dimension one, the ``if'' implication is clear. The converse
is contained in Lemma~1.9 in~\cite{HirPhi_rokhlin_2015}, once it is combined with \autoref{lma:S1rotation}; see also the comments before \autoref{thm: RdimwTRp}.

For the second statement, assume that there is a unital equivariant homomorphism $\varphi\colon C(S^1)\to A_\I\cap A'$ which
is not injective. We denote also by $\alpha\in \Aut(A)$ the order-two automorphism generating the given $\Z_2$-action.
We also write $\gamma\in \mathrm{Homeo}(X)$ for the antipodal homeomorphism, which equals $\texttt{Lt}_{-1}$.
Let $Y\subseteq S^1$ be a closed subset satisfying $\gamma(Y)\cap Y=\emptyset$, and such
that, with $X=Y\cup \gamma(Y)$, the map $\varphi$ induces an injective unital equivariant homomorphism
$\overline{\varphi}\colon C(X)\to A_\infty\cap A'$. Then $\varphi(\chi_{Y})$ and $\varphi(\chi_{\gamma(Y)})$
are orthogonal projections in $A_\infty\cap A'$
which witness the fact that $\alpha$ has the \Rp, that is, $\cdimRok(\alpha)=0$.
\end{proof}





\autoref{cor:ZnRdim1} can be used to give lower bounds other than 1 for certain actions of cyclic groups. In the
following example, said corollary is used to show that there exists an action $\alpha$ of $\Z_2$ on a simple AF-algebra
with unique trace such that $\cdimRok(\alpha)$ is exactly two. To our knowledge, this is the first example of an action
on a simple \ca\ with $\cdimRok$ different from $0, 1$ or $\infty$. (For the sake of comparison, we mention here that there
is no known example of a simple \ca\ with nuclear dimension or decomposition rank other than 0, 1 or $\infty$.) We do not know
of similar examples for the Rokhlin dimension \emph{without} commuting towers.

Recall that the $K_0$-group of an AF-algebra is torsion free, and that its $K_1$-group is trivial.

\begin{eg}\label{eg:cdimRok2} We review the construction of a particular case of Example~4.1 in \cite{Phi_finite_2015}.
Let $\beta\in\Aut(C(S^2))$ be the automorphism of order 2 induced by the homeomorphism $x\mapsto -x$ on $S^2$.
Set
\[A_n=C(S^2)\otimes M_3\otimes M_5\otimes \cdots\otimes M_{2n+1}.\]
Let $\{x_n\colon n\in\N\}$ be a dense subset of $S^2$, and, for $n\in\N$, define a homomorphism
$\psi_n\colon C(S^2)\to M_{2n+1}(C(S^2))$ by
\[\psi_n(f)=\diag(f,f(x_n),f(-x_n),\ldots,f(x_n),f(-x_n))\]
for $f\in C(S^2)$. Define maps $\varphi_n\colon A_{n-1}\to A_n$ by
\[\varphi_n=\psi_n\otimes\id_{M_3}\otimes\cdots\otimes\id_{M_{2n-1}}.\]
Define an order 2 automorphism $\alpha$ of $A=\varinjlim A_n$ as follows. Let
\[w_n=\diag\left(1,\left(
                     \begin{array}{cc}
                       0 & 1 \\
                       1 & 0 \\
                     \end{array}
                   \right),\ldots,\left(
                     \begin{array}{cc}
                       0 & 1 \\
                       1 & 0 \\
                     \end{array}
                   \right)
\right)\in M_{2n+1},\]
and set
\[\alpha_n=\beta\otimes \Ad\left(w_1\otimes w_2\otimes\cdots\otimes w_n\right).\]
Then $\alpha_n$ has order two, and there is a direct limit automorphism
$\alpha=\varinjlim\alpha_n$, which also has order two. In this example, we will show that $\alpha$
has Rokhlin dimension with commuting towers exactly equal to 2.

The $C^*$-algebra $A$ is unital, simple, separable, nuclear, satisfies the UCT and has tracial rank zero.
Moreover, $K_1(A)=\{0\}$ (since $K_1(A_n)=\{0\}$ for all $n\in\N$), and $K_0(A)$ is a dimension group.
By Lin's classification of tracial rank zero \ca s, it follows that $A$ is an AF-algebra. It was shown in
part~(6) of Proposition~4.2 in~\cite{Phi_finite_2015} that $K_0(A\rtimes_\alpha\Z_2)$ has torsion
isomorphic to $\Z_2$, so in particular $A\rtimes_\alpha\Z_2$ is not an AF-algebra.

First, note that $\beta$ has Rokhlin dimension with commuting towers at most 2, by Lemma~1.9 in \cite{HirPhi_rokhlin_2015}.
It then follows from part (1) of Theorem~3.8 in \cite{Gar_rokhlin_2017} that $\cdimRok(\alpha_n)\leq 2$,
so $\cdimRok(\alpha)\leq 2$
by part (3) of Theorem~3.8 in \cite{Gar_rokhlin_2017}. Since $A\rtimes_\alpha\Z_2$
is not an AF-algebra, it follows from Theorem~2.2 in~\cite{Phi_tracial_2011} that $\cdimRok(\alpha)>0$.

Suppose that $\cdimRok(\alpha)=1$.
By \autoref{cor:ZnRdim1}, there exists a unital equivariant embedding $C(S^1)\to A_\I\cap A'$. Give $C(S^1,A)$ the
diagonal action of $\Z_2$. By \autoref{cor:freeCXGalgebra}, the crossed product $C(S^1,A)\rtimes\Z_2$ is a locally trivial
$C(S^1/\Z_2)$-algebra with fibers $M_2(A)$.

We claim that $K_0(C(S^1,A)\rtimes\Z_2)$ is torsion free. Choose closed connected sets $Y_1$ and $Y_2$ in $S^1$ such that
\bi
\item $C(S^1,A)\rtimes\Z_2$ is trivial over both $Y_1$ and $Y_2$;
\item $Y_1\cup Y_2=S^1$;
\item $Y_1\cap Y_2$ is homotopic to $\{-1,1\}$. \ei

By Proposition~10.1.13 in~\cite{Dix_algebras_1977} (see also Lemma~2.4 in~\cite{Dad_continuous_2009}),
we can write the crossed product $C(S^1,A)\rtimes\Z_2$ as the pullback
\beqa\xymatrix{
C(S^1,A)\rtimes\Z_2\ar[r]\ar[d] & (C(S^1,A)\rtimes\Z_2)_{Y_1}\ar[d]\\
(C(S^1,A)\rtimes\Z_2)_{Y_2}\ar[r] & (C(S^1,A)\rtimes\Z_2)_{Y_1\cap Y_2},
}
\eeqa
where all the maps are the canonical quotient (restriction) maps. Observe that $(C(S^1,A)\rtimes\Z_2)_{Y_j}$ is
homotopic to $M_2(A)$ for $j=1,2$, and that $(C(S^1,A)\rtimes\Z_2)_{Y_1\cap Y_2}$ is homotopic to $M_2(A)\oplus M_2(A)$.
Using homotopy invariance of $K$-theory, the Mayer-Vietoris exact sequence on $K$-theory for this pullback (see Theorem~21.5.1
in~\cite{Bla_ktheory_1998}) yields
\beqa\xymatrix{
K_0(C(S^1,A)\rtimes\Z_2)\ar[r] & K_0(M_2(A))\oplus K_0(M_2(A)) \ar[r] & K_0(M_2(A)\oplus M_2(A))\ar[d]\\
K_1(M_2(A)\oplus M_2(A))\ar[u] & K_1(M_2(A))\oplus K_1(M_2(A)) \ar[l] & K_1(C(S^1,A)\rtimes\Z_2).\ar[l]
}\eeqa

Since $A$ is an AF-algebra, we have $K_1(M_2(A))=\{0\}$. It follows that the first horizontal map
$K_0(C(S^1,A)\rtimes\Z_2)\to K_0(M_2(A))\oplus K_0(M_2(A))$ is injective. Since $K_0(M_2(A))$ is torsion free,
we conclude that $K_0(C(S^1,A)\rtimes\Z_2)$ is torsion free, and the claim is proved.

Recall from \autoref{prop:ApproxCP} that there is a commutative diagram
\beqa\xymatrix{
A\rtimes_\alpha\Z_2\ar[dr]\ar[rr] && (A\rtimes_\alpha\Z_2)_\I.\\
& C(S^1,A)\rtimes\Z_2\ar[ur] & }\eeqa
Since $K_0(C(S^1,A)\rtimes\Z_2)$ is torsion free by the claim above, part~(7) in~\autoref{prop:P2} implies
that $K_0(A\rtimes_\alpha\Z_2)$ is also torsion free. However, this contradicts part~(6) of Proposition~4.2 in~\cite{Phi_finite_2015},
where it is shown that
$K_0(A\rtimes_\alpha\Z_2)$ has torsion isomorphic to $\Z_2$. This contradiction implies that $\cdimRok(\alpha)=2$, as
desired.
\end{eg}

The example above can be modified to produce a $\Z_2$-action $\gamma$ on a unital Kirchberg
algebra that satisfies the UCT satisfying $\cdimRok(\gamma)=2$ and $\dimRok(\gamma)=1$.
This is the first example of an action whose Rokhlin dimensions with and without commuting towers are
both finite but do not agree. (Examples where $\cdimRok$ is infinite but $\dimRok$ is finite were already known; see
Example~4.8 in~\cite{Gar_rokhlin_2017}, and compare Theorem~4.6 in~\cite{HirPhi_rokhlin_2015} with Theorem~4.20
in~\cite{Gar_compact_2018}.)

\begin{eg}\label{eg:cdimRokdimRokDifferent}
Let $\alpha\colon\Z_2\to\Aut(A)$ be the action from \autoref{eg:cdimRok2}. Set $B=A\otimes\OI$ and let $\gamma\colon \Z_2\to \Aut(B)$
be given by $\gamma_g=\alpha_g\otimes\id_{\OI}$ for $g\in\Z_2$. Since $\alpha$ is (pointwise) outer, so is $\gamma$. Thus,
\autoref{thm:outertRp} implies that $\dimRok(\gamma)\leq 1$. Since $K_0(B\rtimes_\gamma \Z_2)=K_0(A\rtimes_\alpha \Z_2)$
has torsion and $K_0(B)=K_0(A)$ is torsion-free,
there cannot be any embedding $K_0(B\rtimes_\gamma \Z_2)\hookrightarrow K_0(B)$.
It then follows from Theorem~B in~\cite{Gar_compact_2018} that $\gamma$ does not have the Rokhlin property.
We conclude that $\dimRok(\gamma)=1$.

We claim that $\cdimRok(\gamma)=2$. Since $\cdimRok(\alpha)=2$, we deduce from part~(1) of Theorem~3.8 in~\cite{Gar_rokhlin_2017} that
$\cdimRok(\gamma)\leq 2$. It suffices to check that $\cdimRok(\gamma)\neq 1$. Since there are $KK$-equivalences
$A\sim_{\mathrm{KK}} B$ and $A\rtimes_\alpha \Z_2 \sim_{\mathrm{KK}} B\rtimes_\gamma \Z_2$, the argument used
in \autoref{eg:cdimRok2} also applies to $\gamma$ and yields $\cdimRok(\gamma)\neq 1$, as desired.
\end{eg}

We now turn to automatic-reduction results for Rokhlin dimension. One such result already appeared
as Theorem~4.19 in~\cite{Gar_rokhlin_2017}, which asserts that for
locally representable AF-actions of finite groups on AF-algebras, finite Rokhlin dimension
with commuting towers implies the Rokhlin property. In other words, $\cdimRok(\alpha)<\I$
implies $\cdimRok(\alpha)=0$. Below we present two more instances of this phenomenon: \autoref{prop:OtwoRp} and \autoref{thm:UHFRp}.

We will need a preparatory result, of independent interest.
Recall that $F(A)$ denotes the quotient $A_\infty\cap A'/\Ann(A, A_\infty)$, and that we write $\kappa_A\colon A_\infty\cap A'\to F(A)$ for
the quotient map.

\begin{prop}\label{prop:EmbDFxPt}
Let $G$ be a second countable compact group, let $A$ be a separable \ca, and let $\alpha\colon G\to\Aut(A)$ be an action with
$\cdimRok(\alpha)<\infty$. Let $\D$ be a strongly self-absorbing \ca, and suppose that $A$ is $\D$-absorbing. Then there exists a unital embedding
\[\D\to \kappa_A(A^\alpha_\infty\cap A').\]
\end{prop}
\begin{proof}
It is readily checked, by averaging over $G$, that $\kappa_A(A^\alpha_{\alpha,\infty} \cap A')$ coincides with the fixed point algebra $F(A)^{F(\alpha)}$.

Let $X$ be the free $G$-space provided by \autoref{thm:XRpRdim}, and let
\[\varphi\colon C(X)\to F_\alpha(A)\]
be the equivariant unital embedding provided by \autoref{thm:XRpRdim}. Since $G$ is second countable, $C(X)$ is separable. Using $\D$-absorption of $A$,
we may therefore choose a unital embedding $\theta_1\colon \D\to F_\alpha(A)$ whose image commutes with $\varphi(C(X))$. Since $\varphi$ is equivariant, it follows
that $\alpha_g(\theta_1(\D))$ commutes with $\varphi(C(X))$ for all $g\in G$. Define $B_1$ to be the \uca\ generated by $\bigcup\limits_{g\in G}\alpha_g(\theta_1(\D))$.
Then $B_1$ is separable, is $\alpha_\infty$-invariant, commutes with $\varphi(C(X))$, and contains a unital copy of $\D$. Set $E_1=B_1$. By seprability of $E_1$, we can find a unital
embedding $\theta_2\colon \D\to F_\alpha(A)$
whose image commutes with $E_1\cup \varphi(C(X))$. Define $B_2$ to be the \uca\ generated by $\bigcup\limits_{g\in G}\alpha_g(\theta_2(\D))$. Then $B_2$ is separable, $\alpha_\infty$-invariant,
commutes with $E_1\cup \varphi(C(X))$,
and contains a unital copy of $\D$. Set $E_2=C^*(E_1\cup B_2)$. Proceed inductively to construct a separable, $\alpha_\infty$-invariant \uca\ $B_n$, which commutes with $E_{n-1}\cup \varphi(C(X))$
and contains a unital copy
of $\D$. Set $E_n=C^*(E_{n-1}\cup B_n)$. Then the inductive limit $E=\varinjlim E_n$ is a separable, $\alpha_\infty$-invariant unital subalgebra of $F_\alpha(A)$, which commutes with $\varphi(C(X))$
and absorbs $\D$.

Denote by $\gamma$ the diagonal action on $C(X,E)$. Then $\varphi\otimes\id_E\colon C(X,E)\to F_\alpha(A)$ is an equivariant embedding, which
therefore maps $C(X,E)^\gamma$ into $F_\alpha(A)^{F(\alpha)}$.
Consider the algebra
\[C=\{f\in C(X,E)\colon F(\alpha)_g(f(x))=f(g^{-1}\cdot x) \mbox{for all } g\in G, x\in X\}\subseteq C(X,E)^\gamma.\]
Then $C$ is a $C(X/G)$-algebra with fibers isomorphic to $E$. Since $E$ absorbs $\D$ and $X/G$ is finite dimensional, it follows from
Theorem~4.6 in~\cite{HirRorWin_algebras_2007} that $C$ absorbs $\D$. Thus, there exists a unital embedding $\D\to C$. The desired map is obtained as the following composition:
\[\xymatrix{\D\ar[r]& C \ar[r]& C(X,E)^\gamma \ar[rr]_{\varphi\otimes\id_E} &&F_\alpha(A)^{F(\alpha)}}.\]
\end{proof}

A standard argument, as used in \cite{HirWin_rokhlin_2007}, now shows the following.

\begin{cor}
Let $G$ be a second countable compact group, let $A$ be a separable \ca, and let $\alpha\colon G\to\Aut(A)$ be an action with
$\cdimRok(\alpha)<\infty$. Let $\D$ be a strongly self-absorbing \ca, and suppose that $A$ is $\D$-absorbing.
Then $\alpha$ is conjugate to $\alpha\otimes\id_{\mathcal{D}}$.
\end{cor}

The result above is really stronger than the fact that $\mathcal{D}$-absorption is preserved by
taking crossed products by actions with finite Rokhlin dimension with commuting towers.

Here is our first dimension reduction result.

\begin{prop} \label{prop:OtwoRp}
Let $G$ be a finite group and let $\alpha\colon G\to\Aut(\mathcal{O}_2)$ be an action. If $\alpha$
has finite Rokhlin dimension with commuting towers, then it has the Rokhlin property.\end{prop}
\begin{proof}
By the equivalence between parts (1) and (3) of Theorem~4.2 in~\cite{Izu_finiteI_2004}, it is enough to construct a unital map
\[\Ot \to (\mathcal{O}_2^\alpha)_\infty \cap \Ot'.\]
This is an immediate consequence of \autoref{prop:EmbDFxPt}.
\end{proof}

The conclusion of \autoref{prop:OtwoRp} is false if one only
assumes $\dimRok(\alpha)<\I$. See, for example, Example~4.8 in~\cite{Gar_rokhlin_2017}.

For the next reduction result of the Rokhlin dimension, we will need to recall a definition.

\begin{df} (Definition~3.6 in~\cite{Izu_finiteI_2004}).
Let $G$ be a finite abelian group, let $B$ be a unital \ca, and let $\beta\colon G\to\Aut(B)$
be an action. We say that $\beta$ is \emph{strongly approximately inner}, if
there exist unitaries $v_g\in (B^\beta)_\I$, for $g\in G$,
satisfying $\beta_g(b)=v_gbv_g^*$ for all $b\in B$ and all $g\in G$.
\end{df}

In the theorem below, we do not know whether we obtain a similar conclusion
if we only assume that $\cdimRok(\alpha)<\I$, or if we replace $\Z_2$ with a general finite (abelian) group.

We point out that a similar phenomenon was observed in \cite{Gar_circle_2018}.

\begin{thm} \label{thm:UHFRp}
Let $A$ be a separable, unital \ca, and let $\alpha\colon \Z_2 \to\Aut(A)$ be an action. Assume that
$\cdimRok(\alpha)\leq 1$ and that $A$ absorbs the UHF-algebra $M_{2^\I}$ tensorially. Then $\alpha$ has the
Rokhlin property.\end{thm}
\begin{proof}
By Lemma~3.8 in~\cite{Izu_finiteI_2004}, it is enough to show that the dual action
$\widehat{\alpha}\colon\Z_2\to\Aut(A\rtimes_{\alpha}\Z_2)$ is approximately representable.
By a slight abuse of notation, we also denote by $\alpha$ and $\widehat{\alpha}$ the
order-two automorphisms that generate the given actions on $A$ and $A\rtimes_\alpha \Z_2$.

We claim that $\widehat{\alpha}$ is strongly approximately inner. Use \autoref{cor:ZnRdim1}
to find a unital equivariant homomorphism
\[\varphi\colon C(S^1)\to A_{\I}\cap A',\]
and denote by $w\in A_\I\cap A'$
the image of the canonical unitary in $C(S^1)$. Let $u\in A\rtimes_\alpha \Z_2$
denote the unitary implementing $\alpha$. Then
$uwu^*=\alpha_\I(w)=-w$ in $(A\rtimes_\alpha \Z_2)_\I$. In particular,
\[w^*uw=-u=\widehat{\alpha}(u).\]
Moreover, since $w$ commutes with the copy of $A$ in $A_\I$, it follows that
$w^*aw=a=\widehat{\alpha}(a)$ for all $a\in A$. Since $A$ and $u$ generate $B\rtimes_\alpha \Z_2$,
we deduce that $\Ad(w^*)$ coincides with $\widehat{\alpha}$, and thus $\widehat{\alpha}$ is strongly
approximately inner. This proves the claim.

By Lemma~3.10 in~\cite{Izu_finiteI_2004}, and since $A$ absorbs $M_{2^\I}$, it is enough to show that there is a unital map
\[M_2\to \left((A\rtimes_{\alpha}\Z_2)^{\widehat{{\alpha}}}\right)_\I\cap (A\rtimes_{\alpha}\Z_2)'.\]
This is again a consequence of \autoref{prop:EmbDFxPt}, so the proof is finished.
\end{proof}

It is not enough in \autoref{thm:UHFRp} to assume that $\dimRok(\alpha)\leq 1$. See, for example,
the comments in part~(a) of \autoref{rmk:nonCommTws}.

\end{document}